\documentclass[12pt]{amsart}
\usepackage[english]{babel}
\usepackage{graphicx}
\usepackage[T1]{fontenc}
\usepackage{color}
\usepackage{enumerate}

\newtheorem{Defi}{Definition}
\newtheorem{Thm}{Theorem}
\newtheorem{Prop}[Thm]{Proposition}

\newtheorem{Lem}{Lemma}
\newtheorem{Rem}{Remark}
\newtheorem{Exm}{Example}
\newtheorem{Ass}{Assumption}
\newtheorem{CExm}{Counterexample}
\newtheorem*{Prop*}{Proposition}
\newtheorem*{Cor*}{Corollary}
\newtheorem*{Thm*}{Theorem}

\def\eps{\varepsilon}
\def\R{{\mathbb R}}
\def\C{{\mathbb C}}
\def\N{{\mathbb N}}
\def\Z{{\mathbb Z}}
\def\T{{\mathbb T}}
\def\X{{\scriptstyle X}}
\def\XX{{\scriptscriptstyle X}}

\def\2scale{\stackrel{2-scale}{\relbar\joinrel\relbar\joinrel\relbar\joinrel\relbar\joinrel\rightharpoonup}}
\def\2scaleh{\stackrel{2-scale-{\bf h}^*}{\relbar\joinrel\relbar\joinrel\relbar\joinrel\relbar\joinrel\relbar\joinrel\relbar\joinrel\rightharpoonup}}
\def\sflow{\stackrel{\Sigma-\Phi_\tau}{\relbar\joinrel\relbar\joinrel\relbar\joinrel\relbar\joinrel\rightharpoonup}}

\begin{document}

\title[Convergence along mean flows]{Convergence along mean flows}

\bibliographystyle{plain}

\author[Thomas Holding]{Thomas Holding}
\address{T.H.: Mathematics Institute, University of Warwick, Coventry CV4 7AL, United Kingdom.}
\email{T.Holding@warwick.ac.uk}

\author[Harsha Hutridurga]{Harsha Hutridurga}
\address{H.H.: Department of Mathematics, Imperial College London, London, SW7 2AZ, United Kingdom.}
\email{h.hutridurga-ramaiah@imperial.ac.uk}

\author[Jeffrey Rauch]{Jeffrey Rauch}
\address{J.R.: Department of Mathematics, University of Michigan, Ann Arbor, MI 48109-1043, USA.}
\email{rauch@math.michigan.edu}

\begin{abstract}
We develop a technique of multiple scale asymptotic expansions along mean flows and a corresponding notion of weak multiple scale convergence.
These are applied to homogenize convection dominated parabolic equations with rapidly oscillating, locally periodic coefficients and $\mathcal{O}(\eps^{-1})$ mean convection term.
Crucial to our analysis is the introduction of a fast time variable, $\tau=t/\eps$, not apparent in the heterogeneous problem.
The effective diffusion coefficient is expressed in terms of the average of Eulerian cell solutions along the orbits of the mean flow in the fast time variable.
To make this notion rigorous, we use the theory of ergodic algebras with mean value.
\end{abstract}

\maketitle

\vspace{0.2 cm}


\setcounter{tocdepth}{1}
\tableofcontents

\section{Introduction}\label{sec:introduction}

This article studies the homogenization of parabolic equations of convection-diffusion type with locally periodic (in space), rapidly oscillating coefficients. This work addresses the self-similar diffusive scaling in these equations, i.e. for an unknown scalar density $u^\eps(t,x)$, we consider the Cauchy problem for a convection-diffusion equation with large convection term:
\begin{align}\label{eq:intro:convect-diffuse}
\frac{\partial u^\eps}{\partial t}
+ \frac{1}{\eps} {\bf b}\left(x,\frac{x}{\eps}\right) \cdot \nabla u^\eps 
- \nabla \cdot \left( {\bf D}\left(x,\frac{x}{\eps}\right) \nabla u^\eps\right) 
= 0
\qquad
\mbox{ for }(t,x)\in\, ]0,T[\times\R^d
\end{align}
with $0<\eps\ll1$ the scale of heterogeneity. This scaling corresponds to the long-term behaviour which can be described in terms of the effective or homogenized limit of the above scaled system.

It has remained a largely open problem to determine the homogenized limit of the scaled equation \eqref{eq:intro:convect-diffuse}. This present work gives a partial answer to this question in the sense that we homogenize the non-homogeneous equation with locally periodic coefficients under some structural assumptions on the flows associated with certain vector fields. This is achieved by the introduction of a new notion of weak convergence in $L^p$ spaces with $1<p< \infty$.

Under no diffuse scaling, i.e. with no large convection term, homogenization of such equations is classical. In such a scenario, we can either employ the method of asymptotic expansions (see for instance, the monographs \cite{bensoussan2011asymptotic, sanchez1980non}) which provides us with the approximation
\begin{align}\label{eq:intro:2scale-exp-classic}
u^\eps(t,x) \approx u_0(t,x) + \eps u_1\left(t, x, \frac{x}{\eps}\right) + \eps^2 u_1\left(t, x, \frac{x}{\eps}\right) + \cdots
\end{align}
or employ a weak convergence approach of the two-scale convergence method introduced by G. Nguetseng in \cite{nguetseng1989general} and further developed by G. Allaire in \cite{allaire1992homogenization}. The cornerstone result of the two-scale convergence method is that, up to extraction of a subsequence, any uniformly (w.r.t. $\eps$) bounded sequence $\{u^\eps\}$ in some $L^p$ space with $1< p < \infty$ satisfies
\begin{align*}
\lim_{\eps\to0}
\iint\limits_{(0,T)\times\R^d}
u^\eps(t,x) \psi\left(t, x, \frac{x}{\eps} \right)
\, {\rm d}x\, {\rm d}t
=
\iiint\limits_{(0,T)\times\R^d\times\T^d}
u_0(t,x,y)
\psi(t,x,y)
\, {\rm d}y\, {\rm d}x\, {\rm d}t
\end{align*}
for some $u_0\in L^p((0,T)\times\R^d\times\T^d)$ called the weak two-scale limit and for any smooth $\psi(t,x,y)$ which is periodic in the $y$ variable.

Any weak convergence approach to homogenize a partial differential equation would involve passing to the limit (as the heterogeneity length scale tends to zero) in the weak formulation associated to the partial differential equation. This would require passing to the limit in products of weakly converging sequences. The main feature of the two-scale convergence method is that the particular choice of test functions allows us to pass to the limit in such products. If $u^\eps(t,x)$ weakly two-scale converges to $u_0(t,x,y)\in L^p((0,T)\times\R^d\times\T^d)$ and if the coefficient function $a(t,x,y),$ which is periodic in the $y$ variable, is admissible (roughly speaking, continuous or approximable by continuous functions in a certain sense -- see Definition \ref{defn:abs:admissible-2-scale} for precise statement), then the product has the convergence
\begin{align*}
a
\left(
t, x, \frac{x}{\eps}
\right)
u^\eps(t,x)
\rightharpoonup
\int\limits_{\T^d}
a(t,x,y) u_0(t,x,y)
\, {\rm d}y
\qquad
\mbox{ as }\eps\to0,
\end{align*}
in the sense of distributions.

In recent years, there have been numerous publications in the mathematics literature dedicated to generalize the notion of two-scale convergence (originally developed to handle periodic structures) to address the homogenization of partial differential equations with coefficients that belong to some ergodic algebras. Typically, all these works are about the study of the limiting behaviour (as $\eps\to0$) of the integral
\begin{align*}
\int\limits_{\R^d}
v^\eps(x) \psi\left(x,\frac{x}{\eps}\right)
\, {\rm d}x
\end{align*}
when $\{v^\eps\}$ is a uniformly bounded sequence in some Lebesgue space $L^p$ with $1 < p < \infty$ and $\psi(x,y)$ belongs to certain ergodic algebra in the $y$ variable. The notion of \emph{algebras with mean value} play a crucial role in these theories. This notion goes back to the work of Zhikov and Krivenko \cite{zhikov1983averaging} in the early 1980's (also see the book of Jikov, Kozlov and Oleinik \cite{jikov1994homogenization} for a pedagogical exposition). We cite some of the references in this context which we have consulted in developing our theory: \cite{casado2002two, nguetseng2003homogenization, nguetseng2004homogenization, nguetseng2011sigma}.

With regard to the homogenization of the scaled equation \eqref{eq:intro:convect-diffuse}, the known results are when the rapidly oscillating coefficients are purely periodic, i.e. of the type ${\bf b}\left(\frac{x}{\eps}\right)$, ${\bf D}\left(\frac{x}{\eps}\right)$. The case when the fluid field ${\bf b}(\cdot)$ is of zero mean was treated in \cite{bensoussan2011asymptotic, mclaughlin1985convection} using two-scale asymptotic expansions of the form \eqref{eq:intro:2scale-exp-classic}. They do not prove convergence. Over two decades ago, to address the case of fluid field ${\bf b}(\cdot)$ with non-zero mean, G. Papanicolaou suggested in \cite{papanicolaou1995diffusion} a modified two-scale asymptotic expansion where the coefficients in the expansion are taken along rapidly moving coordinates:
\begin{align}\label{eq:intro:2scale-drift-exp-classic}
u^\eps(t,x)
\approx
u_0\left( t, x-\frac{{\bf b}^* t}{\eps} \right)
+ \eps u_1\left( t, x-\frac{{\bf b}^* t}{\eps}, \frac{x}{\eps} \right)
+ \eps^2 u_2\left( t, x-\frac{{\bf b}^* t}{\eps}, \frac{x}{\eps} \right)
+ \cdots
\end{align}
The constant ${\bf b}^*\in\R^d$ is the mean field associated with ${\bf b}(\cdot)$. Note that the case ${\bf b}^* = 0$ coincides with the classical expansion \eqref{eq:intro:2scale-exp-classic}. We cite the works in \cite{allaire2007homogenization, allaire2010two, allaire2012homogenization} where the above expansion with drift is employed in homogenizing reactive transport models in periodic porous media.

Analogous to the two-scale convergence method, Maru{\v{s}}i{\'c}-Paloka and Piatnitski introduced a notion of weak convergence in \cite{maruvsic2005homogenization} called the two-scale convergence with drift (see \cite{allaire2008periodic} for a pedagogical exposition of this method) characterizing the limit
\begin{align*}
\lim_{\eps\to0}
\iint\limits_{(0,T)\times\R^d}
u^\eps(t,x) \psi \left( t, x-\frac{{\bf b}^* t}{\eps}, \frac{x}{\eps} \right)
\, {\rm d}x\, {\rm d}t
\end{align*}
where $\psi(t,x,y)$ is periodic in the $y$ variable and as usual the family $\{u^\eps\}$ is uniformly bounded (w.r.t. $\eps$) in some $L^p$ space with $1<p< \infty$.

Neither the modified two-scale expansion \eqref{eq:intro:2scale-drift-exp-classic} nor the notion of two-scale convergence with drift seem capable of treating equation \eqref{eq:intro:convect-diffuse} with locally periodic, rapidly oscillating coefficients, i.e. when ${\bf b}$ depends upon both $x$ and $y$. We cite the work of P-E. Jabin and A. Tzavaras \cite{jabin2009kinetic} which treats the homogenization of \eqref{eq:intro:convect-diffuse} with locally periodic fluid field ${\bf b}(x,y)$ and diffusion coefficient being unity. They treat a special case when the mean field $\bar{\bf b}(x)$ of the locally periodic fluid field ${\bf b}(x,y)$ vanishes, i.e. $\bar{\bf b}(x)\equiv0$ for all $x$. They introduce a notion of \emph{kinetic decomposition} to address this problem. As far as the authors are aware, the techniques of \cite{jabin2009kinetic} are not capable of addressing the case of non-zero mean field.

In this work, we introduce a new multiple scale expansion
\begin{align}\label{eq:intro:2scale-flow-exp-classic}
u^\eps(t,x)
\approx
u_0\left( t, \Phi_{-t/\eps}(x) \right)
+ \eps u_1\left( t, \frac{t}{\eps}, \Phi_{-t/\eps}(x), \frac{x}{\eps} \right)
+ \eps^2 u_2\left( t, \frac{t}{\eps}, \Phi_{-t/\eps}(x), \frac{x}{\eps} \right)
+ \cdots
\end{align}
which we call \emph{multiple scale expansion along mean flows}. The coefficient functions $u_i$ in \eqref{eq:intro:2scale-flow-exp-classic} are taken on rapidly moving coordinates $\Phi_{-t/\eps}(x)$ which is the flow associated with the mean field $\bar{\bf b}(x)$ of the locally periodic fluid field ${\bf b}(x,y)$. A novelty of our method is the introduction of the fast time variable $\tau:=t/\eps$. The main assumption in this work is on the Jacobian matrix $J(\tau,x)$ associated with the flow $\Phi_\tau(x)$.
\begin{align*}
\mbox{\centering {\bf Assumption:} \emph{There is a uniform constant $C$ such that $|J(\tau,x)|\le C$ for all $(\tau,x)\in\R\times\R^d$.}}
\end{align*}

The above assumption is trivially satisfied in all the previously known works on the homogenization of \eqref{eq:intro:convect-diffuse} because the Jacobian matrix associated with the flows in all these works is the identity.

Under this assumption, we derive a homogenized diffusion equation for the zeroth order approximation $u_0$ in \eqref{eq:intro:2scale-flow-exp-classic} with an explicit expression for the effective diffusion coefficient. The diffusion equation for $u_0$ is in Lagrangian coordinates because of the structure of the asymptotic expansion. The effect of Lagrangian stretching on the gradient of the scalar density $u^\eps$, i.e. creating large gradients has been widely studied in the literature in the case of non-oscillating coefficients (see for e.g. \cite{haynes2005controls, berestycki2005elliptic, constantin2008diffusion, hamel2010extinction}). If the above assumption is not made on the Jacobian matrix, we cannot expect a nontrivial limit as the large gradients can drive the solution to zero quickly. The mathematical model considered in this article is one of the simplified models for turbulent diffusion studied widely in the physics and mathematics literature -- for further details consult \cite[Section~2]{majda1999simplified}.

Taking inspiration from the work of Maru{\v{s}}i{\'c}-Paloka and Piatnitski \cite{maruvsic2005homogenization}, we devise a weak convergence approach which involves the characterization of the limit
\begin{align*}
\lim_{\eps\to0}
\iint\limits_{(0,T)\times\R^d}
u^\eps(t,x) \psi \left( t, \Phi_{-t/\eps}(x), \frac{t}{\eps}, \frac{x}{\eps} \right)
\, {\rm d}x\, {\rm d}t
\end{align*}
with a uniformly bounded family $\{u^\eps\}$ in some $L^p$ space ($1<p<+\infty$) and the test function $\psi(t,x,\tau,y)$ being periodic in the $y$ variable and belongs to an \emph{ergodic algebra with mean value} in the $\tau$ variable. We call this notion of convergence \emph{weak $\Sigma$-convergence along flows}.

To use this new notion of convergence, our strategy is to use test functions of the form $\psi \left( t, \Phi_{-t/\eps}(x), \frac{t}{\eps}, \frac{x}{\eps} \right)$ in the weak formulation of the scaled problem \eqref{eq:intro:convect-diffuse}. This weak formulation would have terms involving the Jacobian matrix associated with the flow $\Phi_{-t/\eps}(x)$. Note that the Jacobian matrix depends on the fast time variable, i.e. appears as $J\left(\frac{t}{\eps}, x\right)$, because of the chosen time scale in the flow. Our strategy, hence, is to consider test functions that belong to some ergodic algebra in the fast time variable $\tau$. 

Inspired by the notion of \emph{admissible} functions introduced by G. Allaire in \cite{allaire1992homogenization} and further clarified by M. Radu in her PhD thesis \cite{radu1992homogenization}, we introduce a notion of \emph{admissible} functions adapted to the weak $\Sigma$-convergence along flows (see Definition \ref{defn:abs:admissible-test-fn}). Another novelty of our approach is to consider the \emph{flow-representation} of functions (see Subsection \ref{ssec:tilde-representation} for precise definition). Our main result is to show that if the flow-representations of the fluid field ${\bf b}(x,y)$, the diffusion matrix ${\bf D}(x,y)$, the Jacobian matrix $J(\tau,x)$ are admissible, then we can derive the effective limit diffusion equation. These assumptions on the coefficients and the Jacobian matrix are very essential for our analysis as is evident from the counterexamples that are constructed in Section \ref{sec:Exm} of this paper.

The main homogenization result of this article is Theorem \ref{thm:hom:hom}. We summarize this result below (consult Theorem \ref{thm:hom:hom} in Section \ref{sec:homog_result} for precise statement).
\begin{Thm*}
Let $\Phi_\tau(x)$ be the flow associated with the mean field $\bar{\bf b}(x)$. Suppose the associated Jacobian matrix $J(\tau,x)$ is a uniformly bounded function of $\tau$ and $x$ variables. Let the flow-representations of the coefficients in \eqref{eq:intro:convect-diffuse} and that of the Jacobian matrix belong to certain ergodic algebra with mean value. Then the solution family $u^\eps(t,x)$ weakly $\Sigma$-converges along the flow $\Phi_\tau$ to the unique solution of the homogenized equation
\begin{align*}
\frac{\partial u_0}{\partial t}
-
\nabla_\XX \cdot \Big(
\mathfrak{D}(\X)\nabla_\XX u_0
\Big)
= 0
\end{align*}
where the effective diffusion matrix $\mathfrak{D}(\X)$ is given in terms of certain averages of solutions to cell problems and the averages are taken along the orbits of the mean flow.
\end{Thm*}

\pagebreak[3]
{\centering {\bf Outline of the paper:}}
\begin{itemize}
\item In Section \ref{sec:heuristics}, we introduce the method of \emph{multiple scale asymptotic expansions along mean flows} to derive the effective equation for the scaled equation \eqref{eq:heu:CD}-\eqref{eq:heu:IV}. This result is recorded as Proposition \ref{prop:heu:formal-result} which gives an explicit expression for the effective diffusion matrix.
\item Section \ref{sec:abs} introduces the new notion of weak multiple scale convergence. In Subsections \ref{ssec:abs:algebras-w.m.v.} through \ref{ssec:abs:technicalities}, we recall enough of the theory of algebras with mean value. The notion of $\Sigma$-convergence along flows is introduced in Subsection \ref{ssec:abs:2-scale-convergence-along-flows}. The main compactness result with regard to this new notion of convergence is given by Theorem \ref{Thm:abs:compactness}. In Subsection \ref{ssec:abs:additional-bounds}, we obtain compactness results on the gradient sequences (in the sense of corrector results in homogenization).
\item Section \ref{sec:homog_result} deals with the homogenization result. The main result of this section is Theorem \ref{thm:hom:hom}. The main assumptions made on the coefficients and the Jacobian matrix are explained in Subsection \ref{ssec:hom:assumptions}.
\item Section \ref{sec:Exm} provides some discussion on the assumptions made on the coefficients and the Jacobian matrix. In particular, we give some examples of fluid fields with bounded Jacobian matrices and show that unbounded growth in the Jacobian matrix can lead to trivial and singular behaviour of the limit $u_0$. We also provide an explicit example of an equation, where the assumptions on the flow-representation of the coefficients do not hold, leading to two different homogenized equations in the $\eps\to0$ limit.
\item In Section \ref{sec:explicit}, we perform asymptotic analysis on some explicit convection-diffusion models which highlights the effectiveness of this new approach in addressing the large convection terms. Finally, in Section \ref{sec:conclusions}, we give some concluding remarks.
\end{itemize}

\noindent{\bf Acknowledgements.} 
This work was initiated during J.R.'s visit to the University of Cambridge. It was supported by the ERC grant \textsc{matkit}.
The authors would like to thank Gr\'egoire Allaire for his fruitful suggestions during the preparation of this article.
The authors would also like to thank Mariapia Polambaro for helpful discussions regarding Euclidean motions and for bringing to our attention the work of P-E. Jabin and A. Tzavaras \cite{jabin2009kinetic}. H.H. acknowledges the support of the ERC grant \textsc{matkit}. 
T.H. is supported by the UK Engineering and Physical Sciences Research Council (EPSRC) grant EP/H023348/1 for the University of Cambridge Centre for Doctoral Training, the Cambridge Centre for Analysis.
We would like to thank the referee for helpful comments.

\section{Asymptotic expansion along flows}\label{sec:heuristics}

\subsection{Mathematical model}

Let ${\bf b}(x,y):\R^d\times\T^d\to\R^d$ be a prescribed time-independent fluid field which is incompressible in both the $x$ and $y$ variables, i.e.
\begin{align}\label{eq:heu:null-divergence-fluid-field}
\nabla_x\cdot {\bf b}(x,y) = \nabla_y\cdot {\bf b}(x,y) = 0
\qquad
\mbox{ for a.e. }(x,y)\in\R^d\times\T^d.
\end{align}
Define the associated mean field as
\begin{align}\label{eq:heu:mean-field}
\bar{\bf b}(x) := \int\limits_{\T^d} {\bf b}(x,y)\, {\rm d}y.
\end{align}
{\centering{\textsc{Notation:} For any matrix $\mathbf{B}$, its transpose is denoted by ${}^\top\mathbf{B}$.}}

Let ${\bf D}(x,y)\in L^\infty(\R^d\times\T^d;\R^{d\times d})$ be a given time-independent symmetric (i.e. ${\bf D} = {}^\top{\bf D}$) matrix-valued diffusion coefficient which is assumed to be uniformly coercive, i.e.
\begin{align}\label{eq:heu:coercive-diffusion-matrix}
\exists\, \lambda, \Lambda >0
\, \, \mbox{ s.t. }
\, \, 
\lambda|\xi|^2 
\le 
{}^\top\!\!\, \xi\, {\bf D}(x,y)\, \xi
\le \Lambda|\xi|^2
\quad
\forall\xi\in\R^d
\, \, \mbox{ and for a.e. }
(x,y)\in\R^d\times\T^d.
\end{align}
Let $0<\eps\ll1$ be the scale of heterogeneity. Let us consider a scaled Cauchy problem with rapidly oscillating coefficients for an unknown scalar density $u^\eps(t,x):[0,T[\times\R^d\to[0,\infty)$.
\begin{subequations}
\begin{align}
& \frac{\partial u^\eps}{\partial t} 
+ \frac1\eps {\bf b}\left(x,\frac{x}{\eps}\right)\cdot\nabla u^\eps
- \nabla\cdot\left({\bf D}\left(x,\frac{x}{\eps}\right) \nabla u^\eps\right) 
= 0 & \mbox{ for }(t,x)\in\, ]0,T[\times\R^d,\label{eq:heu:CD}\\[0.2 cm]
& u^\eps(0,x) = u^{in}(x) & \mbox{ for }x\in\R^d.\label{eq:heu:IV}
\end{align}
\end{subequations}
The \emph{two-scale expansions with drift} method (see \cite{maruvsic2005homogenization, donato2005averaging, allaire2007homogenization, allaire2008periodic, allaire2010two, allaire2012homogenization}) employs the asymptotic expansion for the unknown density:
\begin{align}\label{eq:heu:drift-expansion}
u^\eps(t,x) =
\sum_{i=0}^\infty
\eps^i u_i\left(t,x-\frac{{\bf b}^*t}{\eps}, \frac{x}{\eps}\right),
\end{align}
where the drift velocity ${\bf b}^*\in\R^d$ is a constant and the coefficient functions $u_i(t,x,y)$ are assumed to be periodic in the $y$ variable. Remark that the coefficient functions $u_i$ in \eqref{eq:heu:drift-expansion} are written in moving coordinates. To be precise, consider the ordinary differential equation
\begin{align}\label{eq:heu:ode-constant-drift}
\dot{{\scriptstyle X}} = {\bf b}^*;\qquad
{\scriptstyle X}(0) = x.
\end{align}
Denote by $\Phi_\tau(x)$ the flow associated with \eqref{eq:heu:ode-constant-drift}. The flow evaluated at the time instant $(-t/\eps)$ is nothing but the moving coordinates taken in the asymptotic expansion \eqref{eq:heu:drift-expansion}, i.e.
\begin{align*}
\Phi_{-t/\eps}(x) = x-\frac{{\bf b}^*t}{\eps}.
\end{align*}
The idea of considering the asymptotic expansion along moving coordinates was mentioned by G. Papanicolaou in a survey paper \cite{papanicolaou1995diffusion}. It should be noted that the \emph{two-scale expansions with drift} method can handle the homogenization of convection-diffusion equation \eqref{eq:heu:CD} only when the fluid field is purely periodic, i.e. ${\bf b}(x,y)\equiv{\bf b}(y)$. In that case, the constant drift velocity is  taken to be
\begin{align*}
{\bf b}^* = \int\limits_{\T^d} {\bf b}(y)\, {\rm d}y.
\end{align*}
Taking cues from the constant drift scenario, consider the autonomous system
\begin{align}\label{eq:heu:ode-mean-field}
\dot{{\scriptstyle X}} = \bar{{\bf b}}({\scriptstyle X});\qquad
{\scriptstyle X}(0) = x.
\end{align}
Again, denoting the flow associated with \eqref{eq:heu:ode-mean-field} by $\Phi_\tau(x)$, we postulate the asymptotic expansion in the spirit of \eqref{eq:heu:drift-expansion}:
\begin{align}\label{eq:heu:mean-field-expansion}
u^\eps(t,x) =
\sum_{i=0}^\infty
\eps^i u_i\left( t, \frac{t}{\eps}, \Phi_{-t/\eps}(x), \frac{x}{\eps}\right).
\end{align}
Note that the coefficient functions $u_i(t, \tau, {\scriptstyle X}, y)$ in \eqref{eq:heu:mean-field-expansion} depend on an additional variable $\tau$ which we shall call the \emph{fast time variable}. We assume that the coefficient functions $u_i(t, \tau, {\scriptstyle X}, y)$ are periodic in the $y$ variable. The structural assumption on the coefficients $u_i(t,\tau,\X,y)$ with regard to the $\tau$ variable is a little bit subtle. We shall assume that the coefficient functions, as a function of $\tau$, belong to an \emph{ergodic algebra with mean value}. This shall guarantee the existence of certain weak* limits. This will be made more rigorous in a later stage of the article (see Section \ref{sec:abs}). The authors of \cite{bourgeat2003averaging} also introduced a fast time variable in their asymptotic expansion. However, they do not consider the expansion along moving coordinates as is the case in \eqref{eq:heu:mean-field-expansion}. Also, the authors of \cite{bourgeat2003averaging} assume that the coefficient functions decay exponentially in the fast time variable.

\subsection{Flow representation}\label{ssec:tilde-representation}

We introduce a notion of \emph{flow representation} that is very central to our analysis. The choice of considering rapidly moving coordinates in the expansion \eqref{eq:heu:mean-field-expansion} is equivalent to expressing the convection-diffusion equation \eqref{eq:heu:CD} in Lagrangian coordinates. This necessitates the consideration of the coefficient functions in the convection-diffusion equation in Lagrangian coordinates. Essentially, the flow representation takes into account the underlying flow structure associated with the mean field $\bar{\bf b}(x)$.  

To be precise, consider a function $f:\R^d\to\R$ and a flow $\Phi_\tau(x):\R\times\R^d\to\R^d$. The flow representation of $f$ is given by the function $\widetilde{f}:\R\times\R^d\to\R$ defined as
\begin{align*}
\widetilde{f}(\tau,x):=f(\Phi_{\tau}(x)).
\end{align*}
We shall use the following convention for their flow representations when we encounter locally periodic functions, i.e. functions of the form $f(x,y):\R^d\times\T^d\to\R$
\begin{align*}
\widetilde{f}(\tau,x,y):=f(\Phi_{\tau}(x),y).
\end{align*}
The following observations are obvious for the flow representations:
\begin{align*}
& \widetilde{f}(0,x) = f(x);
\\
& \widetilde{f}(\tau, \Phi_{\tau'}(x)) = f(\Phi_{\tau+\tau'}(x))
\qquad
\mbox{ for any }\tau,\tau'\in\R;
\\
& \widetilde{f}(\tau, \X) = f(x) 
\qquad
\mbox{ with the convention }\X:=\Phi_{-\tau}(x).
\end{align*}
When we encounter vector-valued functions, it should be noted that their flow representations are taken component-wise. It should also be noted that the flow $\Phi_\tau$ used in giving the flow representation of a function can be any one-parameter group of transformation and need not be associated with any vector field.

\begin{Rem}
The one parameter group $U^\tau$ defined by $(U^\tau f)(x) := \widetilde{f}(\tau, x)$ is generated (at least formally, i.e. without regard to functional spaces) by the skew-symmetric operator $\bar{\bf b}(x)\cdot \nabla$.
\end{Rem}

\subsection{Flows associated with vector fields}\label{ssec:flows}
Let $J(\tau,x)$ denote the Jacobian matrix of the flow $\Phi_\tau$ generated by \eqref{eq:heu:ode-mean-field}, i.e.
\begin{equation}\label{eq:heu:Jacobian-matrix}
J(-\tau,x)
=\begin{bmatrix}
\frac{\partial\Phi^1_\tau}{\partial x_1}&\dotsb&\frac{\partial\Phi^1_\tau}{\partial x_d}\\
\vdots&&\vdots\\
\frac{\partial\Phi^d_\tau}{\partial x_1}&\dotsb&\frac{\partial\Phi^d_\tau}{\partial x_d}
\end{bmatrix}
=\left(\frac{\partial\Phi^i_\tau}{\partial x_j}\right)_{i,j=1}^d.
\end{equation}
We have used the convention that $J(\tau,x)$ is the Jacobian of the \emph{backwards} flow $\Phi_{-\tau}(x)$ to ease notation as it is this that appears throughout. The flow representation of the Jacobian matrix function $J:\R\times\R^d\to\R^{d\times d}$ is defined by
\begin{align*}
\widetilde{J}(\tau,\Phi_{-\tau}(x)) 
= 
\widetilde{J}(\tau, \X) 
= 
J(\tau, x).
\end{align*}
In order to ensure the validity of the proposed asymptotic expansion \eqref{eq:heu:mean-field-expansion} we make the assumption of uniform boundedness on the Jacobian matrix:
\begin{Ass}\label{heu:ass:bound-on-J}
There is a constant $C$ such that $|J(\tau,x)|\le C$ for all $\tau\in\mathbb{R}$ and $x\in\mathbb{R}^d$.
\end{Ass}
To finish this subsection we record some facts regarding the change of variables. Although these are well known, we provide a proof in the Appendix for the convenience of the reader.

\begin{Lem}\label{lem:heu:basic-facts}
Let $\bar{\bf b}\in C^1(\R^d)$, then the following hold:
\begin{enumerate}[(i)]
\item $\nabla_{\XX}\cdot {}^\top\!\!\, \widetilde{J}(\tau, \X)=0$ in the sense of distributions.
\item $\nabla_{\XX}\cdot\left(\widetilde{J}(\tau,\X)\widetilde{f}(\tau,\X,y)\right)=0$ in the sense of distributions, for any vector field $f(x,y)$ which is of null-divergence in the $x$ variable, i.e. $\nabla_x\cdot f(x,y)=0$.
\item For any $\phi,\varphi\in C^\infty_c(\R^d;\R)$ we have the integration by parts formula:
\begin{align*}
\int\limits_{\R^d}
\phi(\X)
\left(
{}^\top\!\!\, \widetilde{J}(\tau,\X)\nabla_\XX \varphi(\X)
\right)
\, {\rm d}\X
= 
- \int\limits_{\R^d}
\varphi(\X)
\left(
{}^\top\!\!\, \widetilde{J}(\tau,\X)\nabla_\XX \phi(\X)
\right)
\, {\rm d}\X.
\end{align*}
\item For any $\tau\in\mathbb{R}$ and $x\in \R^d$ it holds that
\begin{align}\label{eq:heu:lem:bar-b-Phi=J-bar-b-x}
\bar{{\bf b}}(\X)=\bar{{\bf b}}\left(\Phi_{-\tau}(x)\right)
= J(\tau,x) \bar{{\bf b}}(x)=\widetilde{J}(\tau,\X)\widetilde{\bar{{\bf b}}}(\tau,\X).
\end{align}
\end{enumerate}
\end{Lem}

\subsection{Multiple scale expansion along mean flows}\label{ssec:asymptotic-expansion}

We present a strategy to formally arrive at an effective equation for \eqref{eq:heu:CD}-\eqref{eq:heu:IV} by using the asymptotic expansion \eqref{eq:heu:mean-field-expansion} postulated earlier. In the case of constant drift \eqref{eq:heu:drift-expansion}, we have the following chain rules for differentiating the coefficient functions in the space and times variables:
\begin{align*}
\nabla_x \left( u_i \left( t, x - \frac{{\bf b}^*t}{\eps}, \frac{x}{\eps} \right) \right)
&= \nabla_x u_i \left( t, x - \frac{{\bf b}^*t}{\eps}, \frac{x}{\eps} \right)
+ \frac{1}{\eps} \nabla_y u_i \left( t, x - \frac{{\bf b}^*t}{\eps}, \frac{x}{\eps} \right),
\\[0.2 cm]
\frac{\partial}{\partial t} \left(u_i \left( t, x - \frac{{\bf b}^*t}{\eps}, \frac{x}{\eps} \right) \right)
&=\frac{\partial u_i}{\partial t} \left( t, x - \frac{{\bf b}^*t}{\eps}, \frac{x}{\eps} \right)-\frac1\eps{\bf b}^*\cdot \nabla_x u_i \left( t, x - \frac{{\bf b}^*t}{\eps}, \frac{x}{\eps} \right).
\end{align*}
Remark that the above simple expression for the derivative is because of the Jacobian matrix being the identity for the change of variables:
\begin{align*}
x\mapsto x-\frac{{\bf b}^*t}{\eps}.
\end{align*}
However, for the change of variables
\begin{align*}
x\mapsto \Phi_{-t/\eps}(x)
\end{align*}
where the flow $\Phi_\tau$ is associated with \eqref{eq:heu:ode-mean-field}, the associated chain rules for differentiating the coefficient functions in the asymptotic expansion \eqref{eq:heu:mean-field-expansion} with respect to the space and time variables shall be
\begin{equation*}
\begin{array}{ll}
\displaystyle
\nabla_x \left( u_i \left( t, \frac{t}{\eps}, \Phi_{-t/\eps}(x), \frac{x}{\eps} \right) \right)
= & \displaystyle
{}^\top \!\!\, \widetilde{J}\left(\frac{t}{\eps}, \Phi_{-t/\eps}(x) \right) \nabla_{\scriptscriptstyle X} u_i \left( t, \frac{t}{\eps}, \Phi_{-t/\eps}(x), \frac{x}{\eps} \right)\\[0.5 cm]
\displaystyle
& \displaystyle
+ \frac{1}{\eps} \nabla_y u_i \left( t, \frac{t}{\eps}, \Phi_{-t/\eps}(x), \frac{x}{\eps} \right),
\end{array}
\end{equation*}
\begin{equation*}
\begin{array}{ll}
\displaystyle
\frac{\partial}{\partial t} \left( u_i \left( t, \frac{t}{\eps}, \Phi_{-t/\eps}(x), \frac{x}{\eps} \right) \right)
= & \displaystyle
\frac{\partial u_i}{\partial t} \left( t, \frac{t}{\eps}, \Phi_{-t/\eps}(x), \frac{x}{\eps} \right)
+ \displaystyle
\frac{1}{\eps}\frac{\partial u_i}{\partial \tau} \left( t, \frac{t}{\eps}, \Phi_{-t/\eps}(x), \frac{x}{\eps} \right)\\[0.5 cm]
& \displaystyle
- \frac{1}{\eps}\bar{\bf b}\left(\Phi_{-t/\eps}(x)\right) \cdot \nabla_{\scriptscriptstyle X} u_i \left( t, \frac{t}{\eps}, \Phi_{-t/\eps}(x), \frac{x}{\eps} \right).
\end{array}
\end{equation*}
The strategy of any asymptotic expansion method in homogenization is to substitute the postulated expansion into the model equation and solve a cascade of equations for obtaining the coefficient functions in the asymptotic expansion. All the equations in this cascade obtained by this approach have a similar structure. Next, we state a standard Fredholm type result which guarantees the solvability of such equations provided the source terms satisfy a compatibility condition.
\begin{Lem}\label{lem:heu:fredholm}
Let $x\in\R^d$ be a fixed parameter. Suppose $g(x,\cdot)\in L^2(\T^d)$ be the source term in the boundary value problem:
\begin{equation}\label{eq:heu:fredholm}
{\bf b}(x,y)\cdot\nabla_y f - \nabla_y\cdot\left( {\bf D}(x,y)\nabla_y f\right) = g(x,y) \qquad \mbox{ in }\T^d.
\end{equation}
Then there exists a unique solution $f\in H^1(\T^d)/\R:=\{f\in H^1(\T^d):\int_{\T^d} f \,{\rm d} y=0\}$ to (\ref{eq:heu:fredholm}) if and only if the source term satisfies
\begin{equation}\label{eq:heu:fredholm-compatible}
\int\limits_{\T^d} g(x,y)\, {\rm d}y = 0.
\end{equation}
\end{Lem}
Next, we record a formal result on the homogenized equation for the scaled equation with rapidly oscillating coefficients \eqref{eq:heu:CD}-\eqref{eq:heu:IV}.

\begin{Prop}[formal result]\label{prop:heu:formal-result}
Under Assumption \ref{heu:ass:bound-on-J} and the assumption \eqref{eq:heu:mean-field-expansion}, the solution to the Cauchy problem \eqref{eq:heu:CD}-\eqref{eq:heu:IV} formally satisfies
\begin{align}\label{eq:heu:u-eps-approx}
u^\eps(t,x) 
\approx
u_0 \left( t, \Phi_{-t/\eps}(x) \right)
+ \eps u_1 \left( t, \frac{t}{\eps}, \Phi_{-t/\eps}(x), \frac{x}{\eps} \right)
\end{align}
where the first order corrector $u_1$ can be written as
\begin{align}\label{eq:heu:form-of-u_1}
u_1(t,\X,\tau,y) 
= 
\widetilde{\omega}(\tau, \X, y)\cdot  {}^\top \!\!\, \widetilde{J}(\tau,\X) \nabla_{\scriptscriptstyle X} u_0(t,\X)
\end{align}
and the zeroth order term $u_0$ satisfies the homogenized diffusion equation
\begin{subequations}
\begin{align}
& \frac{\partial u_0}{\partial t} = \nabla_{\scriptscriptstyle X} \cdot \Big( \mathfrak{D}(\X) \nabla_{\scriptscriptstyle X} u_0\Big)
\qquad
\mbox{ for }(t,\X)\in \, ]0,T[\times\R^d,\label{eq:heu:homogenized-equation}
\\
& u_0(0,\X) = u^{in}(\X)
\qquad \qquad \qquad
\mbox{ for }\X\in\R^d.\label{eq:heu:homogenized-initial}
\end{align}
\end{subequations}
The effective diffusion coefficient is given by
\begin{equation}\label{eq:heu:effective-diffusion}
\begin{aligned}
\mathfrak{D}(\X)
=
\lim_{\ell\to\infty}
\frac{1}{2\ell}
\int\limits_{-\ell}^{+\ell}
\widetilde{J}(\tau,\X) 
\mathfrak{B}(\tau,\X)
{}^\top \!\! \, \widetilde{J}(\tau,\X)
\, {\rm d}\tau,
\end{aligned}
\end{equation}
where the elements of the matrix $\mathfrak{B}$ are given by
\begin{equation}\label{eq:heu:effective-diffusion-mathfrak-B}
\begin{aligned}
\mathfrak{B}_{ij}(\tau,\X)
& = 
\int\limits_{\T^d}
\widetilde{\bf D}(\tau,\X,y)
\Big(
\nabla_y\widetilde{\omega}_j(\tau,\X,y)
+ 
{\bf e}_j
\Big)
\cdot
\Big(
\nabla_y\widetilde{\omega}_i(\tau,\X,y)
+
{\bf e}_i
\Big)
\, {\rm d}y
\\
& +
\int\limits_{\T^d}
\left( 
\widetilde{{\bf b}}(\tau,\X,y)
\cdot 
\nabla_y \widetilde{\omega}_i(\tau,\X,y) 
\right)
\widetilde{\omega}_j(\tau,\X,y)
\, {\rm d}y
\\
& +
\int\limits_{\T^d}
\widetilde{\bf D}(\tau,\X,y)
\nabla_y\widetilde{\omega}_j(\tau,\X,y)
\cdot
{\bf e}_i
\, {\rm d}y
-
\int\limits_{\T^d}
\widetilde{\bf D}(\tau,\X,y)
\nabla_y\widetilde{\omega}_i(\tau,\X,y)
\cdot
{\bf e}_j
\, {\rm d}y
\end{aligned}
\end{equation}
for $i,j\in\{1,\cdots,d\}$. Furthermore, the components of $\omega$ satisfy the cell problem
\begin{align}\label{eq:heu:cell-problem}
{\bf b}(x,y)\cdot \Big( \nabla_y \omega_i + {\bf e}_i \Big)
- \nabla_y \cdot \Big( {\bf D}(x,y) \Big( \nabla_y \omega_i + {\bf e}_i \Big) \Big)
= \bar{\bf b}(x)\cdot {\bf e}_i
\qquad
\mbox{ in }\T^d,
\end{align}
for each $i\in\{1,\cdots,d\}$ and with the standard canonical basis $({\bf e}_i)_{1\le i\le d}$ in $\R^d$.
\end{Prop}

\begin{Rem}\label{rem:heu:finite-expansion}
Even though we postulate an infinite sum in the asymptotic expansion \eqref{eq:heu:mean-field-expansion}, we compute only the zeroth and first order coefficients as in \eqref{eq:heu:u-eps-approx}. Our goal is to obtain an evolution equation for the zeroth order approximation, i.e. the homogenized equation \eqref{eq:heu:homogenized-equation}. For this purpose, the first order approximation \eqref{eq:heu:u-eps-approx} suffices. Continuing the expansion for higher order coefficients in the asymptotic expansion \eqref{eq:heu:mean-field-expansion} in the spirit of the theory of matching asymptotics is out of the scope of this present article.
\end{Rem}

\begin{Rem}\label{rem:heu:dispersion-effect}
The dispersion effects are evident from the expression \eqref{eq:heu:effective-diffusion} of the effective diffusion in the homogenized equation, i.e. the effective diffusion coefficients depend on the convective velocity. This is because of the strong convection in the scaled convection-diffusion equation \eqref{eq:heu:CD}.
\end{Rem}

\begin{Rem}\label{rem:heu:cell-problem-flow-rep}
The cell problem \eqref{eq:heu:cell-problem} in Proposition \ref{prop:heu:formal-result} is given in fixed spatial coordinate, i.e. $\omega\equiv\omega(x,y)$. As our analysis essentially considers the asymptotic expansion in moving coordinates along flows, we can recast the cell problem \eqref{eq:heu:cell-problem} along flows, i.e. for the flow representation of the cell solutions $\widetilde{\omega}(\tau,\X,y)$:
\begin{align}\label{eq:heu:cell-problem-flow-rep}
\widetilde{\bf b}(\tau,\X,y)\cdot 
\Big( 
\nabla_y \widetilde{\omega}_i(\tau,\X,y) + {\bf e}_i 
\Big)
- 
\nabla_y \cdot 
\Big( 
\widetilde{\bf D}(\tau,\X,y) 
\left( \nabla_y \widetilde{\omega}_i(\tau,\X,y) + {\bf e}_i \right) 
\Big)
= 
\widetilde{\bar{\bf b}}(\tau,\X)\cdot {\bf e}_i.
\end{align}
The above problem is posed on $\T^d$. The spatial variable $\X$ and the fast time variable $\tau$ are treated as parameters.
\end{Rem}

\begin{Rem}\label{rem:heu:effective-diffusion-not-symmetric}
Even though the molecular diffusion matrix ${\bf D}(x,y)$ is assumed to be symmetric, the effective diffusion matrix $\mathfrak{D}(\X)$ in the homogenized equation \eqref{eq:heu:homogenized-equation} need not be symmetric as is evident from the expression \eqref{eq:heu:effective-diffusion-mathfrak-B} for $\mathfrak{B}(\tau,\X)$. The elements of the symmetric part of $\mathfrak{B}$ are
\[
\mathfrak{B}^{\rm sym}_{ij}
=
\int\limits_{\T^d}
\widetilde{\bf D}(\tau,\X,y)
\Big(
\nabla_y\widetilde{\omega}_j(\tau,\X,y)
+ 
{\bf e}_j
\Big)
\cdot
\Big(
\nabla_y\widetilde{\omega}_i(\tau,\X,y)
+
{\bf e}_i
\Big)
\, {\rm d}y
\]
and the elements of the skew-symmetric part of the matrix $\mathfrak{B}$ are
\begin{align*}
\mathfrak{B}^{\rm asym}_{ij}
& =
\int\limits_{\T^d}
\widetilde{\omega}_j(\tau,\X,y)
\left( 
\widetilde{\bar{\bf b}}(\tau,\X)
-
\widetilde{{\bf b}}(\tau,\X,y)
\right)
\cdot {\bf e}_i
\, {\rm d}y
\\
& -
\int\limits_{\T^d}
\widetilde{\bf D}(\tau,\X,y)
\nabla_y\widetilde{\omega}_i(\tau,\X,y)
\cdot
\Big(
\nabla_y\widetilde{\omega}_j(\tau,\X,y)
+
{\bf e}_j
\Big)
\, {\rm d}y,
\end{align*}
where we have used the cell problem for flow representations \eqref{eq:heu:cell-problem-flow-rep} to arrive at the above simplified expression for the skew-symmetric part. Note that the symmetric and skew-symmetric parts of the effective diffusion matrix $\mathfrak{D}(\X)$ are given in terms of $\mathfrak{B}$ by:
\begin{align*}
\mathfrak{D}^{\rm sym}(\X)
=
\lim_{\ell\to\infty}
\frac{1}{2\ell}
\int\limits_{-\ell}^{+\ell}
\widetilde{J}(\tau,\X) 
\mathfrak{B}^{\rm sym}(\tau,\X)
{}^\top \!\! \, \widetilde{J}(\tau,\X)
\, {\rm d}\tau,
\\
\mathfrak{D}^{\rm asym}(\X)
=
\lim_{\ell\to\infty}
\frac{1}{2\ell}
\int\limits_{-\ell}^{+\ell}
\widetilde{J}(\tau,\X) 
\mathfrak{B}^{\rm asym}(\tau,\X)
{}^\top \!\! \, \widetilde{J}(\tau,\X)
\, {\rm d}\tau.
\end{align*}
The contribution of the non-symmetric part of the effective diffusion to the dynamics of the homogenized equation \eqref{eq:heu:homogenized-equation} is because of the fact that the effective diffusion coefficient $\mathfrak{D}$ is space dependent. In a purely periodic setting, i.e. when ${\bf b}(x,y)\equiv {\bf b}(y)$ and ${\bf D}(x,y)\equiv {\bf D}(y)$, the matrix $\mathfrak{B}$ in Proposition \ref{prop:heu:formal-result} is space independent. As the Jacobian matrix in the purely periodic case is the identity, we have that $\mathfrak{D}=\mathfrak{B}$. The matrix $\mathfrak{D}$ being independent of the spatial variable implies that the parabolic homogenized equation \eqref{eq:heu:homogenized-equation} is unaffected by the skew-symmetric part of the effective diffusion matrix.
\end{Rem}

\begin{Rem}\label{rem:heu:average-tau}
The expression \eqref{eq:heu:effective-diffusion} for the effective diffusion involves the averaging in the fast time variable. In this section dealing with formal derivation of the homogenized limit, we admit that the limits in the expression of the effective diffusion exist and are finite. In Section \ref{sec:abs}, we introduce a notion of weak convergence in some Lebesgue function spaces which proves that these limits indeed exist and are finite under certain assumptions on the coefficients. Note that some of these assumptions are required: in Counterexample \ref{Exm:CExm:non-uniqueness} in Section \ref{sec:Exm}, we provide an explicit example where these limits do not exist, and in fact multiple limit equations can be obtained on different sequences $\eps\to0$.
\end{Rem}
\begin{Rem}
An interesting feature in the expression \eqref{eq:heu:effective-diffusion} is that the integrands are all in their flow representations. This suggests that the effective diffusion is the cumulative effect of the convection and diffusion effects averaged along the flows.
\end{Rem}

\begin{proof}[Proof of Proposition \ref{prop:heu:formal-result}]
The equations at different orders of $\eps$ obtained by inserting the asymptotic expansion \eqref{eq:heu:mean-field-expansion} in the scaled equation \eqref{eq:heu:CD} are
\begin{equation}\label{eq:heu:cascade}
\begin{array}{ccl}
\displaystyle
\mathcal{O}(\eps^{-2}): &
\displaystyle
\widetilde{{\bf b}} \cdot \nabla_y u_0 - \nabla_y \cdot \left( \widetilde{{\bf D}}\nabla_y u_0\right) &
\displaystyle
= 0,
\\[0.3 cm]
\displaystyle
\mathcal{O}(\eps^{-1}): &
\displaystyle
\widetilde{{\bf b}} \cdot \nabla_y u_1 - \nabla_y \cdot \left( \widetilde{{\bf D}}\nabla_y u_1\right) &
\displaystyle
= \nabla_y \cdot \left( \widetilde{{\bf D}} {}^\top \!\!\, \widetilde{J} \nabla_{\scriptscriptstyle X} u_0\right)
+ {}^\top \!\!\, \widetilde{J} \nabla_{\scriptscriptstyle X} \cdot \left( \widetilde{{\bf D}}\nabla_y u_0 \right)
\\[0.3 cm]
& &
\displaystyle
+
\left( \widetilde{\bar{{\bf b}}} - \widetilde{{\bf b}} \right)\cdot  \left( {}^\top \!\!\, \widetilde{J} \nabla_{\scriptscriptstyle X} u_0\right)
- \frac{\partial u_0}{\partial \tau},
\\[0.3 cm]
\displaystyle
\mathcal{O}(\eps^{0}): &
\displaystyle
\widetilde{{\bf b}} \cdot \nabla_y u_2 - \nabla_y \cdot \left( \widetilde{{\bf D}}\nabla_y u_2\right) &
\displaystyle
= -\frac{\partial u_0}{\partial t} -\frac{\partial u_1}{\partial \tau}
+
\left( \widetilde{\bar{{\bf b}}} -  \widetilde{{\bf b}}\right)\cdot  \left( {}^\top \!\!\, \widetilde{J} \nabla_{\scriptscriptstyle X} u_1\right)
\\[0.3 cm]
& & 
\displaystyle
+ {}^\top \!\!\, \widetilde{J} \nabla_{\scriptscriptstyle X}
\cdot
\left(
\widetilde{{\bf D}}
\left(
{}^\top \!\!\, \widetilde{J} \nabla_{\scriptscriptstyle X} u_0
+
\nabla_y u_1
\right)
\right)
+ \nabla_y \cdot \left( \widetilde{{\bf D}} {}^\top \!\!\, \widetilde{J} \nabla_{\scriptscriptstyle X} u_1\right),
\end{array}
\end{equation}
where the flow representation of the Jacobian matrix and coefficients are used. Note that the relation \eqref{eq:heu:lem:bar-b-Phi=J-bar-b-x} is needed, for example, to show the right hand side of the $\mathcal{O}(\eps^{-2})$ equation is zero. We remark that all the equations in \eqref{eq:heu:cascade} have the same structure as the boundary value problem \eqref{eq:heu:fredholm} addressed in Lemma \ref{lem:heu:fredholm} which says that the solvability of these equations is subject to satisfying the compatibility condition \eqref{eq:heu:fredholm-compatible}.

The compatibility condition \eqref{eq:heu:fredholm-compatible} is trivially satisfied for the equation of $\mathcal{O}(\eps^{-2})$ in \eqref{eq:heu:cascade}. Further, the equation of $\mathcal{O}(\eps^{-2})$ in \eqref{eq:heu:cascade} implies that $u_0$ is independent of $y$, i.e.
\begin{align*}
u_0(t,\tau,{\scriptstyle X},y)\equiv u_0(t,\tau,{\scriptstyle X}).
\end{align*}
So the term involving $\nabla_y u_0$ in the equation of $\mathcal{O}(\eps^{-1})$ vanishes. To check if the right hand side of the equation of $\mathcal{O}(\eps^{-1})$ in \eqref{eq:heu:cascade} satisfies the compatibility condition \eqref{eq:heu:fredholm-compatible}, consider
\begin{align*}
\int\limits_{\T^d}
\nabla_y \cdot \left( \widetilde{{\bf D}}(\tau,\X,y) {}^\top \!\!\, \widetilde{J} \nabla_{\scriptscriptstyle X} u_0\right)
\, {\rm d}y
+ \int\limits_{\T^d}
\left( \widetilde{\bar{{\bf b}}}(\tau,\X) - \widetilde{{\bf b}}(\tau,\X,y) \right)\cdot  \left( {}^\top \!\!\, \widetilde{J} \nabla_{\scriptscriptstyle X} u_0\right)
\, {\rm d}y
- \int\limits_{\T^d}
\frac{\partial u_0}{\partial \tau}
\, {\rm d}y.
\end{align*}
The first integral in the previous expression vanishes by integration by parts. The second integral in the previous expression vanishes as well, thanks to the definition \eqref{eq:heu:mean-field} of the mean field $\bar{{\bf b}}(x)$, and as neither $J$ nor $u_0$ depend upon $y$. In the third integral, since $u_0$ is independent of the $y$ variable, in order to satisfy the compatibility condition, we should have that $u_0$ is independent of the fast time variable. Hence we have $u_0(t,\tau,{\scriptstyle X},y)\equiv u_0(t,{\scriptstyle X})$.

The linearity of equations in \eqref{eq:heu:cascade} implies that we can separate the variables in the first order corrector as in \eqref{eq:heu:form-of-u_1}.
The function $\widetilde{\omega}(\tau, \X, y)$ is the flow representation of the function $\omega=\left(\omega_i\right)_{1\le i\le d}$ whose components solve the cell problem \eqref{eq:heu:cell-problem-flow-rep} (see Remark \ref{rem:heu:cell-problem-flow-rep}).

Finally, we write the compatibility condition for the equation of $\mathcal{O}(\eps^{0})$ in \eqref{eq:heu:cascade}:
\begin{align*}
\int\limits_{\T^d}
& \frac{\partial u_0}{\partial t} 
\, {\rm d}y
+ \int\limits_{\T^d}
\frac{\partial u_1}{\partial \tau}
\, {\rm d}y
\\
& =
\int\limits_{\T^d}
\left( \widetilde{\bar{{\bf b}}}(\tau,\X) - \widetilde{{\bf b}}(\tau,\X,y) \right)\cdot  \left( {}^\top \!\!\, \widetilde{J} \nabla_{\scriptscriptstyle X} u_1\right)
\, {\rm d}y
+ \int\limits_{\T^d}
{}^\top \!\!\, \widetilde{J} \nabla_{\scriptscriptstyle X}
\cdot
\left(
\widetilde{{\bf D}}(\tau,\X,y)
\left(
{}^\top \!\!\, \widetilde{J} \nabla_{\scriptscriptstyle X} u_0
+ \nabla_y u_1
\right)
\right)
\, {\rm d}y.
\end{align*}
The previous expression contains terms that depend on the fast time variable $\tau$. We propose to average the above equation in the $\tau$ variable:
\begin{equation}\label{eq:heu:compatibility+tau-average}
\begin{aligned}
\frac{\partial u_0}{\partial t} 
+
\lim_{\ell\to\infty}
\frac{1}{2\ell}
\int\limits_{-\ell}^{+\ell}
\int\limits_{\mathbb{T}^d}
\frac{\partial u_1}{\partial \tau}
& \, {\rm d}y\, {\rm d}\tau
= 
\lim_{\ell\to\infty}
\frac{1}{2\ell}
\int\limits_{-\ell}^{+\ell}
\int\limits_{\T^d}
\left( \widetilde{\bar{{\bf b}}}(\tau,\X) - \widetilde{{\bf b}}(\tau,\X,y) \right)\cdot  \left( {}^\top \!\!\, \widetilde{J}(\tau,\X)  \nabla_{\scriptscriptstyle X} u_1\right)
\, {\rm d}y\, {\rm d}\tau
\\
&
+ \lim_{\ell\to\infty}
\frac{1}{2\ell}
\int\limits_{-\ell}^{+\ell}
\int\limits_{\T^d}
{}^\top \!\!\, \widetilde{J}(\tau,\X)  \nabla_{\scriptscriptstyle X}
\cdot
\left(
\widetilde{{\bf D}}(\tau,\X,y)
\left(
{}^\top \!\!\, \widetilde{J}(\tau,\X)  \nabla_{\scriptscriptstyle X} u_0
\right)
\right)
\, {\rm d}y\, {\rm d}\tau
\\
& + \lim_{\ell\to\infty}
\frac{1}{2\ell}
\int\limits_{-\ell}^{+\ell}
\int\limits_{\T^d}
{}^\top \!\!\, \widetilde{J}(\tau,\X)  \nabla_{\scriptscriptstyle X}
\cdot
\left(
\widetilde{{\bf D}}(\tau,\X,y)
\nabla_y u_1
\right)
\, {\rm d}y\, {\rm d}\tau.
\end{aligned}
\end{equation}
The second term on the left hand side of the previous expression is zero. The first term on the right hand side of \eqref{eq:heu:compatibility+tau-average} can be successively written as
\begin{align*}
& \lim_{\ell\to\infty}
\frac{1}{2\ell}
\int\limits_{-\ell}^{+\ell}
\int\limits_{\T^d}
\left( \widetilde{\bar{{\bf b}}}(\tau,\X) - \widetilde{\bf b}(\tau,\X,y) \right)\cdot  \left( {}^\top \!\!\, \widetilde{J}(\tau,\X)  \nabla_{\scriptscriptstyle X} u_1\right)
\, {\rm d}y\, {\rm d}\tau
\\
& =
\lim_{\ell\to\infty}
\frac{1}{2\ell}
\int\limits_{-\ell}^{+\ell}
\int\limits_{\T^d}
\widetilde{J}(\tau,\X)\left( \widetilde{\bar{{\bf b}}}(\tau,\X) - \widetilde{\bf b}(\tau,\X,y) \right)
\cdot 
\nabla_{\scriptscriptstyle X} \left( \widetilde{\omega}(\tau,\X,y)\cdot  {}^\top \!\!\, \widetilde{J}(\tau,\X) \nabla_{\scriptscriptstyle X} u_0(t,\X)\right)
\, {\rm d}y\, {\rm d}\tau
\\
& =
\nabla_{\scriptscriptstyle X}
\cdot
\lim_{\ell\to\infty}
\frac{1}{2\ell}
\int\limits_{-\ell}^{+\ell}
\int\limits_{\T^d}
\widetilde{J}(\tau,\X)\left( \widetilde{\bar{{\bf b}}}(\tau,\X) - \widetilde{\bf b}(\tau,\X,y) \right){}^\top \!\!\, \widetilde{\omega}(\tau,\X,y) {}^\top \!\!\, \widetilde{J}(\tau,\X) \nabla_{\scriptscriptstyle X} u_0(t,\X)
\, {\rm d}y\, {\rm d}\tau,
\end{align*}
where we are able to move the $\X$ derivative thanks to Lemma \ref{lem:heu:basic-facts}.(ii). The second term on the right hand side of \eqref{eq:heu:compatibility+tau-average} evaluates to
\begin{align*}
& \lim_{\ell\to\infty}
\frac{1}{2\ell}
\int\limits_{-\ell}^{+\ell}
\int\limits_{\T^d}
{}^\top \!\!\, \widetilde{J}(\tau,\X)  \nabla_{\scriptscriptstyle X}
\cdot
\left(
\widetilde{{\bf D}}(\tau,\X,y)
\left(
{}^\top \!\!\, \widetilde{J}(\tau,\X)  \nabla_{\scriptscriptstyle X} u_0
\right)
\right)
\, {\rm d}y\, {\rm d}\tau
\\
& =
\nabla_{\scriptscriptstyle X}
\cdot
\lim_{\ell\to\infty}
\frac{1}{2\ell}
\int\limits_{-\ell}^{+\ell}
\int\limits_{\T^d}
\widetilde{J}(\tau,\X) 
\widetilde{\bf D}(\tau,\X,y)
{}^\top \!\!\, \widetilde{J}(\tau,\X)  \nabla_{\scriptscriptstyle X} u_0
\, {\rm d}y\, {\rm d}\tau.
\end{align*}
The third term on the right hand side of \eqref{eq:heu:compatibility+tau-average} can be successively written as
\begin{align*}
& \lim_{\ell\to\infty}
\frac{1}{2\ell}
\int\limits_{-\ell}^{+\ell}
\int\limits_{\T^d}
{}^\top \!\!\, \widetilde{J}(\tau,\X)  \nabla_{\scriptscriptstyle X}
\cdot
\left(
\widetilde{{\bf D}}(\tau,\X,y)
\nabla_y u_1
\right)
\, {\rm d}y\, {\rm d}\tau
\\
& =
\nabla_{\scriptscriptstyle X}
\cdot
\lim_{\ell\to\infty}
\frac{1}{2\ell}
\int\limits_{-\ell}^{+\ell}
\int\limits_{\T^d}
\widetilde{J}(\tau,\X) 
\widetilde{\bf D}(\tau,\X,y)
\nabla_y \left( \widetilde{\omega}(\tau,\X,y)\cdot  {}^\top \!\!\, \widetilde{J}(\tau,\X) \nabla_{\scriptscriptstyle X} u_0(t,\X)\right)
\, {\rm d}y\, {\rm d}\tau
\\
& =
\nabla_{\scriptscriptstyle X}
\cdot
\lim_{\ell\to\infty}
\frac{1}{2\ell}
\int\limits_{-\ell}^{+\ell}
\int\limits_{\T^d}
\widetilde{J}(\tau,\X) 
\widetilde{\bf D}(\tau,\X,y)
{}^\top \!\!\, \nabla_y \widetilde{\omega}(\tau,\X,y){}^\top \!\!\, \widetilde{J}(\tau,\X) \nabla_{\scriptscriptstyle X} u_0(t,\X)
\, {\rm d}y\, {\rm d}\tau.
\end{align*}
Again, we are able to move the $\X$ derivative past the Jacobian thanks to Lemma \ref{lem:heu:basic-facts}.(i). Considering all the above observations, the compatibility condition \eqref{eq:heu:compatibility+tau-average} can be rewritten as a diffusion equation for $u_0(t,\X)$, i.e. \eqref{eq:heu:homogenized-equation}-\eqref{eq:heu:homogenized-initial}. The expression for the effective diffusion coefficient is given by
\begin{equation*}
\begin{aligned}
\mathfrak{D}(\X)
& = 
\lim_{\ell\to\infty}
\frac{1}{2\ell}
\int\limits_{-\ell}^{+\ell}
\int\limits_{\T^d}
\widetilde{J}(\tau,\X)\left( \widetilde{\bar{{\bf b}}}(\tau,\X) - \widetilde{\bf b}(\tau,\X,y) \right)  {}^\top \!\!\, \widetilde{\omega}(\tau,\X,y)  {}^\top \!\!\, \widetilde{J}(\tau,\X)\, {\rm d}y\, {\rm d}\tau
\\
& 
+
\lim_{\ell\to\infty}
\frac{1}{2\ell}
\int\limits_{-\ell}^{+\ell}
\int\limits_{\T^d}
\widetilde{J}(\tau,\X) 
\widetilde{\bf D}(\tau,\X,y)
{}^\top \!\!\, \widetilde{J}(\tau,\X)
\, {\rm d}y\, {\rm d}\tau
\\
& 
+
\lim_{\ell\to\infty}
\frac{1}{2\ell}
\int\limits_{-\ell}^{+\ell}
\int\limits_{\T^d}
\widetilde{J}(\tau,\X) 
\widetilde{\bf D}(\tau,\X,y)
{}^\top \!\!\, \nabla_y \widetilde{\omega}(\tau,\X,y){}^\top \!\!\, \widetilde{J}(\tau,\X)
\, {\rm d}y\, {\rm d}\tau.
\end{aligned}
\end{equation*}
Moving the $y$ integration inside, the expression for the effective diffusion becomes
\begin{equation*}
\begin{aligned}
\mathfrak{D}(\X)
& = 
\lim_{\ell\to\infty}
\frac{1}{2\ell}
\int\limits_{-\ell}^{+\ell}
\widetilde{J}(\tau, \X)
\left(
\int\limits_{\T^d}
\left( \widetilde{\bar{{\bf b}}}(\tau, \X) - \widetilde{{\bf b}}(\tau,\X,y) \right)
{}^\top \!\!\, \widetilde{\omega}(\tau,\X,y)
\, {\rm d}y
\right)
{}^\top \!\!\, \widetilde{J}(\tau,\X)
\, {\rm d}\tau
\\
& 
+
\lim_{\ell\to\infty}
\frac{1}{2\ell}
\int\limits_{-\ell}^{+\ell}
\widetilde{J}(\tau,\X)
\left(
\int\limits_{\T^d} 
\Big\{
\widetilde{\bf D}(\tau,\X,y)
+
\widetilde{\bf D}(\tau,\X,y)
{}^\top \!\!\, \nabla_y \widetilde{\omega}(\tau,\X,y)
\Big\}
\, {\rm d}y
\right)
{}^\top \!\!\, \widetilde{J}(\tau,\X)
\, {\rm d}\tau
\\
& =
\lim_{\ell\to\infty}
\frac{1}{2\ell}
\int\limits_{-\ell}^{+\ell}
\widetilde{J}(\tau,\X) 
\mathfrak{B}(\tau,\X)
{}^\top \!\!\, \widetilde{J}(\tau,\X)
\, {\rm d}\tau,
\end{aligned}
\end{equation*}
where the elements of $\mathfrak{B}$ are given by
\begin{align*}
\mathfrak{B}_{ij}(\tau,\X)
= 
\int\limits_{\T^d}
\Big\{
\left( \widetilde{\bar{{\bf b}}}_i(\tau, \X) - \widetilde{{\bf b}}_i(\tau,\X,y) \right)
\widetilde{\omega}_j(\tau,\X,y)
+
\widetilde{\bf D}_{ij}(\tau,\X,y)
+
\widetilde{\bf D}(\tau,\X,y)\nabla_y\widetilde{\omega}_j(\tau,\X,y)
\cdot {\bf e}_i
\Big\}
\, {\rm d}y
\end{align*}
for each $i,j\in\{1,\cdots,d\}$. To simplify the expression for the matrix $\mathfrak{B}$, we test the equation \eqref{eq:heu:cell-problem-flow-rep} for the flow-representation $\widetilde{\omega}_i$ by $\widetilde{\omega}_j$ and deduce
\begin{align*}
\int\limits_{\T^d}
\left( \widetilde{\bar{{\bf b}}}_i(\tau, \X) - \widetilde{{\bf b}}_i(\tau,\X,y) \right)
\widetilde{\omega}_j(\tau,\X,y)
\, {\rm d}y
=
\int\limits_{\T^d}
\left( 
\widetilde{{\bf b}}(\tau,\X,y)
\cdot 
\nabla_y \widetilde{\omega}_i(\tau,\X,y) 
\right)
\widetilde{\omega}_j(\tau,\X,y)
\, {\rm d}y
\\
+
\int\limits_{\T^d}
\widetilde{\bf D}(\tau,\X,y)\nabla_y\widetilde{\omega}_j(\tau,\X,y)
\cdot
\nabla_y\widetilde{\omega}_i(\tau,\X,y)
\, {\rm d}y
+
\int\limits_{\T^d}
\widetilde{\bf D}(\tau,\X,y)\nabla_y\widetilde{\omega}_j(\tau,\X,y)
\cdot
{\bf e}_i
\, {\rm d}y.
\end{align*}
Using the above equation, we can rewrite the elements of the matrix $\mathfrak{B}$ as in \eqref{eq:heu:effective-diffusion-mathfrak-B}.
\end{proof}

\begin{Rem}\label{rem:heu:add-constant-cell-solution}
The solution $\omega_i$ to the cell problem \eqref{eq:heu:cell-problem} is unique up to addition of constants in the $y$-variable, i.e. up to addition of a function $\eta(t,\tau,\X)$. However, any such function would not contribute to the expression of the effective diffusion. It is evident from the equation \eqref{eq:heu:compatibility+tau-average}. So, for our purposes at hand, we shall not dwell on characterizing $\eta(t,\tau,\X)$. It should be noted that the first order corrector obtained in \eqref{eq:heu:u-eps-approx} essentially is considering the oscillations in the space variable. We have not characterized the first order corrector with regard to the fast time variable.
\end{Rem}

\begin{Prop}\label{prop:heu:homo-eq-exist}
The homogenized equation \eqref{eq:heu:homogenized-equation}-\eqref{eq:heu:homogenized-initial} admits a unique solution such that
\begin{align*}
u_0(t,\X) \in C([0,T];L^2(\R^d));
\qquad
\nabla_\XX u_0(t,\X) \in [L^2((0,T)\times\R^d)]^d.
\end{align*}
\end{Prop}

\begin{proof}
The elements of the symmetric part of the matrix $\mathfrak{B}$ are given by
\begin{align*}
\mathfrak{B}^{\mbox{\tiny sym}}_{ij}(\tau,\X)
= 
\int\limits_{\T^d}
\widetilde{\bf D}(\tau,\X,y)
\Big(
\nabla_y\widetilde{\omega}_j(\tau,\X,y)
+ 
{\bf e}_j
\Big)
\cdot
\Big(
\nabla_y\widetilde{\omega}_i(\tau,\X,y)
+
{\bf e}_i
\Big)
\, {\rm d}y.
\end{align*}
It is positive definite because
\begin{align*}
{}^\top\!\!\, \xi 
\, \mathfrak{B}^{\mbox{\tiny sym}}
\, \xi
\ge
\lambda
\int\limits_{\T^d}
|\nabla_y \widetilde{\omega}_\xi + \xi|^2
\, {\rm d}y
=\lambda
\int\limits_{\T^d}
|\nabla \widetilde{\omega}_\xi|^2+2\xi\cdot\nabla_y \widetilde{\omega}_\xi+|\xi|^2
\, {\rm d}y
\ge \lambda |\xi|^2,
\quad
\mbox{ for all }\xi\in\R^d
\end{align*}
where $\widetilde{\omega}_\xi:=\widetilde{\omega} \cdot \xi$, and the last inequality follows from the vanishing of the second of the three terms in the integrand due to $y$-periodicity of $\widetilde{\omega}$.

Take the effective diffusion matrix $\mathfrak{D}$ and consider, for $\xi\ne0$,
\begin{align*}
{}^\top\!\!\, \xi 
\, \mathfrak{D}
\, \xi
& =
\lim_{\ell\to\infty}
\frac{1}{2\ell}
\int\limits_{-\ell}^{+\ell}
{}^\top\!\!\, \xi 
\widetilde{J}(\tau,\X) 
\mathfrak{B}(\tau,\X)
{}^\top \!\! \, \widetilde{J}(\tau,\X)
\, \xi
\, {\rm d}\tau
\\
& =
\lim_{\ell\to\infty}
\frac{1}{2\ell}
\int\limits_{-\ell}^{+\ell}
{}^\top
\left[
{}^\top \!\! \, \widetilde{J}(\tau,\X)
\, \xi
\right]
\mathfrak{B}(\tau,\X)
{}^\top \!\! \, \widetilde{J}(\tau,\X)
\, \xi
\, {\rm d}\tau\\
&\ge \lambda \lim_{\ell\to\infty}
\frac{1}{2\ell}
\int\limits_{-\ell}^{+\ell} 
\left|
{}^\top\!\!\, \widetilde{J}(\tau,\X)\xi
\right|^2\, {\rm d}\tau 
\ge C\lambda |\xi|^2>0.
\end{align*}
This holds, thanks firstly to the positive definite property of the matrix $\mathfrak{B}$, and secondly to uniform bounds from below on the Jacobian matrix, i.e. 
\begin{equation*}
C^{-1}|\xi|=C^{-1}|J(\tau,x)^{-1}J(\tau,x)\xi|=C^{-1}|J(-\tau,\Phi_{-\tau}(x))J(\tau,x)\xi|\le |J(\tau,x)\xi|\le C|\xi|
\end{equation*}
for the uniform constant $C$ given by Assumption \ref{heu:ass:bound-on-J}, and where we have used that the flow is autonomous to express the inverse of the Jacobian in terms of the Jacobian at a different point. Thus we have shown that the effective diffusion coefficient $\mathfrak{D}(\X)$ is positive definite. On the other hand, $\mathfrak{D}(\X)$ is uniformly bounded from above. Then, it is a standard process to prove existence and uniqueness for \eqref{eq:heu:homogenized-equation}-\eqref{eq:heu:homogenized-initial} (cf. \cite{ladyzenskaja1968linear} if necessary).
\end{proof}

\section{$\Sigma$-convergence along flows}\label{sec:abs}

This section puts forth a new notion of convergence in $L^p$-spaces (with $1<p<\infty$) which gives a rigorous justification of (at least) the first two terms in the asymptotic expansion along mean flows \eqref{eq:heu:mean-field-expansion} postulated in Section \ref{sec:heuristics}, i.e. to justify the approximation \eqref{eq:heu:u-eps-approx} in Proposition \ref{prop:heu:formal-result}. This work is inspired from the seminal works of G. Nguetseng \cite{nguetseng1989general} and G. Allaire \cite{allaire1992homogenization}. In Section \ref{sec:heuristics}, we have formally derived the homogenized limit and obtained an explicit expression for the effective diffusion \eqref{eq:heu:effective-diffusion}. As mentioned in Remark \ref{rem:heu:average-tau}, there was an inherent assumption that the limits in the fast time variable exist and are finite.

The works \cite{nguetseng1989general, allaire1992homogenization} are in the context of periodic homogenization. G. Allaire does mention in \cite{allaire1992homogenization} that it would be interesting to extend the two-scale convergence theory from the periodic setting to the more general almost-periodic setting (see p.1484 in \cite{allaire1992homogenization}). This has been addressed in the past one and a half decade \cite{casado2002two, nguetseng2003homogenization, nguetseng2004homogenization, nguetseng2011sigma, sango2011generalized}. In all these new developments, a central role is played by the notion of \emph{algebra with mean value} introduced by Zhikov and Krivenko in \cite{zhikov1983averaging}.

In this section we present the abstract framework of \emph{$\Sigma$-convergence along flows}. In subsections \ref{ssec:abs:algebras-w.m.v.}-\ref{ssec:abs:technicalities} we develop enough of the theory of \emph{algebras with mean value} for our later purposes. As we do not aim to extend this theory beyond what already exists, we shall not give the theory in full generality and we refer the reader to existing literature (e.g. \cite{casado2002two, nguetseng2003homogenization, nguetseng2004homogenization, nguetseng2010reiterated, sango2011generalized, ambrosio2009multiscale}, see also \cite{francuu2012outline} for an introductory exposition and \cite{jikov1994homogenization} for a pedagogical exposition) for a more complete presentation and full proofs. In subsections \ref{ssec:abs:2-scale-convergence-along-flows}-\ref{ssec:abs:additional-bounds} we introduce the new concept of \emph{$\Sigma$-convergence along flows} and prove compactness results.

\subsection{Algebras with mean value}\label{ssec:abs:algebras-w.m.v.}

We shall denote the space of bounded uniformly continuous functions on $\R$ by $\mbox{BUC}(\R)$.

\begin{Defi}[Algebra with mean value]\label{def:abs:algebra-w.m.v.}
An \emph{algebra with mean value} (or \emph{algebra w.m.v.}, in short) is a Banach sub-algebra $\mathcal{A}$ of $\mbox{\em BUC}(\R)$ such that the following hold:
\begin{enumerate}
\item[(i)] $\mathcal{A}$ contains the constants.
\item[(ii)] $\mathcal{A}$ is translation invariant, i.e. for every $f\in\mathcal{A}$ and $a\in\R$, $f(\cdot-a)\in\mathcal{A}$.
\item[(iii)] Any $f\in \mathcal{A}$ possesses a mean value $M(f)$, by which we mean that 
\begin{align*}
f\left(\frac{\cdot}{\eps}\right)\rightharpoonup M(f) \text{ in }L^\infty(\R)\text{-weak* as }\eps\to0.
\end{align*}
\end{enumerate}
\end{Defi}
Note that the mean value can be equivalently expressed as
\begin{align*}
M(f)
=
\lim_{\ell\to\infty}\frac1{2\ell}\int\limits_{-\ell}^{+\ell}f(\tau)
\, {\rm d}\tau
\end{align*}
and that this limit exists for any $f\in\mathcal{A}$.

The theory of algebra w.m.v. is developed for the Banach space of bounded uniformly continuous functions on $\R^d$, i.e. in any arbitrary dimension. As this current work considers a fast time variable (i.e. in one dimension), we recall all the essential notions in this theory with emphasis on one dimension.

\subsection{Gelfand representation theory}\label{ssec:abs:gelfand-representation}

\begin{Defi}[Spectrum of a Banach algebra]\label{def:abs:spectrum-banach-algebra}
Given a commutative Banach algebra $\mathcal{A}$ with an identity $1\in\mathcal{A}$, we define its \emph{spectrum} $\Delta(\mathcal{A})$ as the set of algebra homeomorphisms, i.e. the maps $s:\mathcal{A}\to\C$ such that
\begin{enumerate}
\item $s$ is linear, i.e. for all $f,g\in\mathcal{A}, \lambda\in\C$, $s(f+g)=s(f)+s(g)$ and $s(\lambda f)=\lambda s(f)$,
\item $s$ is multiplicative, i.e. for all $f,g\in\mathcal{A}$, $s(fg)=s(f)s(g)$,
\item $s$ preserves the identity, i.e. $s(1)=1$.
\end{enumerate}
The elements $s\in\Delta(\mathcal{A})$ are called the \emph{characters} of $\mathcal{A}$.
\end{Defi}
As $\Delta(\mathcal{A})\subset \mathcal{A}'$, the topological dual of $\mathcal{A}$, we equip $\Delta(\mathcal{A})$ with the weak* subspace topology induced by $\mathcal{A}'$. This makes $\Delta(\mathcal{A})$ a compact Hausdorff space by the Banach-Alaoglu theorem. 

Of central importance to the study of Banach algebras is the Gelfand transform. We denote by $C(\Delta(\mathcal{A}))$, the space of complex-valued continuous functions on $\Delta(\mathcal{A})$.

\begin{Defi}[Gelfand transform]\label{def:abs:Gelfand-transform}
The \emph{Gelfand transform} is the map $\mathcal{G}:\mathcal{A}\to C(\Delta(\mathcal{A}))$ defined by $\mathcal{G}(f)(s)=s(f)$.
\end{Defi}

{\centering{\textsc{Notation:} For brevity, we denote $\mathcal{G}(f)$ as $\widehat{f}$.}}

The importance of the spectrum and Gelfand transform is in the following result, which allows us to replace the analysis of functions in $C(\R)$ with functions on a \emph{compact} space.
\begin{Thm}[Gelfand-Naimark]\label{thm:abs:Gelfand-Naimark}
Let $\mathcal{A}$ be a $\mathcal{C}^*$ algebra. Then $\mathcal{G}$ is an isometric isomorphism of $\mathcal{A}$ into $C(\Delta(\mathcal{A}))$.
\end{Thm}
The mean value operator $M$ is a bounded linear functional on $\mathcal{A}$. By identifying $\mathcal{A}$ with $C(\Delta(\mathcal{A}))$ using the Gelfand transform, and applying the Riesz representation theorem we arrive at the following proposition, the observation of which forms the basis of Nguetseng's formalism of \emph{Homogenization Structures} \cite{nguetseng2003homogenization, nguetseng2004homogenization}.
\begin{Prop}\label{prop:abs:beta}
Let $\mathcal{A}$ be an algebra w.m.v.. Then the mean value operator $M$ is represented by a Radon probability measure $\beta$ on $\Delta(\mathcal{A})$, i.e. for all $f\in\mathcal{A}$ we have:
\begin{align}\label{eq:abs:beta}
M(f)
=
\int\limits_{\Delta(\mathcal{A})}\widehat{f}(s)
\, {\rm d}\beta(s).
\end{align}
\end{Prop}
This allows us to introduce the space $L^2(\Delta(\mathcal{A})):=L^2(\Delta(\mathcal{A}),{\rm d}\beta)$. Note that $\beta$ may not be supported on the whole of $\Delta(\mathcal{A})$. Indeed, in the Example \ref{exm:abs:a.w.m.v-convergence-at-infinity} below it is a Dirac mass at a single point.

\subsection{Examples of algebras with mean value}\label{ssec:abs:examples-alegras-w.m.v}

To give some intuition for these objects we provide some examples.

\begin{Exm}[Periodic functions]\label{exm:abs:a.w.m.v-periodic}
Let $\mathcal{A}$ be the set of continuous functions from $\R$ to $\C$ which are periodic with period $L$. Then the characters $s\in\Delta(\mathcal{A})$ are the maps defined by $s_t(f)=f(t)$ for $t\in \R/(L\Z)$, so that $\Delta(\mathcal{A})$ can be identified with the torus of length $L$. The Gelfand transform takes $f\in \mathcal{A}\subset C(\R)$ to its representative on the torus. The mean value operator $M$ is given by 
\begin{align*}
M(f)
=
\int\limits_{\Delta(\mathcal{A})}\widehat{f}(s)
\, {\rm d}\beta(s)
=
\frac1L \int\limits_0^L f(\tau)
\, {\rm d}\tau.
\end{align*}
\end{Exm}

\begin{Exm}[Functions that converge at infinity]\label{exm:abs:a.w.m.v-convergence-at-infinity}
Let $\mathcal{A}$ be the space of continuous functions $f:\R\to\C$ that converge to a limit at infinity, i.e. $\lim_{|\tau|\to\infty}f(\tau)$ exists. Then the spectrum $\Delta(\mathcal{A})$ are the point evaluation maps $s_t(f)=f(t)$ for $t\in\R\cup\{\infty\}$, and the spectrum can be identified with $\bar{\R}$ the one point compactification of $\R$. Under this identification, the Gelfand transform takes a function $f\in \mathcal{A}$ to a function $\widehat{f}:\bar{\R}\to\C$ with $\widehat{f}(t)=f(t)$ for $t\in \R$ and $\widehat{f}(\infty)=\lim_{|\tau|\to\infty}f(\tau)$. The mean value operator $M$ acts by $M(f)=\widehat{f}(\infty)$.
\end{Exm}

\begin{Exm}[Almost-periodic functions]\label{exm:abs:AP}
Let $\mathsf{T}(\R)$ denote the set of all trigonometric polynomials, i.e. all $f(t)$ that are finite linear combinations of the functions in the set
\begin{align*}
\Big\{
\cos(kt), \sin(kt)
: k\in\R
\Big\}.
\end{align*}
The space of almost-periodic functions in the sense of Bohr \cite{bohr1947almost} is the closure of $\mathsf{T}(\R)$ in the supremum norm.

A function $f(t)\in L^2_{loc}(\R)$ is called almost-periodic in the sense of Besicovitch if there is a sequence in $\mathsf{T}(\R)$ that converges to $f(t)$ in the Besicovitch semi-norm (given by \eqref{eq:abs:Besicovitch-seminorm} below).

A function $f(t)\in\mbox{BUC}(\R)$ is said to be almost-periodic if the set of translates
\begin{align}\label{eq:abs:translates}
\Big\{
f(\cdot - a)
: a\in\R
\Big\}
\end{align}
is relatively compact in $\mbox{BUC}(\R)$.
\end{Exm}
All the above three definitions of \emph{almost-periodic functions} are equivalent \cite{jikov1994homogenization}.

We also give the example of weakly almost-periodic functions due to Eberlein \cite{eberlein1949abstract}.
\begin{Exm}[Weakly almost-periodic functions]\label{exm:abd:WAP}
A continuous function $f(t)\in\mbox{BUC}(\R)$ is weakly almost periodic if the set of translates \eqref{eq:abs:translates} is relatively weakly compact in $\mbox{BUC}(\R)$.
\end{Exm}
Readers are to consult \cite{sango2011generalized} for more information on the space of weakly almost-periodic functions.
\subsection{Besicovitch spaces}\label{ssec:Besicovitch}

\begin{Defi}[Besicovitch space]\label{def:abs:besicovitch}
For an algebra w.m.v. $\mathcal{A}$ the corresponding \emph{Besicovitch space} $\mathcal{B}^2=\mathcal{B}^2_{\mathcal{A}}$ is the abstract completion of $\mathcal{A}$ with respect to the Besicovitch semi-norm:
\begin{align}\label{eq:abs:Besicovitch-seminorm}
\| f \|_{\mathcal{B}_\mathcal{A}^2}^2
=
\limsup_{\ell\to\infty}\frac1{2\ell}\int\limits_{-\ell}^{+\ell} |f(\tau)|^2
\, {\rm d}\tau.
\end{align}
\end{Defi}
Note that the elements of $\mathcal{B}^2$ are equivalence classes of functions that are indistinguishable under \eqref{eq:abs:Besicovitch-seminorm}. The mean value operator $M$ extends to a bilinear form $M(fg)$ on $\mathcal{B}^2_{\mathcal{A}}$. It is a standard result (see e.g. \cite{sango2011generalized}) that the Gelfand transform is an isometric isomorphism between $\mathcal{B}^2$ and $L^2(\Delta(\mathcal{A}))$. Note that $\mathcal{B}^2_\mathcal{A}$ inherits the \emph{translation invariance} (in the sense of Definition \ref{def:abs:algebra-w.m.v.}(iii)) from $\mathcal{A}$.

\begin{Defi}[Ergodic algebra w.m.v.]\label{def:abs:ergodic-algebra}
An algebra w.m.v. $\mathcal{A}$ is said to be \emph{ergodic} if any $f\in \mathcal{B}_\mathcal{A}^2$ satisfying 
\begin{align*}
\| f(\cdot)-f(\cdot-a) \|_{\mathcal{B}^2_\mathcal{A}} = 0\qquad\text{for all }a\in\R
\end{align*}
is equivalent in $\mathcal{B}^2_\mathcal{A}$ to a constant.
\end{Defi}
It is easy to see that the constant in Definition \ref{def:abs:ergodic-algebra} must be $M(f)$. 

\begin{Rem}
All of the examples of algebras w.m.v. given in Subsection \ref{ssec:abs:examples-alegras-w.m.v} are ergodic.
\end{Rem}
For our purposes the importance of \emph{ergodicity} of an algebra w.m.v. is the following lemma, whose proof may be found in \cite{ambrosio2009multiscale}.
\begin{Lem}\label{lem:abs:mean-value-tau-derivative}
Let $\mathcal{A}$ be an ergodic algebra w.m.v. and $f\in \mathcal{B}^2_{\mathcal{A}}$ have the property that, for any $g\in \mathcal{A}$ with $\frac{{\rm d}g}{{\rm d}\tau}\in\mathcal{A}$ we have:
\begin{align*}
M\left(f\frac{{\rm d}g}{{\rm d}\tau}\right)
=
\int\limits_{\Delta{(\mathcal{A})}}\widehat{f}(s)\widehat{\frac{{\rm d}g}{{\rm d}\tau}}(s)
\, {\rm d}\beta(s)
= 0
\end{align*}
where the first equality is automatic. Then $f=M(f)$ in $\mathcal{B}_{\mathcal{A}}^2$ and equivalently $\widehat{f}=M(f)$ $\beta$-almost everywhere.
\end{Lem} 

\subsection{Product algebras and vector valued algebras}\label{ssec:abs:technicalities}

We wish to consider continuous functions $f(\tau,y)$ for which heuristically `$f$ is in $\mathcal{A}$ as a function of $\tau$' and `$f$ is in $C(\T^d)$ as a function of $y$'. To make sense of this, we recall that the tensor product $\mathcal{A}\otimes C(\T^d)$ is defined by
\begin{align*}
\mathcal{A}\otimes C(\T^d)
:=
\left\{\sum_{i=1}^Nf_ig_i:N\in\N, f_1,\dotsc,f_N\in \mathcal{A},\text{ and } g_1,\dotsc,g_N\in C(\T^d) \right\}
\end{align*}
and we define $\mathcal{A}\odot C(\T^d)$ as the closure of $\mathcal{A}\otimes C(\T^d)$ in the Banach algebra $\mbox{BUC}(\R\times\T^d)$. Note that by construction $\mathcal{A}\otimes C(\T^d)$ is dense in $\mathcal{A}\odot C(\T^d)$. More discussion of product algebras may be found in \cite{nguetseng2003homogenization, nguetseng2004homogenization}.

We will often need to use vector valued algebras of functions mapping to $\C^d$. This poses essentially no additional complications; we refer the reader to e.g. \cite{ambrosio2009multiscale} for details.

\subsection{$\Sigma$-convergence along flows}\label{ssec:abs:2-scale-convergence-along-flows}

Throughout this section we shall consider a flow $\Phi_\tau(x):\R\times\R^d\to\R^d$. One might think of $\Phi_\tau(x)$ as the flow of an autonomous ODE
\begin{align*}
\dot{\X} = \bar{\bf b}(\X),
\end{align*}
as was considered in Section \ref{sec:heuristics}, but this assumption will not be needed in this section. We will, however, make the following assumptions on the flow $\Phi_\tau(x)$.

\begin{Ass}\label{ass:abs:flow}
We assume that the flow $\Phi_\tau(x)$ satisfies the following:
\begin{enumerate}
\item[(i)] $\Phi_\tau(x)$ is continuously differentiable from $\R\times\R^d$ to $\R^d$.
\item[(ii)] $\Phi_\tau(x)$ satisfies the group property, i.e. $\Phi_t(\Phi_s(x))=\Phi_{t+s}(x)$ for all $t,s\in\R$ and $x\in\R^d$.
\item[(iii)] The Jacobian $J$ of $\Phi_\tau(x)$ defined by \eqref{eq:heu:Jacobian-matrix} is an \emph{uniformly bounded} function of $\tau$, locally uniformly in $x$,  i.e. for any compact $K\subset\R^d$ we have
\begin{align*}
\sup_{x\in K}\sup_{\tau\in\R}|J(\tau,x)|<\infty .
\end{align*}
\item[(iv)] For any $\tau\in \R$, $\Phi_\tau(x)$ is volume preserving, i.e. $\det(J(\tau, x))=1$.
\end{enumerate}
\end{Ass}

We now define the notion of weak $\Sigma$-convergence along flows, which generalizes the notion of two-scale convergence with drift introduced in \cite{maruvsic2005homogenization} and also the notion of $\Sigma$-convergence introduced in \cite{nguetseng2003homogenization}.

\begin{Defi}[weak $\Sigma$-convergence along flow]\label{def:abs:2scale-flow}
Let $\mathcal{A}$ be an algebra w.m.v.. Suppose $\Phi_\tau(x)$ be a flow satisfying Assumption \ref{ass:abs:flow} and let $u^\eps(t,x)$ be a sequence in $L^2((0,T)\times\R^d)$. We say that $u^\eps$ weakly $\Sigma$-converges along $\Phi_\tau(x)$ to a limit $u_0(t,\X,s,y)\in L^2((0,T)\times\R^d\times \Delta(\mathcal{A})\times \T^d)$ if, for any smooth test function $\psi(t,\X ,\tau,y)$ which is periodic in the $y$ variable and belongs to $\mathcal{A}$ in the $\tau$ variable, we have
\begin{equation}\label{eq:abs:sigma-convergence-flow-defn}
\begin{aligned}
\lim_{\eps\to 0}
\iint\limits_{(0,T)\times\R^d}
& 
u^\eps(t,x)\psi\left(t,\Phi_{-t/\eps}(x), \frac{t}{\eps}, \frac{x}{\eps}\right)
\, {\rm d}x\, {\rm d}t
= \\
&
\iiiint\limits_{(0,T)\times\R^d\times\Delta(\mathcal{A})\times\T^d}
u_0(t,\X,s,y)\widehat{\psi}(t,\X,s,y)
\, {\rm d}y \, {\rm d}\beta(s) \, {\rm d}\X\, {\rm d}t,
\end{aligned}
\end{equation}
where $\widehat{\psi}=\mathcal{G}(\psi)$ is the Gelfand transform of $\psi$ (Definition \ref{def:abs:Gelfand-transform}), $\beta$ is given by \eqref{eq:abs:beta} and $\mathcal{A}\odot C(\T^d)$ is defined in subsection \ref{ssec:abs:technicalities}.
\end{Defi}
{\centering\textsc{Notation:} We denote the weak $\Sigma$-convergence along flow $\Phi_\tau(x)$ by $u^\eps \sflow u_0$.}

{\centering\textsc{Convention:} Whenever the limit \eqref{eq:abs:sigma-convergence-flow-defn} holds, we call \emph{$u_0$ the $\Sigma$-$\Phi_\tau$ weak limit of $u^\eps$}.}
\begin{Rem}\label{rem:abs:characteristics-sigma-convergence}
The test functions in \eqref{eq:abs:sigma-convergence-flow-defn} are taken along rapidly moving coordinates in their second variable. This is analogous to the choice of test functions in the theory of two-scale convergence with drift \cite{maruvsic2005homogenization, allaire2008periodic}. Note also that the the $\Sigma$-$\Phi_\tau$ weak limit of the family $u^\eps(t,x)$ depends on the choice of the flow $\Phi_\tau(x)$. It should be noted that when $\Phi_\tau(x)=x$ for all $\tau\in\R$ and for each $x\in\R^d$, i.e. when the test functions in \eqref{eq:abs:sigma-convergence-flow-defn} are taken on a fixed coordinate system, the weak convergence given in Definition \ref{def:abs:2scale-flow} coincides with the notion of weak $\Sigma$-convergence with regard to the product algebra $\mathcal{A}\odot C(\T^d)$ developed in \cite{nguetseng2011sigma, sango2011generalized}.
\end{Rem}

\begin{Rem}\label{rem:abs:test-fn-no-space-oscillations}
Definition \ref{def:abs:2scale-flow} makes sense even for test functions $\psi(t,\X,\tau)$ without oscillations in space, and in this case the limit $u_0$ will be a function of $(t,\X,s)$ only.
\end{Rem}

\subsection{Compactness}\label{ssec:abs:compact}

To show that the Definition \ref{def:abs:2scale-flow} is not empty, we give the following weak-compactness result, which is the main result of this section.

\begin{Thm}\label{Thm:abs:compactness}
Let $\mathcal{A}$ be an algebra w.m.v.. Suppose $\Phi_\tau(x)$ be a flow satisfying Assumption \ref{ass:abs:flow} and let $u^\eps(t,x)$ be a uniformly (with respect to $\eps$) bounded sequence in $L^2((0,T)\times\R^d)$. Then there exists a subsequence (still denoted $u^\eps$) and a limit $u_0(t,\X,s,y)\in L^2((0,T)\times\R^d\times\Delta(\mathcal{A})\times\T^d)$ such that
\begin{align*}
u^\eps\sflow u_0
\end{align*}
in the sense of Definition \ref{def:abs:2scale-flow}.
\end{Thm}

To prove the above theorem, we will follow the method of Casado-D\'iaz and Gayte \cite{casado2002two}, as we would like to consider algebras which are not separable. To that end, we will need the following result from \cite{casado2002two}.

\begin{Thm}[Casado-D\'iaz and Gayte, Theorem 2.1. \cite{casado2002two}]\label{thm:abs:CDG-2002}
Let $X$ be a subspace (not necessarily closed) of a reflexive space $Y$
and let $f_n: X \to\R$ be a sequence of linear functionals (not necessarily continuous). Assume there exists a constant $C > 0$ which satisfies
\begin{align*}
\limsup_{n\to\infty}|f_n(x)|\le C \|x\|,\qquad \forall x\in X.
\end{align*}
Then there exists a subsequence $n_k$ and a functional $f\in Y'$ such that
\begin{align*}
\lim_{k\to\infty} f_{n_k}(x)=f(x),\qquad \forall x\in X.
\end{align*}
\end{Thm}

We will also need the following lemma, which is the main novel part of our proof.

\begin{Lem}\label{lem:abs:test-function}
Let $\mathcal{A}$ be an algebra w.m.v. and let $\Phi_\tau(x)$ be a flow satisfying Assumption \ref{ass:abs:flow}. Take $\varphi(t,\X,\tau,y)\in L^2((0,T)\times\R^d; \mathcal{A} \odot C(\T^d))$. Then 
\begin{align*}
\lim_{\eps\to0} 
\iint\limits_{(0,T)\times\R^d}
\left| \varphi\left(t, \Phi_{-t/\eps}(x), \frac{t}{\eps}, \frac{x}{\eps} \right) \right|^2
\, {\rm d}x\, {\rm d}t
= \iiiint\limits_{(0,T)\times\R^d\times\Delta(\mathcal{A})\times\T^d}
\left| \widehat{\varphi}(t,\X,s,y)\right|^2
\, {\rm d}y\, {\rm d}\beta(s)\, {\rm d}\X\, {\rm d}t.
\end{align*}
\end{Lem}

\begin{proof}
By density in $L^2((0,T)\times\R^d; \mathcal{A}\odot C(\T^d))$ of functions of the form
\begin{align*}
\sum_{j=1}^N
g_j(t) h_j(x) f_j(\tau) e^{i n_j\cdot y}
\end{align*}
where $g_j\in C^\infty(0,T)$, $h_j\in C^\infty_c(\R^d)$, $f_j\in\mathcal{A}$ and $n_j\in\Z^d$, and linearity it suffices to show that 
\begin{align*}
& \lim_{\eps\to0}
\iint\limits_{(0,T)\times\R^d}
g(t) h\left(\Phi_{-t/\eps}(x)\right) f\left(\frac{t}\eps\right) e^{in\cdot x/\eps}
\, {\rm d}x\, {\rm d}t
\\
& =
\left(
\int\limits_0^T g(t)
\, {\rm d}t
\right)
\left(
\int\limits_{\R^d}
h(\X)
\, {\rm d}\X
\right)
\left(
\int\limits_{\Delta(\mathcal{A})}
\widehat{f}(s)
\, {\rm d}\beta(s)
\right)
1_{n=0}
\end{align*}
for $g,h,f,n$ in the spaces above, and $1_{n=0}$ is one when $n=0$ and zero otherwise.

We first consider $n=0$, in which case the only $x$ dependence of the integrand is through $h$. By Fubini's theorem we may do the $x$ integration first. By the coordinate change: $\X=\Phi_{-t/\eps}(x)$, which has determinant 1 by the Assumption \ref{ass:abs:flow}.(iv), we have
\begin{align*}
&\iint\limits_{(0,T)\times\R^d}
g(t) h\left(\Phi_{-t/\eps}(x)\right) f\left(\frac{t}\eps\right) 
\, {\rm d}x\, {\rm d}t
\\
&=
\int\limits_0^T
g(t) f\left(\frac{t}\eps\right) 
\left(
\int\limits_{\R^d}
h(\X)
\, {\rm d}\X
\right)
\, {\rm d}t
=
\left(
\int\limits_{\R^d}
h(\X)
\, {\rm d}\X
\right)
\left(
\int\limits_0^T
g(t) f\left(\frac{t}\eps\right)
\, {\rm d} t
\right),
\end{align*}
so it suffices to show that the last integral converges to the required limit. By the definition of the mean value operator, $f(\cdot/\eps)$ converges $L^\infty(\R)$-weak* to $M(f)$. As $g\in L^1(0,T)$ this completes the proof for $n=0$, noting the identification of $M$ with $\beta$ (Proposition \ref{prop:abs:beta}).

Now suppose that $n\ne0$. As in the $n=0$ case we perform the $x$ integration first with $t$ fixed, but this time we do not change coordinates. Define the (formally) self-adjoint differential operator $L_n=-in\cdot\nabla_x$. Then we have the relation
\begin{align*}
e^{in\cdot x/\eps}
=
\frac{\eps}{|n|^2}L_n(e^{in\cdot x/\eps}).
\end{align*}
Substituting this into the $x$ integral and integrating by parts yields
\begin{align*}
\int\limits_{\R^d}
h\left(\Phi_{-t/\eps}(x)\right)
e^{in\cdot x/\eps}
\, {\rm d} x
&=
\frac{\eps}{|n|^2}
\int\limits_{\R^d}
h\left(\Phi_{-t/\eps}(x)\right)
L_n(e^{in\cdot x/\eps})
\, {\rm d} x
\\
&=
\frac{\eps}{|n|^2}
\int\limits_{\R^d}
e^{in\cdot x/\eps}
L_n
\left(
h
\left(
\Phi_{-t/\eps}(x)
\right)
\right)
\, {\rm d} x
\\
&=
\frac{-i\eps}{|n|^2}
\int\limits_{\R^d}
e^{in\cdot x/\eps} n\cdot 
{}^\top\!\!\, \widetilde{J}
\left(
\frac{t}{\eps},\Phi_{-t/\eps}(x)
\right)
\nabla_\XX h(\Phi_{-t/\eps}(x))
\, {\rm d}x.
\end{align*}
Consider the integrand on the last line. By assumption $h$ is smooth with compact support, $K$ say, so by Assumption \ref{ass:abs:flow}.(iii) we may estimate
\begin{align*}
\left|
\int\limits_{\R^d}
{}^\top\!\!\, \widetilde{J}(t/\eps,\Phi_{-t/\eps}(x))\nabla_\XX h(\Phi_{-t/\eps}(x))e^{in\cdot x/\eps}
\, {\rm d} x
\right|
\le C_K \|\nabla h\|_{L^\infty(\R^d)}|\Phi_{-t/\eps}^{-1}(K)|,
\end{align*}
where $|\Phi_{-t/\eps}^{-1}(K)|$ is the Lebesgue measure of the set inside the modulus sign. As $\Phi$ is volume preserving (Assumption \ref{ass:abs:flow}.(iv)) this is equal to the Lebesgue measure of $K$ and is finite. Therefore, the $x$ integral has the bound
\begin{align*}
\left|
\int\limits_{\R^d}
h(\Phi_{-t/\eps}(x))e^{in\cdot x/\eps}
\, {\rm d} x
\right|
\le C\eps
\end{align*}
for some constant $C$. Using this bound in the full $t,x$ integral yields
\begin{align*}
&
\left|\, \, \,
\iint\limits_{(0,T)\times\R^d}
g(t) h\left(\Phi_{-t/\eps}(x)\right) f\left(\frac{t}\eps\right) e^{in\cdot x/\eps}
\, {\rm d}x\, {\rm d}t
\right|
\\
&
\qquad
\le C\eps
\int\limits_0^T|g(t)||f(t/\eps)|
\, {\rm d} t
\le C\eps 
\|f\|_{L^\infty(\R)}
\|g\|_{L^1([0,T])}.
\end{align*}
This completes the $n\not=0$ case and the proof of the lemma.
\end{proof}
\begin{Rem}
It is evident from the above proof that the uniform bound upon the Jacobian (Assumption \ref{ass:abs:flow}.(iii)) is needed only for test functions that depend upon the fast spatial variable $y$. As a consequence, an analogous compactness result for convergence against test functions depending only upon $(t,x,\tau)$ can be obtained without this assumption (see Remark \ref{rem:abs:test-fn-no-space-oscillations}). However, Assumption \ref{ass:abs:flow}.(iii) is needed to identify the $\Sigma$-$\Phi_\tau$ limit of gradient sequences (Proposition \ref{prop:abs:corrector} below).
\end{Rem}

The weak $\Sigma$-convergence along flows is not limited to bounded sequences in $L^2$. Our main result, Theorem \ref{Thm:abs:compactness}, generalises straightaway to bounded sequences in $L^p$ with $1<p<+\infty$.

We are now ready to prove the compactness result.

\begin{proof}[Proof of Theorem \ref{Thm:abs:compactness}]
Let $Y=L^2((0,T)\times\R^d\times\Delta(\mathcal{A})\times\T^d)$ and $X$ be the vector subspace of Gelfand transforms of functions in $L^2((0,T)\times\R^d;\mathcal{A}\odot C(\T^d))$. Now define the linear functionals $F^\eps:X\subset Y\to\R$, by 
\begin{align*}
F^\eps(\widehat{\varphi})=
\iint\limits_{(0,T)\times\R^d}
u^\eps(t,x)
\varphi
\left(t,\Phi_{-t/\eps}(x),\frac t\eps,\frac x\eps\right)
\, {\rm d} x\, {\rm d} t,
\qquad
\widehat{\varphi}\in X.
\end{align*}
By the Cauchy-Schwarz inequality, the $L^2$ boundedness of $\{u^\eps\}$ and Lemma \ref{lem:abs:test-function} we have
\begin{align*}
|F^\eps(\widehat{\varphi})|
&
\le \left(\sup_\eps \|u^\eps\|_{L^2((0,T)\times \R^d)}\right)
\left\|\varphi\Big(t,\Phi_{-t/\eps}(x),\frac t\eps,\frac x\eps\Big)\right\|_{L^2((0,T)\times \R^d)}
\\
&
\le C \|\widehat{\varphi}(t,\X,y,s)\|_{L^2((0,T)\times \R^d\times \T^d\times \Delta(\mathcal{A}))}.
\end{align*}
By Theorem \ref{thm:abs:CDG-2002}, we may pass to a subsequence (still indexed by $\eps$) for which
\begin{align*}
F^\eps(\widehat{\varphi})
\to 
F(\widehat{\varphi})
\text{ as }\eps\to0,
\qquad 
\forall\varphi\in L^2((0,T)\times\R^d;\mathcal{A}\odot C(\T^d))
\end{align*}
where $F\in Y'$. Note that $Y=L^2((0,T)\times \R^d\times \Delta(\mathcal{A})\times\T^d)$ is a Hilbert space, (but is in general non-separable). Therefore, by the Riesz representation theorem, $F$ is represented by
\begin{align*}
F(\widehat{\varphi})
=
\iiiint\limits_{(0,T)\times\R^d\times\Delta(\mathcal{A})\times\T^d}
u_0(t,\X,s,y)\widehat{\varphi}(t,\X,s,y)
\, {\rm d} y\, {\rm d}\beta(s)\, {\rm d} \X\, {\rm d} t
\end{align*}
for some $u_0\in L^2((0,T)\times \R^d\times \T^d\times\Delta(\mathcal{A}))$, which is the desired limit.
\end{proof}

As is classical in the theory of two-scale convergence, we have the following result shedding some light on the product of two sequences that converge in the sense of $\Sigma$-convergence along flows.
\begin{Thm}[Limit of the product]\label{thm:abs:product-two-seq}
Let $u^\eps$ and $v^\eps$ be two families in $L^2((0,T)\times\R^d)$ such that
\begin{align*}
u^\eps \sflow u_0(t,\X,s,y);
\qquad
v^\eps \sflow v_0(t,\X,s,y).
\end{align*}
Assume further that
\begin{align*}
\lim_{\eps\to0}
\|u^\eps\|_{L^2((0,T)\times\R^d)}
=
\|u_0\|_{L^2((0,T)\times\R^d\times\Delta(\mathcal{A})\times\T^d)}.
\end{align*}
Then, we have
\begin{align*}
u^\eps(t,x)\, v^\eps(t,x)
\rightharpoonup 
\iint\limits_{\Delta(\mathcal{A})\times\T^d}
u_0(t,\X,s,y)\, v_0(t,\X,s,y)
\, {\rm d}\beta(s)
\, {\rm d}y
\end{align*}
in the sense of distributions.
\end{Thm}
The proof of Theorem \ref{thm:abs:product-two-seq} is by a density argument. These arguments are similar to the ones found in \cite{allaire1992homogenization} (see p.1488 in \cite{allaire1992homogenization} to be precise). As the proof can be given mutatis mutatndis, we skip the proof of Theorem \ref{thm:abs:product-two-seq}.

Next, we recall the notion of \emph{admissible test functions} given by M. Radu \cite{radu1992homogenization} in the context of two-scale convergence:
\begin{Defi}\label{defn:abs:admissible-2-scale}
Let $\varphi\in L^2(\Omega\times\T^d)$ be a function that can be approximated by a sequence of functions $\varphi_n\in C^\infty(\Omega;C^\infty(\T^d))$ such that for $n\to\infty$:
\begin{align*}
& 
\bullet \quad \left\|\varphi_n - \varphi\right\|_{L^2(\Omega\times\T^d)} \to 0.
\\
&
\bullet \quad \sup_{\eps>0}\left\|(\varphi_n - \varphi)\left(x,\frac{x}{\eps}\right) \right\|_{L^2(\Omega)} \to 0.
\end{align*}
Then $\varphi$ is said to be an admissible test function.
\end{Defi}
Inspired by the above definition, we introduce the notion of \emph{admissible test functions} suitable for the notion of weak $\Sigma$-convergence along flows.
\begin{Defi}[Admissible test functions]\label{defn:abs:admissible-test-fn}
A function $\psi(t,x,\tau, y)$ which is periodic in the $y$ variable and belongs to a certain algebra w.m.v. $\mathcal{A}$ in the $\tau$ variable is said to be an admissible test function if it can be approximated by a sequence of functions $\psi_n(t,x,\tau,y)\in C((0,T)\times\R^d;\mathcal{A}\odot C(\T^d))$ such that for $n\to\infty$:
\begin{align*}
&
\bullet \quad
\left\|
\widehat{\psi}
-
\widehat{\psi}_n
\right\|_{L^2((0,T)\times\R^d\times\Delta(\mathcal{A})\times\T^d)}
\to 0.
\\[0.3 cm]
&
\bullet \quad
\sup_{\eps>0}
\left\|
(\psi - \psi_n)
\left(t, \Phi_{-t/\eps}(x), \frac{t}{\eps}, \frac{x}{\eps}\right)
\right\|_{L^2((0,T)\times\R^d)}
\to 0.
\end{align*}
\end{Defi}
The following result says that having coefficients that are `admissible' in the sense of Definition \ref{defn:abs:admissible-test-fn} enables us to pass to the limit in the product sequence.
\begin{Lem}\label{lem:abs:admissible-coeff}
Let $\mathcal{A}$ be an algebra w.m.v. and $\Phi_\tau(x)$ be a flow satisfying Assumption \ref{ass:abs:flow}. Let the family $u^\eps(t,x)\subset L^2((0,T)\times\R^d)$ be such that
\begin{align*}
u^\eps \sflow u_0(t,\X,s,y).
\end{align*}
Finally, let $a(t,x,\tau,y)$ be admissible in the sense of Definition \ref{defn:abs:admissible-test-fn}. Then, for any smooth test function $\psi(t,x,\tau,y)$ which is periodic in the $y$ variable and which belongs to $\mathcal{A}$ as a function of the $\tau$ variable, we have
\begin{align*}
\lim_{\eps\to0}
& \iint\limits_{(0,T)\times\R^d}
u^\eps(t,x)
a
\left(
t, \Phi_{-t/\eps}(x), \frac{t}{\eps}, \frac{x}{\eps}
\right)
\psi
\left(
t, \Phi_{-t/\eps}(x), \frac{t}{\eps}, \frac{x}{\eps}
\right)
\, {\rm d}x\, {\rm d}t
\\
& =
\iiiint\limits_{(0,T)\times\R^d\times\Delta(\mathcal{A})\times\T^d}
u_0(t,\X,s,y)
\widehat{a}
(t,\X,s,y)
\widehat{\psi}
(t,\X,s,y)
\, {\rm d}y
\, {\rm d}\beta(s)
\, {\rm d}\XX
\, {\rm d}t
\end{align*}
\end{Lem}
The proof of Lemma \ref{lem:abs:admissible-coeff} is by a density argument (this is inherent in the definition of admissibility). These arguments are similar to the ones found in \cite{radu1992homogenization} (see p.6 in \cite{radu1992homogenization} to be precise). As the proof can be given mutatis mutatndis, we skip the details.

\subsection{Additional bounds on derivatives}\label{ssec:abs:additional-bounds}

We first establish conditions under which the $\Sigma$-$\Phi_\tau$ limit does not depend upon $y$. The following result has the flavour of standard two-scale convergence (see e.g. \cite{nguetseng1989general, allaire1992homogenization}), where gradient bounds imply that the two-scale limit is independent of the fast spatial variable. Here the proof is slightly complicated by the flow $\Phi_\tau$, but is otherwise the same.

\begin{Prop}\label{prop:abs:independent-of-y}
Let $\mathcal{A}$ be an algebra w.m.v., $\Phi_\tau$ be a flow satisfying the Assumption \ref{ass:abs:flow}, and $u^\eps\sflow u_0$ in the sense of Definition \ref{def:abs:2scale-flow}. Further if 
\begin{align*}
\sup_{\eps>0}\|\nabla u^\eps\|_{L^2((0,T)\times\R^d)}<\infty,
\end{align*}
then the weak $\Sigma$-$\Phi_\tau$ limit $u_0$ does not depend on the $y$ variable, i.e. $u_0(t,\X,s,y)\equiv u_0(t,\X,s)$.
\end{Prop}

\begin{proof}
Let $\Psi(t,\X,\tau,y)\in [C_c^1((0,T)\times \R^d\times\T^d;\mathcal{A})]^d$, then by the uniform bound on $\nabla u^\eps$ in $L^2$ and Lemma \ref{lem:abs:test-function} we have
\begin{align}\label{eq:abs:corrector1}
\sup_{\eps>0}
\iint\limits_{(0,T)\times \R^d}
\nabla u^\eps(t,x)\cdot\Psi\left(t,\Phi_{-t/\eps}(x),\frac{t}{\eps},\frac{x}{\eps}\right)
\, {\rm d} x\, {\rm d} t
\le C < \infty 
\end{align}
for some constant $C$ depending on $\Psi$. By integration by parts, the integral on the left hand side is equal to
\begin{align*}
-\frac1\eps
\iint\limits_{(0,T)\times\R^d}
&
u^\eps(t,x)\nabla_y\cdot\Psi
\left(t,\Phi_{-t/\eps}(x),\frac{t}{\eps},\frac{x}{\eps}\right)
\, {\rm d} x\, {\rm d} t
+
\mathcal{O}(1)
\end{align*}
where the order $1$ term comes from the gradient hitting $\Phi_{-t/\eps}(x)$ which are bounded due to Assumption \ref{ass:abs:flow}.(iii) on the Jacobian matrix associated with the flow. Multiplying this by $\eps$, using the convergence $u^\eps\sflow u_0$ and comparing to the bound \eqref{eq:abs:corrector1}, we have
\begin{align*}
\iiiint\limits_{(0,T)\times\R^d\times\Delta(\mathcal{A})\times \T^d}
u_0(t,\X,s,y)\nabla_y\cdot \widehat{\Psi}(t,\X,s,y)
\, {\rm d}\beta(s)\, {\rm d} y\, {\rm d} \X\, {\rm d} t
= 0.
\end{align*}
Noting that, by linearity and as it acts in a different variable, the Gelfand transform commutes with $\nabla_y$, we deduce that $u_0$ is orthogonal (in the $L^2(\T^d)$ sense) to all $y$-divergences and is hence independent of $y$.
\end{proof}

To obtain the $\Sigma$-$\Phi_\tau$ limit of the gradient sequence $\nabla u^\eps$, we require that the Jacobian matrix associated the flow lie in an algebra w.m.v.

\begin{Prop}[Two-scale limit for the gradient sequence]\label{prop:abs:corrector}
Let $\mathcal{A}$ be an algebra w.m.v., $\Phi_\tau$ be a flow satisfying Assumption \ref{ass:abs:flow} and let $J(\tau,\Phi_\tau(\X))\in C(\R^d;\mathcal{A})$. Suppose that
\begin{align*}
u^\eps\sflow u_0\text{\qquad and \qquad} \qquad \nabla u^\eps\sflow v_0\qquad \text{ as }\eps\to0
\end{align*}
for a sequence $u^\eps$ in the sense of Definition \ref{def:abs:2scale-flow}. Then we have
\begin{align*}
v_0={}^\top\!\widehat{\widetilde{J}}(s,\X)\nabla_\XX u_0+\nabla_y u_1
\end{align*}
for some $u_1(t,\X,s,y)\in L^2((0,T)\times\R^d\times \Delta(\mathcal{A});H^1(\T^d)).$
\end{Prop}

\begin{Rem}
The above result differs from the classical result for two-scale convergence (see e.g. \cite{allaire1992homogenization}) and two-scale convergence with constant drift \cite{maruvsic2005homogenization, allaire2008periodic}, in the presence of the Jacobian matrix of the flow, which depends on the fast time variable, in the limit. If the flow $\Phi_\tau$ is taken to be a constant drift flow $\Phi_\tau(x)=x+{\bf b}^*\tau$ then the Jacobian is the identity matrix.
\end{Rem}

\begin{proof}[Proof of Proposition \ref{prop:abs:corrector}]
Note that $u_0$ is independent of $y$ by Proposition \ref{prop:abs:independent-of-y}. We test against $\Psi(t,\X,\tau,y)\in [C^1_c(\R_+\times\R_d\times\T^d;\mathcal{A})]^d$ which satisfy $\nabla_y\cdot \Psi=0$. By integration by parts, we obtain
\begin{align*}
&
\iint\limits_{(0,T)\times\R^d}
\nabla u^\eps(t,x) \cdot \Psi \left(t, \Phi_{-t/\eps}(x),  \frac{t}{\eps},\frac{x}{\eps}\right)
\, {\rm d}x\, {\rm d}t
\\
&
= -\iint\limits_{(0,T)\times\R^d}
u^\eps(t,x)
\nabla_x
\cdot
\left(
\Psi
\left(
t, \Phi_{-t/\eps}(x),  \frac{t}{\eps},\frac{x}{\eps}
\right)
\right)
\, {\rm d}x\, {\rm d}t
\\
&=
-\frac1\eps
\iint\limits_{(0,T)\times\R^d}
u^\eps(t,x) 
\nabla_y
\cdot
\Psi
\left(
t, \Phi_{-t/\eps}(x),  \frac{t}{\eps},\frac{x}{\eps}
\right)
\, {\rm d}x\, {\rm d} t
\\
&
\quad  -\iint\limits_{(0,T)\times\R^d}
u^\eps(t,x)
\sum_{i=1}^d
\left(
{}^\top\!\! J
\left(\frac{t}{\eps},x
\right)
\nabla_\XX\Psi_i
\left(
t, \Phi_{-t/\eps}(x),  \frac{t}{\eps},\frac{x}{\eps}
\right)
\right)_i
\, {\rm d}x\, {\rm d}t
\\
&
= -\iint\limits_{(0,T)\times\R^d}
u^\eps(t,x)
\sum_{i=1}^d
\left(
{}^\top\!\!\, \widetilde{J}
\left(
\frac{t}{\eps}, \Phi_{-t/\eps}(x)
\right)
\nabla_\XX\Psi_i
\left(
t, \Phi_{-t/\eps}(x),  \frac{t}{\eps},\frac{x}{\eps}
\right)
\right)_i
\, {\rm d}x\, {\rm d}t.
\end{align*}
By the convergences  $\nabla u^\eps\sflow v_0$ and $u^\eps\sflow u_0$ we may pass to the limit in the first and last lines respectively to obtain
\begin{align*}
& \iiint\limits_{(0,T)\times\R^d\times \Delta(\mathcal{A})\times\T^d}
v_0(t,\X,s,y)\cdot\widehat{\Psi}(t,\X,s,y)
\, {\rm d}\beta(s)\, {\rm d} y\, {\rm d} \X\, {\rm d} t
\\
& = - \iiint\limits_{(0,T)\times\R^d\times \Delta(\mathcal{A})\times \T^d}
u_0(t,\X)
\sum_{i=1}^d
\left(
\widehat{{}^\top\!\!\, \widetilde{J}}(s,\X)
\widehat{\nabla_\XX\Psi_i}(t,\X,s, y)
\right)_i
\, {\rm d}\beta(s)\, {\rm d} y\, {\rm d} \X\, {\rm d} t
\end{align*}
The Gelfand transform is with regard to the $s$-variable. Hence we have the commutation: $\widehat{\nabla_\XX\Psi_i} = \nabla_\XX\widehat{\Psi_i}$. This observation and an integration by parts in the $\X$-variable yields
\begin{align*}
0=
\iiint\limits_{(0,T)\times\R^d\times\Delta(\mathcal{A})\times\T^d}
\left(v_0(t,\X,s,y)
-
\widehat{{}^\top\!\!\, \widetilde{J}}(s,\X)\nabla_\XX u_0(t,\X)\right) \cdot\widehat{\Psi}(t,\X,s,y)
\, {\rm d}\beta(s)\, {\rm d} y\, {\rm d} \X\, {\rm d} t.
\end{align*}
Thus the bracketed expression in the above integrand is orthogonal (in the $L^2(\T^d)$ sense) to $y$-divergence free vector fields, and is hence equal to the $y$-gradient of some function $u_1$. Basic Fourier analysis in $\T^d$ tells us that $u_1$ is bounded in $L^2(\R_+\times \R^d\times \Delta(\mathcal{A});H^1(\T^d))$. This completes the proof of the Proposition.
\end{proof}

\section{Homogenization Result}\label{sec:homog_result}

This section is dedicated to the rigorous derivation of the homogenized equation for \eqref{eq:heu:CD}-\eqref{eq:heu:IV} using the $\Sigma$-convergence along flows developed in Section \ref{sec:abs}.

\subsection{Qualitative analysis}\label{sec:QA}
The compactness results (Theorem \ref{Thm:abs:compactness}, Proposition \ref{prop:abs:corrector}) of previous section demand uniform (with respect to $\eps$) estimates on the solution family $\{u^\eps(t,x)\}$ and on the family of derivatives (in space) of the solution family $\{\nabla u^\eps(t,x)\}$.

\begin{Lem}\label{lem:QA:apriori-estimates}
Suppose the fluid field ${\bf b}(x,y)\in L^\infty(\R^d\times\T^d;\R^d)$ is incompressible in both $x$ and $y$ variables, i.e. satisfying \eqref{eq:heu:null-divergence-fluid-field}. Suppose the molecular diffusion tensor ${\bf D}(x,y)\in L^\infty(\R^d\times\T^d;\R^{d\times d})$ is uniformly coercive, i.e. satisfying \eqref{eq:heu:coercive-diffusion-matrix}. Suppose the initial data $u^{in}(x)\in L^2(\R^d)$. Then we have uniform (with respect to $\eps$) a priori estimates on the solutions to \eqref{eq:heu:CD}-\eqref{eq:heu:IV} given by
\begin{align}\label{eq:QA:apriori-estimates}
\|u^\eps\|_{L^\infty([0,T];L^2(\R^d))}
+ \|\nabla u^\eps\|_{L^2((0,T)\times\R^d)}
\le 
C\|u^{in}\|_{L^2(\R^d)},
\end{align}
for any arbitrary time $T>0$. The constant $C$ in \eqref{eq:QA:apriori-estimates} is independent of $\eps$ and the time instant $T$.
\end{Lem}
As the proof of the above lemma is very classical and follows the energy method, we shall skip the details. Now, we state the following result showing that our model problem \eqref{eq:heu:CD}-\eqref{eq:heu:IV} is well-posed.
\begin{Prop}\label{prop:QA:well-posed}
Suppose the fluid field ${\bf b}(x,y)\in L^\infty(\R^d\times\T^d;\R^d)$ is incompressible in both $x$ and $y$ variables, i.e. satisfying \eqref{eq:heu:null-divergence-fluid-field}. Suppose further that the diffusion tensor ${\bf D}(x,y)\in L^\infty(\R^d\times\T^d;\R^{d\times d})$ is uniformly coercive, i.e. satisfying \eqref{eq:heu:coercive-diffusion-matrix}.
Suppose the initial data $u^{in}\in L^2(\R^d)$. Then, for any fixed $\eps>0$, there exists a unique solution $u^\eps\in L^2((0,T);H^1(\R^d))\cap C^1((0,T);L^2(\R^d))$ to \eqref{eq:heu:CD}-\eqref{eq:heu:IV}.
\end{Prop}
For any fixed $\eps>0$, we can use the a priori bounds \eqref{eq:QA:apriori-estimates} and the Galerkin method to prove the above result. As this approach is very well-established (see Chapter 7 in \cite{evans1998partial} if necessary), we shall skip the proof of the above result as well.

\begin{Rem}\label{rem:QA:regularity-b-xy}
The regularity of the coefficients in \eqref{eq:heu:CD} considered in Lemma \ref{lem:QA:apriori-estimates} and Proposition \ref{prop:QA:well-posed} are quite weak. We shall impose some stronger regularity assumptions on the fluid field ${\bf b}(x,y)$ when we get to the homogenization result later in this section.
\end{Rem}

 We denote the difference between the mean-field and the locally periodic fluid field by
\begin{align}\label{eq:hom:mathcal-F-field}
\mathcal{F}(x,y):= \bar{\bf b}(x) - {\bf b}(x,y),
\qquad\qquad
\mbox{ for }(x,y)\in\R^d\times\T^d.
\end{align}

\begin{Rem}
We specialise the main result to the $3$ dimensional case. All the arguments to follow can be cast in the language of differential forms to generalize the theory to dimensions $d\ge2$, but to simplify presentation and increase accessibility of the proof, we leave this extension to the reader.
\end{Rem} 

The null-divergence assumption on the fluid field ${\bf b}(x,y)$ in the $y$ variable implies that $\mathcal{F}(x,y)$ is divergence free in the $y$-variable. Helmholtz decomposition of vector fields on the torus $\mathbb{T}^3$ yields the following result.
\begin{Lem}\label{lem:hom:helmotz}
There exists $\Upsilon(x,y)\in [L^2(\R^3;H^1(\T^3))]^3$ such that
\begin{align*}
\mathcal{F}(x,y) = \nabla_y\times\Upsilon(x,y);
\quad
\mbox{ with }
\int\limits_{\mathbb{T}^3}\Upsilon(x,y)\, {\rm d}y = 0.
\end{align*}
\end{Lem}
Under the scaling $y=x/\eps$, we have the chain rule:
\begin{align}\label{eq:hom:chain-rule-upsilon}
\nabla_x\times\left( \Upsilon\left(x,\frac{x}{\eps}\right)\right)
= \nabla_x\times\Upsilon\left(x,\frac{x}{\eps}\right) 
+ \frac{1}{\eps}\nabla_y\times\Upsilon\left(x,\frac{x}{\eps}\right).
\end{align}
Hence by Lemma \ref{lem:hom:helmotz}, we have
\begin{align}\label{eq:hom:mathcal-F=epsilon-curl-upsilon}
\mathcal{F}\left(x, \frac{x}{\eps}\right)
=
\eps\, \nabla_x\times\left( \Upsilon\left(x,\frac{x}{\eps}\right)\right)
- \eps\, \nabla_x\times\Upsilon\left(x,\frac{x}{\eps}\right).
\end{align}

\subsection{Assumptions}\label{ssec:hom:assumptions}

In this subsection we shall make precise the assumptions on the fluid field ${\bf b}(x,y)$, the mean field $\bar{{\bf b}}(x)$ and the Jacobian matrix $J(\tau,x)$ associated with the flow $\Phi_\tau$. Throughout, we will assume that $\mathcal{A}$ is a fixed given ergodic algebra w.m.v.. See Section \ref{sec:Exm} for further discussions on the assumptions made here.

\begin{Ass}\label{ass:hom:b-x-y}
The fluid field ${\bf b}(x,y)$ belongs to $C^1(\R^3\times\T^3;\R^3)$ and its flow-representation belongs to $\mathcal{A}$ as follows: 
\begin{align*}
\widetilde{{\bf b}}(\tau,x,y) = {\bf b}(\Phi_\tau(x), y) \in [C^1(\R^3\times\T^3;\mathcal{A})]^3.
\end{align*}
\end{Ass}

\begin{Rem}\label{rem:hom:regularity-algebra-mean-field}
The mean-field $\bar{\bf b}(x)$ is nothing but the $y$-average of the fluid field ${\bf b}(x,y)$. The regularity hypothesis in Assumption \ref{ass:hom:b-x-y} implies that $\bar{\bf b}(x)\in C^1(\R^3;\R^3)$. Furthermore, the linearity of $\mathcal{A}$ implies that the flow-representation of the mean-field belongs to $\mathcal{A}$ as follows: 
\begin{align*}
\widetilde{\bar{{\bf b}}}(\tau,x) = \bar{\bf b}(\Phi_\tau(x)) \in [C^1(\R^3;\mathcal{A})]^3.
\end{align*}
\end{Rem}

\begin{Ass}\label{ass:hom:curl-mathcal-F-algebra}
The field $\nabla_x \times \mathcal{F}(x,y)\in C(\R^3\times\T^3;\R^3)$ and its flow-representation belongs to $\mathcal{A}$ as follows:
\begin{align*}
\widetilde{\nabla_x \times \mathcal{F}}(\tau, x, y)
= \nabla_x \times \mathcal{F}\left(\Phi_\tau(x), y\right)
\in [C(\R^3\times\T^3;\mathcal{A})]^3.
\end{align*}
\end{Ass}

\begin{Rem}\label{rem:hom:flow-Upsilon-algebra}
Lemma \ref{lem:hom:helmotz} implies that the flow-representation of $\Upsilon$ is given by a convolution in the $y$-variable of a Greens function and the flow-representation of $\mathcal{F}$. The linearity of $\mathcal{A}$ allows us deduce that the flow-representation of $\Upsilon(x,y)$ belongs to $\mathcal{A}$ as follows: 
\begin{align*}
\widetilde{\Upsilon}(\tau,x,y)
= \Upsilon(\Phi_\tau(x), y) 
\in [C^1(\R^3\times\T^3;\mathcal{A})]^3.
\end{align*}
\end{Rem}

\begin{Rem}\label{rem:hom:flow-curl-Upsilon-algebra}
Observe that $\zeta:=\nabla_x \times \Upsilon$ solves the equation $\nabla_y\times \zeta=\nabla_x\times\mathcal{F}$. Hence a similar argument as in Remark \ref{rem:hom:flow-Upsilon-algebra} and the hypothesis in Assumption \ref{ass:hom:curl-mathcal-F-algebra} implies that $\zeta$ belongs to $\mathcal{A}$ as follows:
\begin{align*}
\widetilde{\zeta}(\tau,x,y)
= \widetilde{\nabla_x \times \Upsilon}(\tau,x,y)
= \nabla_x\times \Upsilon (\Phi_\tau(x), y)
\in [C(\R^3\times\T^3;\mathcal{A})]^3.
\end{align*}
\end{Rem}

\begin{Ass}\label{ass:hom:D-x-y-algebra}
The molecular diffusion matrix ${\bf D}(x,y)\in [L^\infty(\R^3;C(\T^3))]^{3\times3}$ and its flow-representation belongs to $\mathcal{A}$ as follows:
\begin{align*}
\widetilde{\bf D}(\tau,x,y) 
= {\bf D}(\Phi_\tau(x), y) 
\in [L^\infty(\R^3;C(\T^3)\odot\mathcal{A})]^{3\times 3}.
\end{align*}
\end{Ass}

\begin{Ass}\label{ass:hom:Jacobain-algebra}
The Jacobian matrix associated with the flow $\Phi_\tau$ has the regularity $J(\tau,x)\in [L^\infty(\R^3;\mathcal{A})]^{3\times3}$ and its flow-representation belongs to $\mathcal{A}$ as follows:
\begin{align*}
\widetilde{J}(\tau,x)
= J(\tau, \Phi_\tau(x))
\in [L^\infty(\R^3;\mathcal{A})]^{3\times3}.
\end{align*}
\end{Ass}
\begin{Rem}
Assumption \ref{ass:hom:D-x-y-algebra} is trivially satisfied if the molecular diffusion is purely periodic and bounded, i.e. ${\bf D}(x,y)\equiv {\bf D}(y)\in [L^\infty(\T^3)]^{3\times3} $. Similarly, that the flow representation of ${{\bf b}}$ belongs to $\mathcal{A}$ as in Assumption \ref{ass:hom:b-x-y} follows from Assumption \ref{ass:hom:Jacobain-algebra} if the fluid field ${\bf b}$ has the special form:
\begin{equation*}
{\bf b}(x,y)=\bar{{\bf b}}(x)+{\bf b}^1(y).
\end{equation*}
This may be seen from using Lemma \ref{lem:heu:basic-facts}.(iv) to write 
\begin{equation*}
\widetilde{{\bf b}}(\tau,\X,y)=\widetilde{\bar{{\bf b}}}(\tau,\X)+{\bf b}^1(y)=\left(\widetilde{J}(\tau,\X)\right)^{-1}\bar{{\bf b}}(\X)+{\bf b}^1(y).
\end{equation*}
\end{Rem}

\begin{Rem}\label{rem:hom:admissible-by-assumption}
The assumptions on the coefficients and the Jacobian matrix (Assumption \ref{ass:hom:b-x-y} - Assumption \ref{ass:hom:Jacobain-algebra}) ensure that they are admissible test functions in the sense of Definition \ref{defn:abs:admissible-test-fn}.
\end{Rem}

Next, we state our main result on the homogenization of the scaled convection-diffusion equation \eqref{eq:heu:CD}-\eqref{eq:heu:IV}.

\begin{Thm}\label{thm:hom:hom}
Let $\Phi_\tau$ be the flow associated with the three dimensional autonomous system \eqref{eq:heu:ode-mean-field}. Suppose $u^\eps(t,x)$ be the family of solutions associated with the scaled convection-diffusion equation \eqref{eq:heu:CD}-\eqref{eq:heu:IV}. Suppose $u_0(t,\X)$ be the $\Sigma$-$\Phi_\tau$ limit associated with the solution family. Suppose the fluid field ${\bf b}(x,y)$ satisfies the Assumption \ref{ass:hom:b-x-y}, 
the molecular diffusion ${\bf D}(x,y)$ satisfies the Assumption \ref{ass:hom:D-x-y-algebra},
the Jacobian matrix $J(\tau,x)$ satisfies the Assumption \ref{ass:hom:Jacobain-algebra}
and the field $\mathcal{F}(x,y)$ given by \eqref{eq:hom:mathcal-F-field} satisfies the Assumption \ref{ass:hom:curl-mathcal-F-algebra}.
Then the limit $u_0(t,\X)$ solves the weak formulation
\begin{equation}\label{eq:hom:wf:homogenized-equation}
\begin{aligned}
- \iint\limits_{(0,T)\times\R^3} u_0(t,\X) \frac{\partial \psi}{\partial t}(t,\X)
\, {\rm d}\X\, {\rm d}t
+ \iint\limits_{(0,T)\times\R^3}
\mathfrak{D}(\X)\nabla_\XX u_0(t,\X)\cdot \nabla_\XX \psi(t,\X)
\, {\rm d}\X\, {\rm d}t
\\
- \int\limits_{\R^3}
u^{in}(\X) \psi(0,\X)
\, {\rm d}\X
= 0,
\end{aligned}
\end{equation}
where the effective diffusion matrix $\mathfrak{D}(\X)$ is given by
\begin{align}\label{eq:hom:effective-diffusion}
\mathfrak{D}(\X) = 
\int\limits_{\Delta(\mathcal{A})}
\widehat{\!\! \widetilde{J}}(s,{\scriptstyle X})
\mathfrak{B}(s,\X)
{}^\top\!\!\, \widehat{\widetilde{J}}(s,\X)
\, {\rm d}\beta(s) 
\end{align} 
with the elements of the matrix $\mathfrak{B}(s,\X)$ given by
\begin{equation}\label{eq:hom:effective-diffusion-mathfrak-B}
\begin{aligned}
\mathfrak{B}_{ij}(s,\X)
& = 
\int\limits_{\T^3}
\widehat{\widetilde{{\bf D}}}(s,\X,y)
\left(
\nabla_y \widehat{\widetilde{\omega}}_j(s,\X,y)
+ {\bf e}_j
\right)
\cdot
\left(
\nabla_y \widehat{\widetilde{\omega}}_i(s,\X,y)
+ {\bf e}_i
\right)
\, {\rm d}y
\\
& + \int\limits_{\T^3}
\left( 
\widehat{\widetilde{\bf b}}(s, {\scriptstyle X}, y)
\cdot 
\nabla_y \widehat{\widetilde{\omega}}_i(s,\X,y)
\right)
\widehat{\widetilde{\omega}}_j(s,\X,y)
\, {\rm d}y
\\
& +
\int\limits_{\T^3}
\widehat{\widetilde{{\bf D}}}(s,\X,y)
\nabla_y \widehat{\widetilde{\omega}}_j(s,\X,y)
\cdot
{\bf e}_i
\, {\rm d}y
-
\int\limits_{\T^3}
\widehat{\widetilde{{\bf D}}}(s,\X,y)
\nabla_y \widehat{\widetilde{\omega}}_i(s,\X,y)
\cdot
{\bf e}_j
\, {\rm d}y,
\end{aligned}
\end{equation}
where the $\omega_i$ are the solutions to the cell problem \eqref{eq:heu:cell-problem}.
\end{Thm}

\begin{Rem}\label{rem:hom:homogenized-equation}
The weak formulation \eqref{eq:hom:wf:homogenized-equation} corresponds to solving the homogenized equation
\begin{align}\label{eq:hom:hom:diff}
\frac{\partial u_0}{\partial t}
- 
\nabla_\XX
\cdot 
\Big(
\mathfrak{D}(\X)
\nabla_\XX
u_0(t,\X)
\Big)
= 0
\qquad
\mbox{ in }]0,T[\times\R^d,
\end{align}
\begin{align}\label{eq:hom:hom:IV}
u_0(0,\X)
= 
u^{in}(\X)
\qquad \qquad \qquad
\mbox{ in }\R^d.
\end{align}
This equation is identical to that given in the formal homogenization result Proposition \ref{prop:heu:formal-result}, which may be seen from the identity \eqref{eq:abs:beta}. In particular, the diffusion coefficient $\mathfrak{D}$ may be computed using \eqref{eq:heu:effective-diffusion}-\eqref{eq:heu:effective-diffusion-mathfrak-B}, and the limits therein exist and are finite.
\end{Rem}

\begin{Rem}
Recall that for any function $f(x,y)$, the flow representation is given by
\[
\widetilde{f}(\tau,\X,y) = f\left( \Phi_\tau(\X), y\right) = f\left( x, y\right)\, \mbox{ for }\tau\in\R
\]
with the convention $\X:=\Phi_{-\tau}(x)$. Taking the Gelfand transform of a flow representation should be understood in the abstract as follows:
\[
\widehat{\widetilde{f}}(s,\X,y) = s\left(\widetilde{f}(\tau,\X,y)\right)\, \mbox{ for }s\in \Delta(\mathcal{A}).
\]
\end{Rem}

Theorem \ref{thm:hom:hom} asserts that a $\Sigma$-$\Phi_\tau$ limit of the solution family $u^\eps(t,x)$ solves the homogenized equation \eqref{eq:hom:hom:diff}-\eqref{eq:hom:hom:IV}. The rest of the section is devoted to proving this result. In Lemma \ref{lem:hom:2-scale-compactness}, we first prove that we can extract subsequences off the solution family $\{u^\eps\}$ and the gradient sequence $\{\nabla u^\eps\}$ such that the extracted subsequences admit $\Sigma$-$\Phi_\tau$ limits. Lemma \ref{lem:hom:2-scale-compactness} also proves that the $\Sigma$-$\Phi_\tau$ limit $u_0$ is independent of the $s$ variable.

Inspired by the structure of the $\Sigma$-$\Phi_\tau$ limits in Lemma \ref{lem:hom:2-scale-compactness}, we make a particular choice of the test functions (in Subsection \ref{ssec:hom:choice-test-functions}) in the weak formulation of the scaled convection-diffusion \eqref{eq:heu:CD}-\eqref{eq:heu:IV}.

As we need to pass to the limit in some singular terms in the weak formulation, we prove the limit behaviour of those singular terms in Lemma \ref{lem:hom:singular-terms-sigma-limit}. In Subsection \ref{ssec:hom:cell-problem}, we derive the cell problem. Finally, in Subsection \ref{ssec:hom:hom}, we give the proof of Theorem \ref{thm:hom:hom}.

\begin{Rem}\label{rem:hom:entire-seq}
Even though the $\Sigma$-$\Phi_\tau$ compactness results are up to extraction of a subsequence, the entire sequence $u^\eps$ does converge to the $\Sigma$-$\Phi_\tau$ limit $u_0$ as the homogenized equation is uniquely solvable (Proposition \ref{prop:heu:homo-eq-exist}).
\end{Rem}

\subsection{$\Sigma$-compactness along the flow $\Phi_\tau$}\label{ssec:hom:two-scale-compact}

\begin{Lem}\label{lem:hom:2-scale-compactness}
Let $u^\eps(t,x)$ be the family of solutions to \eqref{eq:heu:CD}-\eqref{eq:heu:IV}. Then, there exists a sub-sequence (still indexed by $\eps$) and limits $u_0(t,\X)\in L^2((0,T);H^1(\R^3))$, $u_1(t,\X,s,y)\in L^2((0,T)\times\R^3\times\Delta(\mathcal{A});H^1(\T^3))$ such that
\begin{align}
& u^\eps \sflow u_0(t,\X),\label{eq:hom:thm:2-scale-compact-u}\\
& \nabla u^\eps \sflow \widehat{{}^\top\!\!\, \widetilde{J}}(s,\X) \nabla_\XX u_0(t,\X) + \nabla_y u_1(t,\X,s,y).\label{eq:hom:thm:2-scale-compact-nabla-u}
\end{align}
\end{Lem}

\begin{proof}
The a priori bounds from Lemma \ref{lem:QA:apriori-estimates} give us the necessary uniform bounds (with respect to $\eps$) so that the result of Proposition \ref{prop:abs:corrector} implies the existence of $u_0(t,\X,s)$ and $u_1(t,\X,s,y)$ such that \eqref{eq:hom:thm:2-scale-compact-u} and \eqref{eq:hom:thm:2-scale-compact-nabla-u} hold. To prove that $u_0$ is independent of the $s$ variable, we shall consider the weak formulation of the $\eps$-problem \eqref{eq:heu:CD}-\eqref{eq:heu:IV} with the test function $\eps \varphi\left(t,\Phi_{-t/\eps}(x),\frac{t}{\eps}\right)$ such that $\varphi(T,\cdot,\cdot) = 0$ and $\varphi(t,x,\cdot), \frac{\partial \varphi}{\partial \tau}(t, x, \cdot) \in \mathcal{A}$. The weak formulation of interest shall be
\begin{equation}\label{eq:hom:wf-tau-independence}
\begin{aligned}
& - \eps \iint\limits_{(0,T)\times\R^3} 
u^\eps(t,x) \frac{\partial\varphi}{\partial t}\left( t,\Phi_{-t/\eps}(x),\frac{t}{\eps} \right)
\, {\rm d}x\, {\rm d}t
- \eps \int\limits_{\R^3}
u^{in}(x) \varphi(0,x,0)
\, {\rm d}x
\\
& + \iint\limits_{(0,T)\times\R^3} 
u^\eps(t,x)
\bar{{\bf b}}\left(\Phi_{-t/\eps}(x)\right)
\cdot
\nabla_\XX \varphi \left( t,\Phi_{-t/\eps}(x),\frac{t}{\eps} \right)
\, {\rm d}x\, {\rm d}t\\
& - \iint\limits_{(0,T)\times\R^3}
u^\eps(t,x) \frac{\partial\varphi}{\partial \tau}\left( t,\Phi_{-t/\eps}(x),\frac{t}{\eps} \right)
\, {\rm d}x\, {\rm d}t\\
& -\iint\limits_{(0,T)\times\R^3}
u^\eps(t,x)
{\bf b}\left(x,\frac{x}{\eps}\right)
\cdot 
{}^\top\!\!\, \widetilde{J}\left(\frac{t}{\eps}, \Phi_{-t/\eps}(x)\right)
\nabla_\XX \varphi\left( t,\Phi_{-t/\eps}(x),\frac{t}{\eps} \right)
\, {\rm d}x\, {\rm d}t\\
& + \eps \iint\limits_{(0,T)\times\R^3}
{\bf D}\left(x,\frac{x}{\eps}\right)
\nabla u^\eps(t,x)
\cdot 
{}^\top\!\!\, \widetilde{J}\left(\frac{t}{\eps}, \Phi_{-t/\eps}(x)\right)
\nabla_\XX \varphi\left( t,\Phi_{-t/\eps}(x),\frac{t}{\eps} \right)
\, {\rm d}x\, {\rm d}t
= 0.
\end{aligned}
\end{equation}
The first and second terms on the left hand side of the above expression are of $\mathcal{O}(\eps)$. The third and the fifth terms in the weak formulation \eqref{eq:hom:wf-tau-independence} together become, using \eqref{eq:hom:mathcal-F-field} and \eqref{eq:hom:mathcal-F=epsilon-curl-upsilon},
\begin{align*}
& \iint\limits_{(0,T)\times\R^3} 
u^\eps(t,x)
\Big(
\bar{{\bf b}}\left(x\right)
- {\bf b}\left(x,\frac{x}{\eps}\right)
\Big)
\cdot
{}^\top\!\!\, \widetilde{J}\left(\frac{t}{\eps}, \Phi_{-t/\eps}(x)\right)
\nabla_\XX \varphi \left( t,\Phi_{-t/\eps}(x),\frac{t}{\eps} \right)
\, {\rm d}x\, {\rm d}t
\\
& = \iint\limits_{(0,T)\times\R^3} 
u^\eps(t,x)
\mathcal{F}\left(x,\frac{x}{\eps}\right)
\cdot
{}^\top\!\!\, \widetilde{J}\left(\frac{t}{\eps}, \Phi_{-t/\eps}(x)\right)
\nabla_\XX \varphi \left( t,\Phi_{-t/\eps}(x),\frac{t}{\eps} \right)
\, {\rm d}x\, {\rm d}t
\\
& = \eps \iint\limits_{(0,T)\times\R^3} 
u^\eps(t,x)
\nabla_x\times\left( \Upsilon\left(x,\frac{x}{\eps}\right)\right)
\cdot
{}^\top\!\!\, \widetilde{J}\left(\frac{t}{\eps}, \Phi_{-t/\eps}(x)\right)
\nabla_\XX \varphi \left( t,\Phi_{-t/\eps}(x),\frac{t}{\eps} \right)
\, {\rm d}x\, {\rm d}t
\\
& - \eps \iint\limits_{(0,T)\times\R^3} 
u^\eps(t,x)
\nabla_x\times\Upsilon\left(x,\frac{x}{\eps}\right)
\cdot
{}^\top\!\!\, \widetilde{J}\left(\frac{t}{\eps}, \Phi_{-t/\eps}(x)\right)
\nabla_\XX \varphi \left( t,\Phi_{-t/\eps}(x),\frac{t}{\eps} \right)
\, {\rm d}x\, {\rm d}t
\end{align*}
\begin{align*}
& = \eps \iint\limits_{(0,T)\times\R^3} 
\Upsilon\left(x,\frac{x}{\eps}\right)
\cdot
\left( \nabla_x u^\eps(t,x) \times {}^\top \!\!\, \widetilde{J}\left(\frac{t}{\eps}, \Phi_{-t/\eps}(x)\right) \nabla_{\scriptscriptstyle X}\varphi\left( t,\Phi_{-t/\eps}(x),\frac{t}{\eps} \right)\right)
\, {\rm d}x\, {\rm d}t
\\
& + \eps \iint\limits_{(0,T)\times\R^3} 
\Upsilon\left(x,\frac{x}{\eps}\right)
\cdot
\left( u^\eps(t,x)\nabla_x \times {}^\top \!\!\, \widetilde{J}\left(\frac{t}{\eps}, \Phi_{-t/\eps}(x)\right) \nabla_{\scriptscriptstyle X}\varphi\left( t,\Phi_{-t/\eps}(x),\frac{t}{\eps} \right)\right)
\, {\rm d}x\, {\rm d}t
\\
& - \eps \iint\limits_{(0,T)\times\R^3} 
u^\eps(t,x)
\nabla_x\times\Upsilon\left(x,\frac{x}{\eps}\right)
\cdot
{}^\top\!\!\, \widetilde{J}\left(\frac{t}{\eps}, \Phi_{-t/\eps}(x)\right)
\nabla_\XX \varphi \left( t,\Phi_{-t/\eps}(x),\frac{t}{\eps} \right)
\, {\rm d}x\, {\rm d}t
\end{align*}
where we have used the Green's formula for the Curl operator. The second term on the far right hand side of the above expression vanishes because
\begin{align*}
\nabla_x 
\times 
{}^\top \!\!\, \widetilde{J}\left(\frac{t}{\eps}, \Phi_{-t/\eps}(x)\right) 
\nabla_{\scriptscriptstyle X}\varphi\left( t,\Phi_{-t/\eps}(x),\frac{t}{\eps} \right)
& = 
\nabla_x 
\times 
{}^\top \!\! J\left(\frac{t}{\eps}, x\right) 
\nabla_{\scriptscriptstyle X}\varphi\left( t,\Phi_{-t/\eps}(x),\frac{t}{\eps} \right)
\\
& =
\nabla_x 
\times 
\nabla_x
\left(
\varphi\left( t,\Phi_{-t/\eps}(x),\frac{t}{\eps} \right)
\right)
= 0,
\end{align*}
as the curl of a gradient is zero. (This is the reason that $\Upsilon$ was chosen in this particular manner).
In the rest of the terms, using the flow-representation, we have
\begin{align*}
& \eps \iint\limits_{(0,T)\times\R^3} 
\widetilde{\Upsilon}\left(\frac{t}{\eps}, \Phi_{-t/\eps}(x),\frac{x}{\eps}\right)
\cdot
\left( \nabla_x u^\eps(t,x) \times {}^\top \!\!\, \widetilde{J}\left(\frac{t}{\eps}, \Phi_{-t/\eps}(x)\right) \nabla_{\scriptscriptstyle X}\varphi\left( t,\Phi_{-t/\eps}(x),\frac{t}{\eps} \right)\right)
\, {\rm d}x\, {\rm d}t
\\
& - \eps \iint\limits_{(0,T)\times\R^3} 
u^\eps(t,x)
\widetilde{\nabla_x\times\Upsilon}\left(\frac{t}{\eps}, \Phi_{-t/\eps}(x),\frac{x}{\eps}\right)
\cdot
{}^\top\!\!\, \widetilde{J}\left(\frac{t}{\eps}, \Phi_{-t/\eps}(x)\right)
\nabla_\XX \varphi \left( t,\Phi_{-t/\eps}(x),\frac{t}{\eps} \right)
\, {\rm d}x\, {\rm d}t.
\end{align*}
These are of $\mathcal{O}(\eps)$. Using the flow-representation for the molecular diffusion ${\bf D}$, the final term on the left hand side of the weak formulation is also of $\mathcal{O}(\eps)$. Hence, passing to the limit as $\eps$ tends to zero in the weak formulation yields
\begin{align*}
\iiint\limits_{(0,T)\times\R^3\times\Delta(\mathcal{A})}
u_0(t,\X,s)\widehat{\frac{\partial \varphi}{\partial \tau}}(t,\X,s)
\, {\rm d}\beta(s)\, {\rm d}\XX\, {\rm d}t
= 0.
\end{align*}
Upon using Lemma \ref{lem:abs:mean-value-tau-derivative}, we deduce that the $u_0$ is independent of the $s$-variable.
\end{proof}

\subsection{Choice of test functions}\label{ssec:hom:choice-test-functions}

We consider the weak formulation of the convection-diffusion equation \eqref{eq:heu:CD}-\eqref{eq:heu:IV} with test function $\psi^\eps(t,x)$ such that $\psi^\eps(T,x)=0$:
\begin{equation}\label{eq:hom:weak-form}
\begin{array}{ll}
\displaystyle
\mathcal{I}_{time}
+ \mathcal{I}_{convect}
+ \mathcal{I}_{diffuse}
+ \mathcal{I}_{initial}
:= \\[0.3 cm]
\displaystyle 
- \iint\limits_{(0,T)\times\R^3} u^\eps(t,x) \frac{\partial\psi^\eps}{\partial t}(t,x)
\, {\rm d}x\, {\rm d}t
+ \frac{1}{\eps} \iint\limits_{(0,T)\times\R^3} {\bf b}\left(x,\frac x\eps\right)\cdot \nabla u^\eps(t,x) \psi^\eps(t,x)
\, {\rm d}x\, {\rm d}t\\[0.3 cm]
\displaystyle
+ \iint\limits_{(0,T)\times\R^3} {\bf D}\left(x,\frac x\eps\right)\nabla u^\eps(t,x) \cdot \nabla \psi^\eps(t,x)
\, {\rm d}x\, {\rm d}t
- \int\limits_{\R^3} u^{in}(x)\psi^\eps(0,x)
\, {\rm d}x 
= 0.
\end{array}
\end{equation}
The choice of the family of test functions $\psi^\eps(t,x)$ is as follows:
\begin{align}\label{eq:hom:test-fn-choice}
\psi^\eps(t,x) = 
\psi\left(t,\Phi_{-t/\eps}(x)\right)
+ \eps \psi_1\left(t,\Phi_{-t/\eps}(x), \frac{t}{\eps}, \frac{x}{\eps}\right)
\end{align}
for
\begin{equation}\label{eq:hom:test-fn-spaces}
\psi\in C^1((0,T)\times\R^3)\qquad\text{and}\qquad\psi_1\in C^1((0,T)\times\R^3;C^1(\T^3)\odot \mathcal{A})
\end{equation}
which are compactly supported in space. We shall treat term by term. To begin with, let us consider the term with the partial time derivative:
\begin{align*}
\mathcal{I}_{time} = 
&- \iint\limits_{(0,T)\times\R^3}
u^\eps(t,x) \frac{\partial \psi}{\partial t}\left(t,\Phi_{-t/\eps}(x)\right)
\, {\rm d}x\, {\rm d}t\\
&+ \frac{1}{\eps} \iint\limits_{(0,T)\times\R^3}
u^\eps(t,x) \bar{\bf b}(\Phi_{-t/\eps}(x))\cdot \nabla_{\scriptscriptstyle X}\psi\left(t,\Phi_{-t/\eps}(x)\right) 
\, {\rm d}x\, {\rm d}t\\
& - \iint\limits_{(0,T)\times\R^3}
u^\eps(t,x) \frac{\partial \psi_1}{\partial \tau} \left(t,\Phi_{-t/\eps}(x), \frac{t}{\eps}, \frac{x}{\eps}\right)
\, {\rm d}x\, {\rm d}t\\
& + \iint\limits_{(0,T)\times\R^3}
u^\eps(t,x)\bar{\bf b}(\Phi_{-t/\eps}(x))\cdot \nabla_{\scriptscriptstyle X}\psi_1 \left(t,\Phi_{-t/\eps}(x), \frac{t}{\eps}, \frac{x}{\eps}\right)
\, {\rm d}x\, {\rm d}t
+ \mathcal{O}(\eps).
\end{align*}
Now, for the convection term:
\begin{align*}
\mathcal{I}_{convect} =
& - \frac{1}{\eps} \iint\limits_{(0,T)\times\R^3}
u^\eps(t,x) {\bf b}\left(x,\frac{x}{\eps}\right)\cdot {}^\top \!\!\, \widetilde{J}\left(\frac{t}{\eps}, \Phi_{-t/\eps}(x)\right) \nabla_{\scriptscriptstyle X}\psi\left(t,\Phi_{-t/\eps}(x)\right)
\, {\rm d}x\, {\rm d}t\\
& + \iint\limits_{(0,T)\times\R^3}
{\bf b}\left(x,\frac{x}{\eps}\right)\cdot \nabla u^\eps(t,x) \psi_1\left(t,\Phi_{-t/\eps}(x), \frac{t}{\eps}, \frac{x}{\eps}\right)
\, {\rm d}x\, {\rm d}t
\end{align*}
Next, for the diffusion term:
\begin{align*}
\mathcal{I}_{diffuse} =
& \iint\limits_{(0,T)\times\R^3}
{\bf D}\left(x,\frac{x}{\eps}\right)\nabla u^\eps(t,x) \cdot {}^\top \!\!\, \widetilde{J}\left(\frac{t}{\eps}, \Phi_{-t/\eps}(x)\right) \nabla_{\scriptscriptstyle X}\psi\left(t,\Phi_{-t/\eps}(x)\right)
\, {\rm d}x\, {\rm d}t\\
& + \iint\limits_{(0,T)\times\R^3}
{\bf D}\left(x,\frac{x}{\eps}\right)\nabla u^\eps(t,x) \cdot \nabla_y \psi_1\left(t,\Phi_{-t/\eps}(x), \frac{t}{\eps}, \frac{x}{\eps}\right)
\, {\rm d}x\, {\rm d}t
+ \mathcal{O}(\eps).
\end{align*}
Finally, for the term involving initial data:
\begin{align*}
\mathcal{I}_{initial} =
- \int\limits_{\R^3} 
u^{in}(x) \psi(0,x)
\, {\rm d}x
+ \mathcal{O}(\eps).
\end{align*}
By using the flow-representation and Lemma \ref{lem:heu:basic-facts}.(iv) we notice that
\begin{align}\label{eq:hom:prop-cell-pb-trick-1}
\bar{\bf b}\left(\Phi_{-t/\eps}(x)\right)
= \widetilde{J}\left(\frac{t}{\eps}, \Phi_{-t/\eps}(x)\right) \bar{\bf b}(x)
= \widetilde{J}\left(\frac{t}{\eps}, \Phi_{-t/\eps}(x)\right) \widetilde{\bar{{\bf b}}}\left(\frac{t}{\eps}, \Phi_{-t/\eps}(x)\right).
\end{align}
Again using the flow-representation, we have
\begin{align}\label{eq:hom:prop-cell-pb-trick-2}
{\bf b}\left(x,\frac{x}{\eps}\right)
= \widetilde{{\bf b}}\left( \frac{t}{\eps}, \Phi_{-t/\eps}(x), \frac{x}{\eps}\right), \qquad {\bf D}(x,y)=\widetilde{{\bf D}}\left( \frac{t}{\eps}, \Phi_{-t/\eps}(x), \frac{x}{\eps}\right).
\end{align}
The observations \eqref{eq:hom:prop-cell-pb-trick-1}-\eqref{eq:hom:prop-cell-pb-trick-2}, combined with the $\Sigma$-compactness result along the flow $\Phi_\tau$ (Lemma \ref{lem:hom:2-scale-compactness}) will allow us to pass to the limit as $\eps\to0$ in all but two singular terms.
\subsection{Singular terms}\label{ssec:hom:singular-terms}
 We record below a result giving the limit of the singular terms in the weak formulation.

\begin{Lem}\label{lem:hom:singular-terms-sigma-limit}
Under Assumption \ref{ass:hom:curl-mathcal-F-algebra} on the field $\mathcal{F}(x,y)$ and for $\psi$ satisfying \eqref{eq:hom:test-fn-spaces}, we have
\begin{equation}\label{eq:hom:lem:sigular-terms-sigma-limit}
\begin{array}{cc}
\displaystyle
\lim_{\eps\to0}
\frac{1}{\eps} \iint\limits_{(0,T)\times\R^3}
u^\eps(t,x) \widetilde{J}\left(\frac{t}{\eps},\Phi_{-t/\eps}(x)\right) \left[ \bar{\bf b}(x) - {\bf b}\left(x,\frac{x}{\eps}\right) \right] \cdot \nabla_{\scriptscriptstyle X} \psi\left(t,\Phi_{-t/\eps}(x)\right)
\, {\rm d}x\, {\rm d}t\\[0.3 cm]
\displaystyle
= \iiiint\limits_{(0,T)\times\R^3\times\Delta(\mathcal{A})\times\mathbb{T}^3}
\left(
\widehat{\widetilde{\bar{\bf b}}}(s, {\scriptstyle X})
- \widehat{\widetilde{\bf b}}(s, {\scriptstyle X}, y)
\right)
\cdot
\left(
u_1(t,{\scriptstyle X},s,y)
{}^\top \!\!\, \widehat{\widetilde{J}}(s,{\scriptstyle X}) \nabla_{\scriptscriptstyle X}\widehat{\psi} (t,{\scriptstyle X})
\right)
\, {\rm d}y\, {\rm d}\beta(s)\, {\rm d}{\scriptstyle X}\, {\rm d}t.
\end{array}
\end{equation}
\end{Lem}

\begin{proof}
Consider the singular terms in the weak formulation:
\begin{align}\label{eq:hom:singular-terms}
\frac{1}{\eps} \iint\limits_{(0,T)\times\R^3}
u^\eps(t,x) \widetilde J\left(\frac{t}{\eps},\Phi_{-t/\eps}(x)\right) \left[ \bar{\bf b}(x) - {\bf b}\left(x,\frac{x}{\eps}\right) \right] \cdot \nabla_{\scriptscriptstyle X} \psi\left(t,\Phi_{-t/\eps}(x)\right)
\, {\rm d}x\, {\rm d}t.
\end{align}
Using the observation \eqref{eq:hom:mathcal-F=epsilon-curl-upsilon} on the field $\mathcal{F}\left(x,\frac{x}{\eps}\right)$, we rewrite the singular terms \eqref{eq:hom:singular-terms} successively as follows:
\begin{align*}
\iint\limits_{(0,T)\times\R^3}
\nabla_x\times\left( \Upsilon\left(x,\frac{x}{\eps}\right)\right) \cdot \left( u_\eps(t,x) {}^\top \!\!\, \widetilde{J}\left(\frac{t}{\eps}, \Phi_{-t/\eps}(x)\right) \nabla_{\scriptscriptstyle X}\psi\left(t,\Phi_{-t/\eps}(x)\right)\right)
\, {\rm d}x\, {\rm d}t\\
- \iint\limits_{(0,T)\times\R^3}
\nabla_x\times\Upsilon\left(x,\frac{x}{\eps}\right) \cdot \left( u_\eps(t,x) {}^\top \!\!\, \widetilde{J}\left(\frac{t}{\eps}, \Phi_{-t/\eps}(x)\right) \nabla_{\scriptscriptstyle X}\psi\left(t,\Phi_{-t/\eps}(x)\right)\right)
\, {\rm d}x\, {\rm d}t\\
= \iint\limits_{(0,T)\times\R^3}
\Upsilon\left(x,\frac{x}{\eps}\right)
\cdot
\left( \nabla_x u_\eps(t,x) \times {}^\top \!\!\, \widetilde{J}\left(\frac{t}{\eps}, \Phi_{-t/\eps}(x)\right) \nabla_{\scriptscriptstyle X}\psi\left(t,\Phi_{-t/\eps}(x)\right)\right)
\, {\rm d}x\, {\rm d}t\\
+ \iint\limits_{(0,T)\times\R^3}
\Upsilon\left(x,\frac{x}{\eps}\right)
\cdot
\left( u_\eps(t,x)\nabla_x \times \left( {}^\top \!\!\, \widetilde{J}\left(\frac{t}{\eps}, \Phi_{-t/\eps}(x)\right) \nabla_{\scriptscriptstyle X}\psi\left(t,\Phi_{-t/\eps}(x)\right)\right)\right)
\, {\rm d}x\, {\rm d}t\\
- \iint\limits_{(0,T)\times\R^3}
\nabla_x\times\Upsilon\left(x,\frac{x}{\eps}\right) \cdot \left( u_\eps(t,x) {}^\top \!\!\, \widetilde{J}\left(\frac{t}{\eps}, \Phi_{-t/\eps}(x)\right) \nabla_{\scriptscriptstyle X}\psi\left(t,\Phi_{-t/\eps}(x)\right)\right)
\, {\rm d}x\, {\rm d}t,
\end{align*}
where we have used the Green's formula for the curl operator. Note that the second term on the right hand side of the previous expression is zero because
\begin{align*}
{}^\top \!\!\, \widetilde{J}\left(\frac{t}{\eps}, \Phi_{-t/\eps}(x)\right) \nabla_{\scriptscriptstyle X}\psi\left(t,\Phi_{-t/\eps}(x)\right)
= \nabla_x \left( \psi\left(t,\Phi_{-t/\eps}(x)\right) \right)
\end{align*}
and because of the fact that curl of a gradient is zero. Next, using the flow-representation, the singular terms \eqref{eq:hom:singular-terms} simplify to the following expression:
\begin{align*}
\iint\limits_{(0,T)\times\R^3}
\widetilde{\Upsilon}\left( \frac{t}{\eps}, \Phi_{-t/\eps}(x), \frac{x}{\eps} \right)
\cdot
\left( \nabla_x u_\eps(t,x) \times {}^\top \!\!\, \widetilde{J}\left(\frac{t}{\eps}, \Phi_{-t/\eps}(x)\right) \nabla_{\scriptscriptstyle X}\psi\left(t,\Phi_{-t/\eps}(x)\right)\right)
\, {\rm d}x\, {\rm d}t\\
- \iint\limits_{(0,T)\times\R^3}
\widetilde{\nabla_x\times\Upsilon}\left( \frac{t}{\eps}, \Phi_{-t/\eps}(x), \frac{x}{\eps} \right)
\cdot 
\left( u_\eps(t,x) {}^\top \!\!\, \widetilde{J}\left(\frac{t}{\eps}, \Phi_{-t/\eps}(x)\right) \nabla_{\scriptscriptstyle X}\psi\left(t,\Phi_{-t/\eps}(x)\right)\right)
\, {\rm d}x\, {\rm d}t.
\end{align*}
Thanks to the Assumption \ref{ass:hom:curl-mathcal-F-algebra} (see Remark \ref{rem:hom:flow-Upsilon-algebra} and Remark \ref{rem:hom:flow-curl-Upsilon-algebra}), we can pass to the limit, as $\eps\to0$, in the previous expression using $\Sigma$-convergence along the flow $\Phi_\tau$ yielding
\begin{equation}\label{eq:hom:singular-terms-sigma-limit}
\begin{array}{cc}
\displaystyle
\iiiint\limits_{(0,T)\times\R^3\times\Delta(\mathcal{A})\times\mathbb{T}^3}
\widehat{\widetilde{\Upsilon}}(s, {\scriptstyle X}, y)
\cdot
\left(
\widehat{{}^\top \!\!\, \widetilde{J}}(s,{\scriptstyle X}) \nabla_{\scriptscriptstyle X} u_0(t,{\scriptstyle X})
\times
\widehat{{}^\top \!\!\, \widetilde{J}}(s,{\scriptstyle X}) \nabla_{\scriptscriptstyle X}\widehat{\psi} (t,{\scriptstyle X})
\right)
\, {\rm d}y\, {\rm d}\beta(s)\, {\rm d}{\scriptstyle X}\, {\rm d}t\\[0.3 cm]
\displaystyle
+ \iiiint\limits_{(0,T)\times\R^3\times\Delta(\mathcal{A})\times\mathbb{T}^3}
\widehat{\widetilde{\Upsilon}}(s, {\scriptstyle X}, y)
\cdot
\left(
\nabla_y u_1(t,{\scriptstyle X},s,y)
\times
\widehat{{}^\top \!\!\, \widetilde{J}}(s,{\scriptstyle X}) \nabla_{\scriptscriptstyle X}\widehat{\psi} (t,{\scriptstyle X})
\right)
\, {\rm d}y\, {\rm d}\beta(s)\, {\rm d}{\scriptstyle X}\, {\rm d}t\\[0.3 cm]
\displaystyle
- \iiiint\limits_{(0,T)\times\R^3\times\Delta(\mathcal{A})\times\mathbb{T}^3}
\widehat{\widetilde{\nabla_x\times\Upsilon}}(s, {\scriptstyle X}, y)
\cdot
\left(
u_0(t,{\scriptstyle X})
\widehat{{}^\top \!\!\, \widetilde{J}}(s,{\scriptstyle X}) \nabla_{\scriptscriptstyle X}\widehat{\psi} (t,{\scriptstyle X})
\right)
\, {\rm d}y\, {\rm d}\beta(s)\, {\rm d}{\scriptstyle X}\, {\rm d}t.
\end{array}
\end{equation}
By construction, $\Upsilon$ is of zero average in the $y$ variable (Lemma \ref{lem:hom:helmotz}). Hence the first term in \eqref{eq:hom:singular-terms-sigma-limit} vanishes. In the second term of \eqref{eq:hom:singular-terms-sigma-limit}, we use the Green's formula for the curl operator in $y$ variable leading to the following expression:
\begin{align*}
\iiiint\limits_{(0,T)\times\R^3\times\Delta(\mathcal{A})\times\mathbb{T}^3}
\nabla_y \times \widehat{\widetilde{\Upsilon}}(s, {\scriptstyle X}, y)
\cdot
\left(
u_1(t,{\scriptstyle X},s,y)
{}^\top \!\!\, \widehat{\widetilde{J}}(s,{\scriptstyle X}) \nabla_{\scriptscriptstyle X}\widehat{\psi} (t,{\scriptstyle X})
\right)
\, {\rm d}y\, {\rm d}\beta(s)\, {\rm d}{\scriptstyle X}\, {\rm d}t.
\end{align*}
Again by Lemma \ref{lem:hom:helmotz}, we have
\begin{align*}
\nabla_y \times \widehat{\widetilde{\Upsilon}}(s, {\scriptstyle X}, y)
= \widehat{\widetilde{\mathcal{F}}}(s, {\scriptstyle X}, y)
= \widehat{\widetilde{\bar{\bf b}}}(s, {\scriptstyle X})
- \widehat{\widetilde{\bf b}}(s, {\scriptstyle X}, y).
\end{align*}
Hence, the second term of \eqref{eq:hom:singular-terms-sigma-limit} is the same as
\begin{align}\label{eq:hom:singular-term-2nd-term-limit}
\iiiint\limits_{(0,T)\times\R^3\times\Delta(\mathcal{A})\times\mathbb{T}^3}
\left(
\widehat{\widetilde{\bar{\bf b}}}(s, {\scriptstyle X})
- \widehat{\widetilde{\bf b}}(s, {\scriptstyle X}, y)
\right)
\cdot
\left(
u_1(t,{\scriptstyle X},s,y)
\widehat{{}^\top \!\!\, \widetilde{J}}(s,{\scriptstyle X}) \nabla_{\scriptscriptstyle X}\widehat{\psi} (t,{\scriptstyle X})
\right)
\, {\rm d}y\, {\rm d}\beta(s)\, {\rm d}{\scriptstyle X}\, {\rm d}t.
\end{align}
Regarding the third term in \eqref{eq:hom:singular-terms-sigma-limit}, remark that
\begin{align*}
\widehat{\widetilde{\nabla_x\times\Upsilon}}(s, {\scriptstyle X}, y)
= \widehat{{}^\top \!\!\, \widetilde{J}}(s,{\scriptstyle X})
\left(
\nabla_{\scriptscriptstyle X} \times \widehat{\widetilde{\Upsilon}}(s, {\scriptstyle X}, y)
\right).
\end{align*}
Hence the third term in \eqref{eq:hom:singular-terms-sigma-limit} rewrites as (upon using Green's formula):
\begin{align*}
\iiiint\limits_{(0,T)\times\R^3\times\Delta(\mathcal{A})\times\mathbb{T}^3}
\widehat{\widetilde{\Upsilon}}(s, {\scriptstyle X}, y)
\cdot
\nabla_{\scriptscriptstyle X} \times
\left(
u_0(t,{\scriptstyle X})
\widehat{\widetilde{J}}(s,{\scriptstyle X})
\widehat{{}^\top \!\!\, \widetilde{J}}(s,{\scriptstyle X}) 
\nabla_{\scriptscriptstyle X}\widehat{\psi} (t,{\scriptstyle X})
\right)\, {\rm d}y\, {\rm d}\beta(s)\, {\rm d}{\scriptstyle X}\, {\rm d}t.
\end{align*}
Thanks again to the construction that $\Upsilon$ is of zero average in the $y$ variable (Lemma \ref{lem:hom:helmotz}), the above expression vanishes. Hence the only non-zero term in the limit expression is \eqref{eq:hom:singular-term-2nd-term-limit}. This is nothing but the limit in \eqref{eq:hom:lem:sigular-terms-sigma-limit}.
\end{proof}

\subsection{Cell problem}\label{ssec:hom:cell-problem}

Now, we record a result to give the limit equation for the weak formulation involving the test function $\psi_1$, i.e. to derive the cell problem.

\begin{Prop}\label{prop:hom:sigma-limit-cell-problem}
Let $\Phi_\tau$ be the flow associated with the autonomous system \eqref{eq:heu:ode-mean-field}. Under Assumption \ref{ass:hom:b-x-y} on the fluid field ${\bf b}(x,y)$, Assumption \ref{ass:hom:D-x-y-algebra} on the molecular diffusion matrix ${\bf D}(x,y)$ and Assumption \ref{ass:hom:Jacobain-algebra} on the Jacobian matrix $J(\tau,x)$, the $\Sigma$-$\Phi_\tau$ limit $u_1(t,\X,s,y)$ obtained in Lemma \ref{lem:hom:2-scale-compactness} can be written as
\begin{align}\label{eq:hom:u-1-seprable-variables}
u_1(t,\X,s,y)
= 
\widehat{\widetilde{\omega}}(s,\X,y)
\cdot
{}^\top\!\!\, \widehat{\widetilde{J}}(s,\X)\nabla_\XX u_0(t,\X),
\end{align}
where the components of $\omega(x,y)\in [L^\infty(\R^3;H^1(\T^3))]^d$ with $\int_{\T^3}\omega\,{\rm d}y=0$ solve the cell problem:
\begin{align}\label{eq:hom:prop:cell-probem}
{\bf b}(x,y)\cdot \left(\nabla_y \omega_i + {\bf e}_i\right)
- \nabla_y \cdot \left( {\bf D}(x,y) \left(\nabla_y \omega_i + {\bf e}_i\right) \right)
= \bar{{\bf b}}(x)\cdot {\bf e}_i\qquad \mbox{ in }\T^3,
\end{align}
for each $i\in\{1,2,3\}$, where $\{{\bf e}_i\}_{i=1}^3$ denote the canonical basis in $\R^3$ and $x$ is viewed as a parameter.
\end{Prop}

\begin{proof}
Taking $\psi\equiv0$ in the weak formulation \eqref{eq:hom:weak-form} and passing to the limit in the sense of $\Sigma$-convergence along the flow $\Phi_\tau$, we obtain
\begin{align*}
& - \iiiint\limits_{(0,T)\times\R^3\times\Delta(\mathcal{A})\times\mathbb{T}^3}
u_0(t, {\scriptstyle X})
\widehat{\frac{\partial \psi_1}{\partial \tau}}(t, {\scriptstyle X}, s, y)
\, {\rm d}y\, {\rm d}\beta(s)\, {\rm d}{\scriptscriptstyle X}\, {\rm d}t\\
& + \iiiint\limits_{(0,T)\times\R^3\times\Delta(\mathcal{A})\times\mathbb{T}^3}
u_0(t, {\scriptstyle X})
\left\{
\widehat{\widetilde{J}}(s, {\scriptstyle X})
\widehat{\widetilde{\bar{{\bf b}}}}(s, {\scriptstyle X})
\cdot
\nabla_{\scriptscriptstyle X}\widehat{\psi_1} (t, {\scriptstyle X}, s, y)
\right\}
\, {\rm d}y\, {\rm d}\beta(s)\, {\rm d}{\scriptstyle X}\, {\rm d}t\\
& + \iiiint\limits_{(0,T)\times\R^3\times\Delta(\mathcal{A})\times\mathbb{T}^3}
\left\{
\widehat{\widetilde{{\bf b}}}(s, {\scriptstyle X}, y)
\cdot
\widehat{{}^\top \!\!\, \widetilde{J}}(s,{\scriptstyle X}) \nabla_{\scriptscriptstyle X} u_0(t,{\scriptstyle X})
\right\}
\widehat{\psi_1} (t, {\scriptstyle X}, s, y)
\, {\rm d}y\, {\rm d}\beta(s)\, {\rm d}{\scriptstyle X}\, {\rm d}t
\end{align*}
\begin{align*}
& + \iiiint\limits_{(0,T)\times\R^3\times\Delta(\mathcal{A})\times\mathbb{T}^3}
\left\{
\widehat{\widetilde{{\bf b}}}(s, {\scriptstyle X}, y)
\cdot
\nabla_y u_1(t,{\scriptstyle X},s,y)
\right\}
\widehat{\psi_1} (t, {\scriptstyle X}, s, y)
\, {\rm d}y\, {\rm d}\beta(s)\, {\rm d}{\scriptstyle X}\, {\rm d}t\\
& + \iiiint\limits_{(0,T)\times\R^3\times\Delta(\mathcal{A})\times\mathbb{T}^3}
\widehat{\widetilde{\bf D}}(s,\X,y)
\widehat{{}^\top \!\!\, \widetilde{J}}(s,{\scriptstyle X}) \nabla_{\scriptscriptstyle X} u_0(t,{\scriptstyle X})
\cdot
\nabla_y \widehat{\psi_1} (t, {\scriptstyle X}, s, y)
\, {\rm d}y\, {\rm d}\beta(s)\, {\rm d}{\scriptstyle X}\, {\rm d}t\\
& + \iiiint\limits_{(0,T)\times\R^3\times\Delta(\mathcal{A})\times\mathbb{T}^3}
\widehat{\widetilde{\bf D}}(s,\X,y)
\nabla_y u_1(t,{\scriptstyle X},s,y)
\cdot
\nabla_y \widehat{\psi_1} (t, {\scriptstyle X}, s, y)
\, {\rm d}y\, {\rm d}\beta(s)\, {\rm d}{\scriptstyle X}\, {\rm d}t
= 0.
\end{align*}
The first term in the previous equation vanishes as $u_0$ is independent of the $s$-variable (Lemma \ref{lem:hom:2-scale-compactness}). Finally, performing an integration by parts in the $\scriptstyle X$ variable in the second integral of the previous equation shall lead to
\begin{equation}\label{eq:hom:prop-sigma-limit-u-1}
\begin{array}{lc}
\displaystyle
\iiiint\limits_{(0,T)\times\R^3\times\Delta(\mathcal{A})\times\mathbb{T}^3}
\widehat{\widetilde{{\bf D}}}(s,\X,y)
\nabla_y u_1(t,{\scriptstyle X},s,y)
\cdot
\nabla_y \widehat{\psi_1} (t, {\scriptstyle X}, s, y)
\, {\rm d}y\, {\rm d}\beta(s)\, {\rm d}{\scriptstyle X}\, {\rm d}t\\[0.3 cm]
\displaystyle
+ \iiiint\limits_{(0,T)\times\R^3\times\Delta(\mathcal{A})\times\mathbb{T}^3}
\widehat{\widetilde{{\bf D}}}(s,\X,y)
\widehat{{}^\top \!\!\, \widetilde{J}}(s,{\scriptstyle X}) \nabla_{\scriptscriptstyle X} u_0(t,{\scriptstyle X})
\cdot
\nabla_y \widehat{\psi_1} (t, {\scriptstyle X}, s, y)
\, {\rm d}y\, {\rm d}\beta(s)\, {\rm d}{\scriptstyle X}\, {\rm d}t\\[0.3 cm]
\displaystyle
+ \iiiint\limits_{(0,T)\times\R^3\times\Delta(\mathcal{A})\times\mathbb{T}^3}
\left\{
\widehat{\widetilde{{\bf b}}}(s, {\scriptstyle X}, y)
\cdot
\nabla_y u_1(t,{\scriptstyle X},s,y)
\right\}
\widehat{\psi_1} (t, {\scriptstyle X}, s, y)
\, {\rm d}y\, {\rm d}\beta(s)\, {\rm d}{\scriptstyle X}\, {\rm d}t\\[0.3 cm]
\displaystyle
+ \iiiint\limits_{(0,T)\times\R^3\times\Delta(\mathcal{A})\times\mathbb{T}^3}
\left\{
\left(
\widehat{\widetilde{{\bf b}}}(s, {\scriptstyle X}, y)
-
\widehat{\widetilde{\bar{{\bf b}}}}(s, {\scriptstyle X})
\right)
\cdot
\widehat{{}^\top \!\!\, \widetilde{J}}(s,{\scriptstyle X}) \nabla_{\scriptscriptstyle X} u_0(t,{\scriptstyle X})
\right\}
\widehat{\psi_1} (t, {\scriptstyle X}, s, y)
\, {\rm d}y\, {\rm d}\beta(s)\, {\rm d}{\scriptstyle X}\, {\rm d}t
= 0,
\end{array}
\end{equation}
where we have used Lemma \ref{lem:heu:basic-facts}.(ii) and that ${\bf b}$ is of zero $x$-divergence. The weak formulation \eqref{eq:hom:prop-sigma-limit-u-1} is associated with the following PDE for $u_1(t,\X,s,y)$ in $\T^3$:
\begin{align*}
\widehat{\widetilde{{\bf b}}}(s, {\scriptstyle X}, y)
\cdot
\left(
\nabla_y u_1
+ \widehat{{}^\top \!\!\, \widetilde{J}}(s,{\scriptstyle X}) \nabla_\XX u_0
\right)
& -
\nabla_y
\cdot
\left(
\widehat{\widetilde{{\bf D}}}(s,\X,y)
\left(
\nabla_y u_1
+  \widehat{{}^\top \!\!\, \widetilde{J}}(s,{\scriptstyle X}) \nabla_\XX u_0
\right)
\right)
\\
& =
\widehat{\widetilde{\bar{{\bf b}}}}(s, {\scriptstyle X})
\cdot
\widehat{{}^\top \!\!\, \widetilde{J}}(s,{\scriptstyle X}) \nabla_\XX u_0.
\end{align*}
The above PDE is linear and hence separation of variables may be performed as in \eqref{eq:hom:u-1-seprable-variables}. Taking \eqref{eq:hom:u-1-seprable-variables} and undoing the flow-representation, yields the cell problem \eqref{eq:hom:prop:cell-probem}.
\end{proof}

\subsection{Homogenized problem}\label{ssec:hom:hom}

\begin{proof}[Proof of Theorem \ref{thm:hom:hom}]
Taking $\psi_1\equiv0$ in the weak formulation \eqref{eq:hom:weak-form} and passing to the limit in the sense of $\Sigma$-convergence along $\Phi_\tau$, using Lemma \ref{lem:hom:singular-terms-sigma-limit} for the singular terms, gives
\begin{equation}\label{eq:hom:sigma-limit-wf-hom}
\begin{aligned}
& - \iint\limits_{(0,T)\times\R^3}
u_0(t, {\scriptstyle X})
\frac{\partial \psi}{\partial t}(t, {\scriptstyle X})
\, {\rm d}{\scriptstyle X}\, {\rm d}t\\
& + \iiiint\limits_{(0,T)\times\R^3\times\Delta(\mathcal{A})\times\mathbb{T}^3}
\left(
\widehat{\widetilde{\bar{\bf b}}}(s, {\scriptstyle X})
- \widehat{\widetilde{\bf b}}(s, {\scriptstyle X}, y)
\right)
\cdot
\left(
u_1(t,{\scriptstyle X},s,y)
\widehat{{}^\top \!\! \widetilde{J}}(s,{\scriptstyle X}) \nabla_{\scriptscriptstyle X}\psi (t,{\scriptstyle X})
\right)
\, {\rm d}y\, {\rm d}\beta(s)\, {\rm d}{\scriptstyle X}\, {\rm d}t\\
& + \iiiint\limits_{(0,T)\times\R^3\times\Delta(\mathcal{A})\times\mathbb{T}^3}
\widehat{\widetilde{{\bf D}}}(s,\X,y)
\widehat{{}^\top \!\! \widetilde{J}}(s,{\scriptstyle X}) \nabla_{\scriptscriptstyle X} u_0(t,{\scriptstyle X})
\cdot
\widehat{{}^\top \!\! \widetilde{J}}(s,{\scriptstyle X}) \nabla_{\scriptscriptstyle X} \psi (t, {\scriptstyle X})
\, {\rm d}y\, {\rm d}\beta(s)\, {\rm d}{\scriptstyle X}\, {\rm d}t\\
& + \iiiint\limits_{(0,T)\times\R^3\times\Delta(\mathcal{A})\times\mathbb{T}^3}
\widehat{\widetilde{{\bf D}}}(s,\X,y)
\nabla_y u_1(t,{\scriptstyle X},s,y)
\cdot
\widehat{{}^\top \!\! \widetilde{J}}(s,{\scriptstyle X}) \nabla_{\scriptscriptstyle X} \psi (t, {\scriptstyle X})
\, {\rm d}y\, {\rm d}\beta(s)\, {\rm d}{\scriptstyle X}\, {\rm d}t\\
& - \int\limits_{\R^3}
u^{in}({\scriptstyle X}) \psi(0,{\scriptstyle X})
\, {\rm d}{\scriptstyle X}
= 0.
\end{aligned}
\end{equation}
Substituting \eqref{eq:hom:u-1-seprable-variables} for $u_1(t,s,\X,y)$ in the second term of the above equation yields
\begin{align*}
\iiiint\limits_{(0,T)\times\R^3\times\Delta(\mathcal{A})\times\mathbb{T}^3}
\widehat{\!\! \widetilde{J}}(s,{\scriptstyle X})
\left(
\widehat{\widetilde{\bar{\bf b}}}(s, {\scriptstyle X})
- \widehat{\widetilde{\bf b}}(s, {\scriptstyle X}, y)
\right)
\left(
\widehat{\widetilde{\omega}}(s,\X,y)
\cdot
{}^\top\!\!\, \widehat{\widetilde{J}}(s,\X)\nabla_\XX u_0(t,\X)
\right)
\\
\cdot
\nabla_{\scriptscriptstyle X}\psi (t,{\scriptstyle X})
\, {\rm d}y\, {\rm d}\beta(s)\, {\rm d}{\scriptstyle X}\, {\rm d}t
\\
= \iiiint\limits_{(0,T)\times\R^3\times\Delta(\mathcal{A})\times\mathbb{T}^3}
\left(
\widehat{\!\!\, \widetilde{J}}(s,{\scriptstyle X})
\left(
\widehat{\widetilde{\bar{\bf b}}}(s, {\scriptstyle X})
- \widehat{\widetilde{\bf b}}(s, {\scriptstyle X}, y)
\right)
{}^\top\!\, \widehat{\widetilde{\omega}}(s,\X,y)
{}^\top\!\!\, \widehat{\widetilde{J}}(s,\X)\nabla_\XX u_0(t,\X)
\right)
\\
\cdot 
\nabla_{\scriptscriptstyle X}\psi (t,{\scriptstyle X})
\, {\rm d}y\, {\rm d}\beta(s)\, {\rm d}{\scriptstyle X}\, {\rm d}t.
\end{align*}
Substituting \eqref{eq:hom:u-1-seprable-variables} for $u_1(t,s,\X,y)$ in the fourth term on the left hand side of \eqref{eq:hom:sigma-limit-wf-hom} yields
\begin{align*}
\iiiint\limits_{(0,T)\times\R^3\times\Delta(\mathcal{A})\times\mathbb{T}^3}
\widehat{ \!\! \widetilde{J}}(s,{\scriptstyle X}) 
\widehat{\widetilde{{\bf D}}}(s,\X,y)
\nabla_y
\left(
\widehat{\widetilde{\omega}}(s,\X,y)
\cdot
{}^\top\!\!\, \widehat{\widetilde{J}}(s,\X)\nabla_\XX u_0(t,\X)
\right)
\\
\cdot
\nabla_{\scriptscriptstyle X} \psi (t, {\scriptstyle X})
\, {\rm d}y\, {\rm d}\beta(s)\, {\rm d}{\scriptstyle X}\, {\rm d}t
\\
= \iiiint\limits_{(0,T)\times\R^3\times\Delta(\mathcal{A})\times\mathbb{T}^3}
\left(
\widehat{ \!\! \widetilde{J}}(s,{\scriptstyle X}) 
\widehat{\widetilde{{\bf D}}}(s,\X,y)
{}^\top\nabla_y \widehat{\tilde{\omega}}(s,\X,y)
{}^\top\!\!\, \widehat{\widetilde{J}}(s,\X)\nabla_\XX u_0(t,\X)
\right)
\\
\cdot
\nabla_{\scriptscriptstyle X} \psi (t, {\scriptstyle X})
\, {\rm d}y\, {\rm d}\beta(s)\, {\rm d}{\scriptstyle X}\, {\rm d}t.
\end{align*}
Hence the limit weak formulation \eqref{eq:hom:sigma-limit-wf-hom} rewrites as
\begin{align*}
- \iint\limits_{(0,T)\times\R^3}
u_0(t, {\scriptstyle X})
\frac{\partial \psi}{\partial t}(t, {\scriptstyle X})
\, {\rm d}{\scriptstyle X}\, {\rm d}t
+ \iint\limits_{(0,T)\times\R^3}
\mathfrak{D}(\X)\nabla_\XX u_0(t, {\scriptstyle X})
\cdot 
& \nabla_\XX \psi(t,\X)
\, {\rm d}{\scriptstyle X}\, {\rm d}t
\\
& - \int\limits_{\R^3}
u^{in}({\scriptstyle X}) \psi(0,{\scriptstyle X})
\, {\rm d}{\scriptstyle X}
= 0,
\end{align*}
where the expression for the diffusion matrix $\mathfrak{D}(\X)$ is given by
\begin{align*}
\mathfrak{D}(\X)
& = 
\int\limits_{\Delta(\mathcal{A})}
\widehat{\!\! \widetilde{J}}(s,{\scriptstyle X})
\left(
\int\limits_{\T^3}
\left(
\widehat{\widetilde{\bar{\bf b}}}(s, {\scriptstyle X})
- \widehat{\widetilde{\bf b}}(s, {\scriptstyle X}, y)
\right)
{}^\top\!\, \widehat{\widetilde{\omega}}(s,\X,y)
\, {\rm d}y
\right)
{}^\top\!\!\, \widehat{\widetilde{J}}(s,\X)
\, {\rm d}\beta(s)
\\
& +
\int\limits_{\Delta(\mathcal{A})}
\widehat{\!\! \widetilde{J}}(s,{\scriptstyle X})
\left(
\int\limits_{\T^3}
\widehat{\widetilde{{\bf D}}}(s,\X,y)
\, {\rm d}y
\right)
{}^\top\!\!\, \widehat{\widetilde{J}}(s,\X)
\, {\rm d}\beta(s)
\\
& +
\int\limits_{\Delta(\mathcal{A})}
\widehat{\!\! \widetilde{J}}(s,{\scriptstyle X})
\left(
\int\limits_{\T^3}
\widehat{\widetilde{{\bf D}}}(s,\X,y)
{}^\top\nabla_y \widehat{\widetilde{\omega}}(s,\X,y)
\, {\rm d}y
\right)
{}^\top\!\!\, \widehat{\widetilde{J}}(s,\X)
\, {\rm d}\beta(s).
\end{align*}
Using the cell problem, we arrive at the desired expression for the effective diffusion. As the computations are exactly similar to the ones present in the proof of Proposition \ref{prop:heu:formal-result}, we skip the details.
\end{proof}

\section{Discussion of assumptions}\label{sec:Exm}

In this section we discuss the assumptions (detailed in Section \ref{ssec:hom:assumptions}) of the homogenization result (Theorem \ref{thm:hom:hom}) and give both the examples where they are satisfied and counterexamples where the failure of these assumptions can lead either to trivial or non-unique limits. We remark that the main obstacle to obtaining homogenization results in our setting is, in fact, not related to the oscillating coefficients, but rather to deriving an effective equation in Lagrangian coordinates.  
\subsection{Bounds on the Jacobian}
The main restriction on the fluid flow is Assumption \ref{ass:hom:Jacobain-algebra} which implies that the Jacobian of the flow is \emph{uniformly bounded in time}, i.e. in the $\tau$ variable. This is a highly non-generic assumption, but is needed for the validity of the posited asymptotic expansion \eqref{eq:intro:2scale-flow-exp-classic}. Indeed, if the Jacobian is not uniformly bounded in time, then, for example, the right hand side of the $\mathcal{O}(\varepsilon^{0})$ equation in the cascade \eqref{eq:heu:cascade} may grow to be of $\mathcal{O}(\varepsilon^{-1})$ for sufficiently large values of the fast time variable $\tau$, breaking the formal expansion. First we shall give some examples of mean fluid fields which obey this assumption, which although restrictive, still covers a large class of vector fields.

\begin{Exm}[Constant drift]\label{Exm:exm:constant-drift}
The most obvious example is the constant drift flow $\bar{\bf b}(x)={\bf b}^*$ for a constant vector ${\bf b}^*\in\R^d$. In this case the Jacobian matrix is the identity for all times. This case falls under the regime of \emph{two-scale convergence with drift} studied in \cite{papanicolaou1995diffusion,allaire2010two}.
\end{Exm}

\begin{Exm}[Euclidean motions]\label{Exm:exm:euclidean-motions}
Euclidean motions are the composition of a translation and a rigid rotation. An autonomous flow consists of Euclidean motions if and only if the vector field is given by $\bar{\bf b}(x)={\bf A}x+{\bf b}^*$ for a constant skew-symmetric matrix ${\bf A}$ and a constant vector ${\bf b}^*$. The associated Jacobian matrix is an orthogonal matrix and hence of norm $1$.
\end{Exm}

\begin{Exm}[Asymptotically constant drift]\label{Exm:asymptotically-constant-drift}
Let the mean flow $\bar{\bf b}$ in dimension $d\ge2$ be given by
\begin{align*}
\bar{\bf b}(x)=\begin{cases}
{\bf b}^{*}&\text{when } x_1<-R,\\
{\bf c}(x)&\text{when } x_1\in [-R,R],\\
{\bf b}^{**}&\text{when } x_1>R,
\end{cases}
\end{align*}
where $R>0$, ${\bf e}_1\cdot {\bf b}^{*},{\bf e}_1\cdot{\bf b}^{**}>0$ and  ${\bf c}(x)$ is chosen to make $\bar{\bf b}$ continuously differentiable and divergence free. To ensure that the Jacobian of the flow is uniformly bounded in time we require that any integral curve spends only finite time $T$ in $\{x_1\in [-R,R]\}$, which implies that the Jacobian is norm bounded by $C\exp(T\|\nabla {\bf c}\|_{L^\infty})$. This can easily be achieved by requiring that ${\bf e}_1\cdot {\bf c}(x)\ge c>0$.
\end{Exm}

We remark also that the Jacobian in each of these examples belongs to some algebra w.m.v., specifically the Jacobian in Examples \ref{Exm:exm:constant-drift} and \ref{Exm:asymptotically-constant-drift} belong to the algebra of functions that converge at infinity (see Example \ref{exm:abs:a.w.m.v-convergence-at-infinity}), and the Jacobian in Example \ref{Exm:exm:euclidean-motions} belongs to the algebra of almost periodic functions (see Example \ref{exm:abs:AP}).

\subsection{Necessity of uniformly bounded Jacobian}\label{subsec:necessity-bounded-jacobian}
The assumption of uniform bounds on the Jacobian is not a mere technical assumption. We illustrate this with a counterexample, which we have made as simple as possible to allow explicit calculations.

\begin{CExm}[Blow-up of the Jacobian for a shear flow]\label{Exm:CExm:blow-up}
Consider the simplest example of a shear flow:
\begin{equation*}
\bar{{\bf b}}(x_1,x_2)=\begin{bmatrix}x_2\\0\end{bmatrix}.
\end{equation*}
An easy computation gives that the flow $\Phi_\tau$ generated by this vector field and its Jacobian are given by
\begin{equation*}
\Phi_{\tau}(x_1,x_2)=\begin{bmatrix}
x_1+\tau x_2\\x_2
\end{bmatrix},\qquad 
J(-\tau,x)=\begin{bmatrix}
1&\tau\\0&1
\end{bmatrix}.
\end{equation*}
In particular, the Jacobian grows linearly in time.

Consider the parabolic problem on $]0,T[\times\R^2$ given by
\begin{equation}\label{Exm:eq:blow-up-parabolic}
\frac{\partial u^\eps}{\partial t}+\frac 1\eps \bar{{\bf b}}(x_1,x_2)\cdot\nabla u^\eps-\Delta u^\eps=0
\end{equation}
(Note that this example does not have oscillating coefficients.) The posited asymptotic expansion \eqref{eq:intro:2scale-flow-exp-classic} becomes in this case
\begin{equation*}
u^\eps(t,x_1,x_2)
\ \approx\
 u_0\left(t,\frac{t}{\eps},x_1-\frac{x_2t}{\eps},x_2\right)
 \ +\
 \eps\, u_1\left(t,\frac{t}\eps,x_1-\frac{x_2t}{\eps},x_2\right)\ +\ \dotsb
\end{equation*}
and the cascade of equations \eqref{eq:heu:cascade} becomes
\begin{equation*}
\begin{array}{ccl}
\displaystyle
\mathcal{O}(\eps^{-2}): &
\displaystyle
0 &
\displaystyle
= 0,
\\ [0.3 cm]
\displaystyle
\mathcal{O}(\eps^{-1}): &
\displaystyle
0&
\displaystyle
= -\frac{\partial u_0}{\partial \tau},
\\[0.3 cm]
\displaystyle
\mathcal{O}(\eps^{0}): &
\displaystyle
0&
\displaystyle
= -\frac{\partial u_0}{\partial t} -\frac{\partial u_1}{\partial \tau}
+ (1+\tau^2)\frac{\partial^2 u_0}{\partial \X_1^2}-2\tau\frac{\partial^2 u_0}{\partial \X_1\partial \X_2}+\frac{\partial^2 u_0}{\partial \X_2^2},
\end{array}
\end{equation*}
where the $\mathcal{O}(\eps^{-2})$ equation is trivial due to the lack of oscillating coefficients. But this cascade is only valid for times $\tau \ll \eps^{-1/2}$, as at this value of $\tau$ the $(1+\tau^2)$ coefficient in the posited $\mathcal{O}(\eps^{0})$ equation jumps order. To be a valid asymptotic expansion for the parabolic problem \eqref{Exm:eq:blow-up-parabolic}, we require it to be valid for $\tau\in[0,T/\eps]$, i.e. up to $\mathcal{O}(\eps^{-1})$ values of $\tau$. This means that the posited asymptotic expansion cannot be correct. 

Indeed, the problem \eqref{Exm:eq:blow-up-parabolic} can be explicitly solved using the Fourier transform. Let $(\xi_1,\xi_2)$ be Fourier variables corresponding to $(\X_1,\X_2)$. Then an easy computation yields
\begin{equation*}
\hat{u}^\eps(t,\xi_1,\xi_2)=\exp\left(\int^t_0-|\xi_1|^2-\left|\xi_2-\frac{s}{\eps}\xi_1\right|^2\,{\rm d}s\right)\hat{u}^\eps(0,\xi_1,\xi_2),
\end{equation*}
(where we have abused notation and used $\hat{\cdot}$ to denote the Fourier instead of Gelfand transform). The integrand in the above exponential converges pointwise to $-\infty$ as $\eps\to0$ so long as $s\xi_1\ne0$. Therefore, $\hat{u}\to0$ almost everywhere in $[0,T]\times\mathbb{R}^2$ as $\varepsilon\to0$, and it follows from dominated convergence and Plancherel's theorem that $u^\eps\to 0$ strongly in $L^2([0,T]\times\mathbb{R}^2)$. So, not only is the asymptotic expansion not correct, but the limit as $\eps\to0$ is trivial.
\end{CExm}

This counterexample illustrates a general phenomenon for shear flows, where the convection enhances the diffusion (see for example \cite{fannjiang1994convection} where this is considered in detail for the more complicated case of cats eye flows, and \cite{hamel2010extinction,constantin2008diffusion,berestycki2005elliptic} where conditions under which the solution converges strongly to zero are studied). As a consequence of this enhancement, the time scale on which diffusion is observed is different and one should not expect to obtain a non-trivial limit in the scaling we consider. We give a partial result to this effect below. The authors shall address this problem in a forthcoming publication \cite{holdingetal2016shearflow}.

\begin{Prop}
Let the assumptions of Proposition \ref{prop:QA:well-posed} hold and let $u^\eps$ be the solution to  \eqref{eq:heu:CD}-\eqref{eq:heu:IV}. Let $v^\eps$ be the solution in Lagrangian coordinates, i.e. $v^\eps(t,\X)=u^\eps(t,\Phi_{t/\eps}(\X))$. Let $\xi\in\R^d$ be a unit vector, and suppose that for some (non-empty) open set $A\subset \R^d$ we have, for $\X\in A$,
\begin{equation}\label{Exm:blow-up-J-grows-on-A}
\lim_{\tau\to\infty}|{}^\top\!\!\, \widetilde{J}(\tau,\X)\xi|=\infty.
\end{equation}
Then for any $v_0$ a $L^2((0,T);H^1(\R^d))$-weak limit of $v^\eps$, we have $\xi\cdot\nabla_\XX v_0=0$ on $(0,T)\times A$.
\end{Prop}
\begin{Rem}
As the set $A$ is independent of the choice of initial data $u^{in}$, the initial data can be chosen so that $v_0\not\in C([0,T];L^2(\R^d))$, and in particular so that $v_0$ does not solve a `nice' parabolic PDE with this initial datum.
\end{Rem}
\begin{proof}
Without loss of generality we can assume that $A$ is bounded, and by applying Egorov's theorem it is sufficient to prove the claim for $A$ measurable with the limit \eqref{Exm:blow-up-J-grows-on-A} uniform on $A$. By converting \eqref{eq:QA:apriori-estimates} to Lagrangian coordinates, we have the estimate
\begin{equation*}
\iint\limits_{(0,T)\times\R^d}\left|{}^\top\!\!\, \widetilde{J}(t/\eps,\X)\nabla_\XX v^{\eps}(t,\X)\right|^2\,{\rm d}\X{\rm d}t\le C
\end{equation*}
with $C$ only depending on $u^{in}$. Let $t_0\in(0,T)$ be arbitrary, and define 
\begin{equation*}
\theta(\eps):=\inf_{\X\in A,\tau \ge t_0/\eps}\left|{}^\top\!\!\, \widetilde{J}(\tau,\X)\xi\right|^2,
\end{equation*}
so that $\theta(\eps)\to\infty$ as $\eps\to0$. Then we have
\begin{align*}
\iint\limits_{(t_0,T)\times A}\left|\xi\cdot\nabla_\XX v^\eps(t,\X)\right|^2\, {\rm d}\X{\rm d}t&\le \theta(\eps)^{-1}
\iint\limits_{(t_0,T)\times A}\left|{}^\top\!\!\, \widetilde{J}(t/\eps,\X)\xi\right|^2\left|\xi\cdot\nabla_\XX v^{\eps}(t,\X)\right|^2\, {\rm d}\X{\rm d}t\\
&\le \theta(\eps)^{-1}\iint\limits_{(t_0,T)\times A}\left|{}^\top\!\!\, \widetilde{J}(t/\eps,\X)\nabla_\XX v^{\eps}(t,\X)\right|^2\, {\rm d}\X{\rm d}t\\
&\le \theta(\eps)^{-1}\iint\limits_{(0,T)\times \R^d}\left|{}^\top\!\!\, \widetilde{J}(t/\eps,\X)\nabla_\XX v^{\eps}(t,\X)\right|^2\, {\rm d}\X{\rm d}t\\
&\le C\theta(\eps)^{-1}.
\end{align*}
That $\xi\cdot\nabla_{\XX}v_0=0$ on $(t_0,A)$ follows from upper semi-continuity under weak convergence. That this holds for $(0,T)\times A$ follows as $t_0$ was arbitrary.
\end{proof}

\subsection{Flow representations of coefficients}\label{ssec:Exm:counter}
The main assumptions upon the coefficients ${\bf b}(x,y)$ and ${\bf D}(x,y)$ is that their flow representations $\widetilde{\bf b}(\tau,\X,y)$ and $\widetilde{\bf D}(\tau,x,y)$ belong to some fixed algebra w.m.v. $\mathcal{A}$. The reason we require this is to ensure that we obtain a single unique homogenized equation. This is in contrast to the uniform bounds on the Jacobian, which as described above, we require in order to be sure that we can obtain \emph{any} non-trivial limit. We illustrate the non-uniqueness phenomenon with the following counterexample. We remark that, again, the difficulty is present without any rapid spatial oscillations, or complicated mean flows.
\begin{CExm}[Non-uniqueness of the limit]\label{Exm:CExm:non-uniqueness}
Consider the $1+1$ dimensional parabolic problem on $]0,T[\times \mathbb{R}$ given by
\begin{equation}\label{Exm:eq:counter-example-flow-parabolic}
\frac{\partial u^\eps}{\partial t}
+ \frac{1}{\eps}\frac{\partial u^\eps}{\partial x}-\frac{\partial}{\partial x}\left({\bf D}(x)\frac{\partial u^\eps}{\partial x}\right)=0,
\end{equation}
where ${\bf D}(x)$ is given by
\begin{equation}\label{Exm:eq:counter-example-diffusion-coefficient}
{\bf D}(x)=\begin{cases}
1&\text{ if } |x|\in [2^{(2n)^2},2^{(2n+1)^2})\text{ for some integer }n\ge0,\\
2&\text{ otherwise}.
\end{cases}
\end{equation}
(Note that although this function ${\bf D}$ is not continuous, the example could be easily modified to have ${\bf D}\in C^\infty$.) The corresponding mean flow field $\bar{{\bf b}}$ and its flow and Jacobian are given by
\begin{equation*}
\bar{{\bf b}}(x)=1,\qquad \Phi_\tau(x)=x+\tau,\qquad J(\tau,x)=1,
\end{equation*}
i.e. we are in the constant drift case. The posited asymptotic expansion \eqref{eq:intro:2scale-flow-exp-classic} becomes
\begin{equation*}
u^\eps(t,x_1,x_2)\approx u_0\left(t,\frac{t}{\eps},x-\frac{t}{\eps}\right) + \eps\, u_1\left(t,\frac{t}\eps,x-\frac{t}{\eps}\right)+\dotsb
\end{equation*}
and the cascade of equations \eqref{eq:heu:cascade} is 
\begin{equation*}
\begin{array}{ccl}
\displaystyle
\mathcal{O}(\eps^{-2}): &
\displaystyle
0 &
\displaystyle
= 0,
\\ [0.3 cm]
\displaystyle
\mathcal{O}(\eps^{-1}): &
\displaystyle
0&
\displaystyle
= -\frac{\partial u_0}{\partial \tau},
\\[0.3 cm]
\displaystyle
\mathcal{O}(\eps^{0}): &
\displaystyle
0&
\displaystyle
= -\frac{\partial u_0}{\partial t} -\frac{\partial u_1}{\partial \tau}
+ \frac{\partial }{\partial \X}\left(\widetilde{{\bf D}}(\tau,\X)\frac{\partial u_0}{\partial \X}\right),
\end{array}
\end{equation*}
where the simplicity of the equations is due to the lack of fast spatial oscillations, and the flow representation of ${\bf D}$ is given by
\begin{equation*}
\widetilde{{\bf D}}(\tau,\X)={\bf D}(\X+\tau).
\end{equation*}
Unlike in the previous Counterexample \ref{Exm:CExm:blow-up}, there is nothing obviously wrong with this asymptotic expansion. The problem comes when we try to average the $\mathcal{O}(\eps^0)$ equation in the fast time variable $\tau$. Consider the limit
\begin{equation*}
(M\widetilde{{\bf D}})(\X)=\lim_{l\to\infty}\frac1{2l}\int\limits^l_{-l}\widetilde{{\bf D}}(\tau,\X)\, {\rm d}\tau.
\end{equation*}
We claim that this limit does not exist. Indeed, let $l_n=2^{(2n)^2}$ then for $n\ge1$,
\begin{equation}
\left| \frac1{2l_n}\int\limits^{l_n}_{-l_n}\widetilde{{\bf D}}(\tau,0)\,{\rm d}\tau-2\right|\le \frac{2\cdot 2^{(2n-1)^2}}{2\cdot2^{(2n)^2}}=2^{-4n+1}\to 0 \text{ as }n\to\infty,
\end{equation}
as the contribution from $|\tau|\in [2^{(2n-1)^2},2^{(2n)^2})$ dominates. Similarly, for $l_n'=2^{(2n-1)^2}$ we obtain
\begin{equation*}
\lim_{n\to\infty}\frac1{2l'_n}\int\limits^{l'_n}_{-l'_n}\widetilde{{\bf D}}(\tau,0)\, {\rm d}\tau=1.
\end{equation*}
Thus the limit $l\to\infty$ depends upon the choice of sequence. 

We remark that, as a consequence, $\widetilde{{\bf D}}$, cannot belong to any algebra w.m.v., and therefore none of our results apply to this problem.
\end{CExm}

One might think that the failure of the asymptotic expansion in the above counterexample could be rectified by some smarter choice of expansion. However, this is not the case. For this counterexample it is impossible to obtain a unique homogenized equation as we will now show.

\begin{Prop}
Let $v^\eps(t,\X)=u^\eps(t,\X+t/\eps)$ where $u^\eps$ solves the problem \eqref{Exm:eq:counter-example-flow-parabolic} in Counterexample \ref{Exm:CExm:non-uniqueness} with initial data $u^{in}\in L^2(\mathbb{R})$. Then there are two sequences $\eps\to0$ and $\eps'\to0$ such $v^{\eps}\rightharpoonup u_0$ and $v^{\eps'}\rightharpoonup u_0'$ weakly in $L^2([0,T]\times\mathbb{R})$, which solve the homogenized problems
\begin{equation}\label{Exm:eq:two-different-limit-equations}
\begin{aligned}
\frac{\partial u_0}{\partial t}-\frac{\partial^2 u_0}{\partial \X^2}=0,& \qquad\text{and}\qquad \frac{\partial u'_0}{\partial t}-2\frac{\partial^2 u'_0}{\partial \X^2}=0,
\end{aligned} 
\end{equation}
with $u_0(0,\X)=u_0'(0,\X)=u^{in}(\X)$, which are different equations.
\end{Prop}
\begin{proof}
Without loss of generality, let $T=1$. Define $\eps_n=1/l_n$ and $\eps'_n=1/l'_n$ for $l_n=2^{(2n)^2}$ and $l_n'=2^{(2n-1)^2}$ as in Counterexample \ref{Exm:CExm:non-uniqueness}. We will first show that the following strong convergences 
\begin{equation*}
\begin{aligned}
\widetilde{{\bf D}}(t/\eps_n,\X)\to 2,& \qquad\text{and}\qquad \widetilde{{\bf D}}(t/\eps'_n,\X)\to1,
\end{aligned} 
\end{equation*}
hold in $L^2_{loc}([0,1]\times \mathbb{R})$ as $n\to\infty$. Indeed, let $K\in(0,\infty)$ be arbitrary, then,
\begin{equation}\label{Exm:eq:counter-L^2-convergence-integral}
\int\limits^1_0\int\limits^K_{-K}\left|\widetilde{{\bf D}}(t/\eps_n,\X)-2\right|^2\, {\rm d}t{\rm d}\X=\int\limits^K_{-K}\int\limits_0^{T_n(\X)}\left|{\bf D}(\X+t/\eps_n)-2\right|^2\, {\rm d}t{\rm d}\X+0
\end{equation}
where $T_n(\X)$ is chosen as the solution of $\X+T_n(\X)/\eps_n=2^{(2n-1)^2}$ (or $0$ if this is negative) so that ${\bf D}(\X+t/\eps)=2$ for $t\in[T_n,1]$. Hence
\begin{equation*}
T_n(\X)=2^{(2n-1)^2-4n^2}+\X\varepsilon_n\le 2^{-4n+1}+K\varepsilon_n\to0 \text{ as }n\to\infty,
\end{equation*}
and the convergence of \eqref{Exm:eq:counter-L^2-convergence-integral} to zero follows easily. The proof that $\widetilde{{\bf D}}(t/\eps'_n)$ converges to $1$ is similar, where instead $T_n(\X)$ is chosen so that ${\bf D}(\X+t/\eps_n')=1$ for $t\in[T_n(\X),1]$.

We will now show the convergence of $v^{\eps_n}$ to $u_0$. The argument for $v^{\eps_n'}$ is analogous, using instead the convergence of $\widetilde{{\bf D}}(t/\eps'_n)\to2$, and we leave it to the reader. Straight forward estimates allow us to pass to a subsequence $n_k$ on which $v^{\eps_{n_k}}(t,\X)$ and $\frac{\partial v^{\eps_{n_k}}}{\partial \X}$ converge $L^2([0,1]\times\mathbb{R})$-weak to limits $u_0$ and $\frac{\partial u_0}{\partial \X}$ as $k\to\infty$. Uniqueness of solutions of the equation for $u_0$ will later show that $v^{\eps_n}\to u_0$ as $n\to\infty$, i.e. the original sequence converges. We abuse notation and keep the original sequence. Writing \eqref{Exm:eq:counter-example-flow-parabolic} in $(t,\X)=(t,x-t/\eps)$ coordinates, multiplying by a test function $\varphi(t,\X)$ and integrating by parts, we obtain
\begin{align*}
&\int\limits_{-\infty}^\infty \varphi(0,\X)u^{in}(0,\X)\,{\rm d}\X-\int\limits^1_0\int\limits^\infty_{-\infty} \frac{\partial \varphi}{\partial t}(t,\X)v^{\eps_n}(t,\X)\, {\rm d}\X {\rm d}t\\
&\qquad+\int\limits^1_0\int\limits^\infty_{-\infty}\widetilde{{\bf D}}(t/\eps_n,\X)\frac{\partial \varphi}{\partial \X}(t,\X)\frac{\partial v^{\eps_n}}{\partial \X}(t,\X)\, {\rm d}\X{\rm d}t=0.
\end{align*}
By the weak convergences of $v^{\eps_n}\to u_0$, $\frac{\partial v^{\eps_n}}{\partial \X}\to \frac{\partial u_0}{\partial \X}$, the strong convergence $\widetilde{{\bf D}}(t/\eps_n,\X)\to 2$ and the compact support of $\varphi$ we can pass to the limit as $n\to\infty$ in each of these terms to obtain the weak formulation of the equation \eqref{Exm:eq:two-different-limit-equations} for $u_0$.
\end{proof}

We remark that, although the above counterexample features bad behaviour in the diffusion coefficient, similar examples could be constructed where the undesirable behaviour is in the drift term $\bar{\bf b}$ (or ${\bf b}$) or the Jacobian matrix $J$. The issue here is the appearance of the spatial scale $x=\mathcal{O}(\eps^{-1})$ in the problem due to the $\mathcal{O}(\eps^{-1})$ mean drift. Such a scale is not present when the average convection is zero, i.e. $\bar{{\bf b}}=0$, even in the convection dominated regime. This additional spatial scale is exploited in the choice of diffusion coefficient \eqref{Exm:eq:counter-example-diffusion-coefficient}, which exhibits different behaviour at a sequence of spatial scales tending to infinity. 

Next we show that this bad behaviour is a problem only at infinity, in the sense that if the trajectories of $\Phi_\tau$ are bounded, then our assumptions always hold.

\begin{Prop}\label{Exm:prop:flow-reps}
Let Assumption \ref{heu:ass:bound-on-J} hold. Then exactly one of the following hold:
\begin{enumerate}[(i)]
\item $\Phi_\tau$ has bounded orbits, i.e. for any $x$ the set $\{\Phi_\tau(x):\tau\in\R\}$ is bounded.
\item $\Phi_\tau$ converges to infinity, i.e. for any $x$ we have $|\Phi_\tau(x)|\to \infty$ as $|\tau|\to\infty$.
\end{enumerate}
Let $\mathcal{AP}$ denote the algebra of almost-periodic functions (Example \ref{exm:abs:AP}). In each respective case, the following also holds:
\begin{enumerate}[(i)]
\item $\Phi_\tau$ is uniformly almost-periodic, i.e. $\Phi_\tau(x)\in [C(\R^d;\mathcal{AP})]^d$. For every $f(x,y)\in C(\R^d\times \T^d)$, the flow-representation is uniformly almost-periodic, i.e. $\widetilde{f}\in C(\R^d\times\T^d;\mathcal{AP})$. If additionally $J(\tau,x)$ is uniformly continuous on $\R\times K$ for each compact set $K\subset \R^d$, then $J$ and $\tilde{J}$ are uniformly almost-periodic, i.e. $J,\widetilde{J}\in [C(\R^d;\mathcal{AP})]^{d\times d}$.
\item Let $f\in C(\R^d\times\T^d)$ converge to a limit as $|x|\to\infty$, i.e. $\lim_{|x|\to\infty}f(x,y)$ exists and is finite for each $y\in\T^d$. Then for each $x,y\in \R^d\times\T^d$, the flow-representation $\widetilde{f}(\cdot,x,y)$ belongs to the algebra of functions that converge at infinity (Example \ref{exm:abs:a.w.m.v-convergence-at-infinity}).
\end{enumerate}
\end{Prop} 
\begin{Rem}
The almost-periodicity of the Jacobian of an (locally) uniformly almost-periodic flow is a subtle issue as there are uniformly almost-periodic functions whose derivative is not uniformly almost-periodic. The assumption in (i) that $J$ is uniformly continuous is to side step this issue.
\end{Rem} 
To prove this proposition we need a definition.
\begin{Defi}[Equicontinuous flow]
A one-parameter group $\phi_\tau$ of homeomorphisms of $K\subseteq\R^d$ is \emph{equicontinuous} if for any $\eps>0$ and $x\in K$ there is a $\delta=\delta(x,\eps)$ such that whenever $|x'-x|\le \delta$ and $x'\in K$ it holds that $|\phi_{\tau}(x)-\phi_\tau(x')|\le \eps $ for all $\tau\in\R$.
\end{Defi}
\begin{proof}[Proof of Proposition \ref{Exm:prop:flow-reps}]
We first prove the dichotomy. Let $x\in\R^d$ be fixed. We first claim that either $|\Phi_\tau(x)|\to\infty$ as $|\tau|\to\infty$ or its orbit is bounded. Suppose $|\Phi_\tau(x)|\not\to\infty$ as $|\tau|\to\infty$, then there must be a compact set $K$ containing $x$ and a sequence of times $\tau_n$ with $|\tau_n|\to\infty$ as $n\to\infty$ and $\Phi_{\tau_n}(x)\in K$. Without loss of generality let $0<\tau_1<\tau_2\dotsb$. By integrating Assumption \ref{heu:ass:bound-on-J}, for any $n\ge1$ it holds that
\begin{equation*}
\sup_{\tau_n\le \tau\le\tau_{n+1}}|\Phi_{\tau}(x)-\Phi_{\tau-\tau_n}(x)|\le C|\Phi_{\tau_n}(x)-x|\le C\operatorname{diam}(K)
\end{equation*}
and hence the forward orbit is bounded. We now claim that the backwards orbit is also bounded. Indeed, let $s>0$ be arbitrary, then, using Assumption \ref{heu:ass:bound-on-J} once more we have
\begin{equation*}
|\Phi_{-s}(x)-x|=|\Phi_{-s}(x)-\Phi_{-s}(\Phi_s(x))|\le C|x-\Phi_{s}(x)|
\end{equation*}
and the right hand side is bounded uniformly in $s>0$.

We have shown that the dichotomy holds for some fixed $x$, but this together with Assumption \ref{heu:ass:bound-on-J} imply that the same dichotomy holds for all $x$. Indeed, consider the orbit starting from an arbitrary $x'$, then
\begin{equation*}
\sup_{\tau\in\R}|\Phi_\tau(x)-\Phi_\tau(x')|\le C|x-x'|<\infty 
\end{equation*}
which we obtain by again integrating Assumption \ref{heu:ass:bound-on-J}. This implies that if the orbit of $x$ is bounded (resp. converges to infinity) then the orbit of $x'$ is bounded (resp. converges to infinity).

We now prove the claims, starting with (i). Let $R>0$ be arbitrary, then the set
\begin{equation*}
K_R=\overline{\{\Phi_{\tau}(x):\tau\in\R,|x|\le R\}}
\end{equation*}
is invariant under $\Phi_\tau$ and compact. Moreover, the $K_R$ are nested and cover $\R^d$. It is thus sufficient to prove the claims on $K_R$. Note that $(\Phi_\tau,K,|\cdot|)$ is a compact dynamical system, and Assumption \ref{heu:ass:bound-on-J} implies that it is \emph{equicontinuous} in the sense of the above definition. It is a classical result of topological dynamical systems (see e.g. \cite{Ellis-lecture-notes}) that for compact dynamical systems the property of equicontinuity is equivalent to being \emph{uniformly almost-periodic}, in the sense that $\Phi_\tau(x)\in [C(K_R;\mathcal{AP})]^d$. Now suppose that $f\in C(\R^d\times\T^d)$, then $f$ is uniformly continuous on $K_R\times \T^d$ as this set is compact. Moreover, as $K_R$ is invariant under $\Phi_\tau$ the function $\widetilde{f}(\tau,x,y)$ restricted to $x\in K_R$ depends only on $f$ restricted to $K_R\times \T^d$. Hence $\widetilde{f}(\tau,x,y)=f(\Phi_\tau(x),y)$ (restricted to $x\in K_R$) is the composition of a uniformly continuous function and a uniformly almost-periodic function, and is uniformly almost-periodic. Finally, suppose that $J$ is uniformly continuous on $\R\times K_R$, then the difference quotients defined for any unit vector $\xi\in\R^d$, by
\begin{equation*}
J_{h}(\tau,x)\xi=\frac{\Phi_{-\tau}(x+h\xi)-\Phi_{-\tau}(x)}{|h|}
\end{equation*}
converge in $[C(K_{R/2}\times \R)]^d$ as $\R\ni h\to0$ to $J(\tau,x)\xi$. As both terms in the difference quotient are uniformly almost-periodic, the limit is also. That $\widetilde{J}$ is uniformly almost-periodic can be proved in the same way.

Now we prove the claim for (ii). Let $f(x,y)\in C(\R^d\times \T^d)$ be as assumed and converge to $g(y)$ as $|x|\to\infty$. Clearly $\widetilde{f}(\tau,x,y)\to g(y)$ as $|\tau|\to\infty$, it only remains to show that $\widetilde{f}$ is uniformly continuous. To this end, note that $f$ is continuous on $\overline{\R^d}\times \T$ where $\overline{\R^d}$ is the one-point compactification of $\R^d$. As this set is compact, $f$ is uniformly continuous on this set. Moreover, as $\Phi_\tau(x)$ is uniformly continuous from $\R\times\R^d$ to $\overline{\R^d}$, the composition $\widetilde{f}$ is uniformly continuous, which completes the proof of the proposition.
\end{proof} 

\section{Applications to other models}\label{sec:explicit}

In this section, we consider some explicit models and perform the asymptotic analysis using the $\Sigma$-$\Phi_\tau$ convergence.

\subsection{Lagrangian coordinates}
For a smooth fluid field $\bar{\bf b}(x)\in C^1(\R^d;\R^d)$ and diffusion coefficient ${\bf D}(x)\in L^\infty(\R^d;\R^{d\times d})$, consider the Cauchy problem with large convection term
\begin{align}
\frac{\partial u^\eps}{\partial t}
+
\frac{1}{\eps}
\bar{\bf b}(x)
\cdot
\nabla u^\eps 
-
\nabla
\cdot
\Big(
{\bf D}(x)\nabla u^\eps 
\Big)
=
0
\qquad
\mbox{ for }
(t,x)\in\, ]0,T[\times\R^d.
\end{align}
Let $\Phi_\tau(x)$ be the flow associated with the vector field $\bar{\bf b}(x)$. As Remark \ref{rem:abs:test-fn-no-space-oscillations} suggests, we consider the $\Sigma$-$\Phi_\tau$ convergence with no oscillations in space, i.e. with test functions $\psi\left(t,\Phi_{-t/\eps}(x),\frac{t}{\eps}\right)$:
\begin{align}
\lim_{\eps\to0}
\iint\limits_{(0,T)\times\R^d}
u^\eps(t,x)
\psi\left(t,\Phi_{-t/\eps}(x),\frac{t}{\eps}\right)
\, {\rm d}x\, {\rm d}t
=
\iiint\limits_{(0,T)\times\R^d\times\Delta(\mathcal{A})}
u_0(t,\X,s)
\widehat{\psi}(t,\X,s)
\, {\rm d}\beta(s)
\, {\rm d}\X
\, {\rm d}t.
\end{align}
An argument similar to the proof of Lemma \ref{lem:hom:2-scale-compactness} implies that the above limit function $u_0$ is independent of the $s$ variable. As done earlier in Section \ref{sec:homog_result}, we need to pass to the limit (as $\eps\to0$) in the weak formulation
\begin{align*}
& -
\iint\limits_{(0,T)\times\R^d}
u^\eps(t,x)
\frac{\partial \psi}{\partial t}
\left(
t, \Phi_{-t/\eps}(x)
\right)
\, {\rm d}x\, {\rm d}t
-
\int\limits_{\R^d}
u^{in}(x)
\psi(0,x)
\, {\rm d}x
\\
& +
\iint\limits_{(0,T)\times\R^d}
\widetilde{\bf D}\left(\frac{t}{\eps}, \Phi_{-t/\eps}(x)\right)
\nabla u^\eps(t,x)
\cdot 
{}^\top\!\!\, \widetilde{J}\left(\frac{t}{\eps}, \Phi_{-t/\eps}(x)\right)
\nabla_\XX \psi\left(t, \Phi_{-t/\eps}(x)\right)
\, {\rm d}x\, {\rm d}t
= 
0.
\end{align*}
Under the Assumption \ref{ass:hom:D-x-y-algebra} on the diffusion coefficient ${\bf D}(x)$ and under the Assumption \ref{ass:hom:Jacobain-algebra} on the Jacobian matrix $J(\tau,x)$, the product ${}^\top\!\!\, \widetilde{\bf D}(\tau,\X){}^\top\!\!\, \widetilde{J}(\tau,\X)\nabla_\XX\psi(t,\X)$ is an admissible test function in the sense of Definition \ref{defn:abs:admissible-test-fn}. Hence, passing to the limit yields
\begin{align*}
-
\iint\limits_{(0,T)\times\R^d}
& u_0(t,\X)
\, {\rm d}\X\, {\rm d}t
-
\int\limits_{\R^d}
u^{in}(\X)
\, {\rm d}\X
\\
& +
\iiint\limits_{(0,T)\times\R^d\times\Delta(\mathcal{A})}
\widehat{\widetilde{J}}(s,\X)
\widehat{\widetilde{\bf D}}(s,\X)
\widehat{{}^\top\!\!\, \widetilde{J}}(s,\X)
\nabla_\XX u_0(t,\X)
\cdot
\nabla_\XX \psi(t,\X)
\, {\rm d}\beta(s)\, {\rm d}\X\, {\rm d}t
= 0.
\end{align*}

\begin{Rem}\label{rem:explicit:Lagrangian}
In the above computation, passing to the limit as $\eps\to0$ using $\Sigma$-$\Phi_\tau$ convergence amounts to arriving at a limit equation which is in Lagrangian coordinates
\begin{align*}
\frac{\partial u_0}{\partial t}
- 
\nabla_\XX
\cdot 
\Big(
\mathfrak{D}(\X)
\nabla_\XX
u_0
\Big)
= 0
\end{align*}
where the diffusion coefficient is given by
\begin{align*}
\mathfrak{D}(\X)
=
\int\limits_{\Delta(\mathcal{A})}
\widehat{\widetilde{J}}(s,\X)
\widehat{\widetilde{\bf D}}(s,\X)
\widehat{{}^\top\!\!\, \widetilde{J}}(s,\X)
\, {\rm d}\beta(s).
\end{align*}
\end{Rem}

\subsection{Fluid field with $\mathcal{O}(\eps)$ perturbation}\label{ssec:explicit-eps}

In a next transport model, we consider a smooth fluid field with a particular structure
\begin{align}\label{eq:explicit-eps-fluid-field}
{\bf b}\left(x, \frac{x}{\eps}\right)
= 
{\bf h}\left(\frac{x}{\eps}\right)
+ \eps {\bf h}^1\left(x, \frac{x}{\eps}\right).
\end{align}
The convection-diffusion equation that we consider is
\begin{align}\label{eq:explicit:eps-CD}
\frac{\partial u^\eps}{\partial t}
+
\frac{1}{\eps}
{\bf h}\left(\frac{x}{\eps}\right)
\cdot
\nabla u^\eps 
+
{\bf h}^1\left(x,\frac{x}{\eps}\right)
\cdot
\nabla u^\eps 
-
\nabla
\cdot
\Big(
{\bf D}\left(\frac{x}{\eps}\right)\nabla u^\eps 
\Big)
=
0
\, \, 
\mbox{ for }
(t,x)\in\, ]0,T[\times\R^d.
\end{align}
As only the field ${\bf h}\left(\frac{x}{\eps}\right)$ is of $\mathcal{O}(\eps^{-1})$ in \eqref{eq:explicit:eps-CD}, we need to consider the flow associated with the mean field 
\begin{align*}
{\bf h}^*:=\int\limits_{\T^d}{\bf h}(y)\, {\rm d}y,
\qquad
\mbox{ i.e. }
\Phi_\tau = x + {\bf h}^* \tau.
\end{align*}
This suggests the use of \emph{two-scale convergence with drift} \cite{maruvsic2005homogenization, allaire2008periodic}. The solution family $u^\eps$ satisfies the uniform a priori bounds:
\begin{align*}
\|u^\eps\|_{L^2((0,T)\times\R^d)} \le C;
\qquad
\|\nabla u^\eps\|_{L^2((0,T)\times\R^d)} \le C.
\end{align*}
The compactness results in \emph{two-scale convergence with drift} theory implies the existence of $u_0\in L^2((0,T);H^1(\R^d))$ and $u_1\in L^2((0,T)\times\R^d;H^1(\T^d))$ such that
\begin{align}\label{eq:explicit-2scale-hstar}
u^\eps \2scaleh u_0(t,x);
\qquad
\nabla u^\eps \2scaleh \nabla_x u_0(t,x) + \nabla_y u_1(t,x,y).
\end{align}
The idea is indeed to pass to the limit in the weak formulation with 
$$
\psi\left(t, x-\frac{{\bf h}^*t}{\eps}\right) + \eps \psi_1\left(t, x-\frac{{\bf h}^*t}{\eps}, \frac{x}{\eps}\right)
$$ 
as the test function which vanishes at time instant $t=T$.
\begin{align*}
& -
\iint\limits_{(0,T)\times\R^d}
u^\eps(t,x)
\frac{\partial \psi}{\partial t}
\left(
t, x - \frac{{\bf h}^*t}{\eps}
\right)
\, {\rm d}x\, {\rm d}t
+
\iint\limits_{(0,T)\times\R^d}
u^\eps(t,x)
{\bf h}^*
\cdot 
\nabla_x \psi_1
\left(
t, x - \frac{{\bf h}^*t}{\eps}, \frac{x}{\eps}
\right)
\, {\rm d}x\, {\rm d}t
\\
& +
\frac{1}{\eps}
\iint\limits_{(0,T)\times\R^d}
u^\eps(t,x)
\left(
{\bf h}^*
-
{\bf h}\left(\frac{x}{\eps}\right)
\right)
\cdot 
\nabla_x \psi
\left(
t, x - \frac{{\bf h}^*t}{\eps}
\right)
\, {\rm d}x\, {\rm d}t
\\
& +
\iint\limits_{(0,T)\times\R^d}
{\bf h}^1\left(x,\frac{x}{\eps}\right)
\cdot 
\nabla u^\eps(t,x)
\psi
\left(
t, x - \frac{{\bf h}^*t}{\eps}
\right)
\, {\rm d}x\, {\rm d}t
\\
& 
+
\iint\limits_{(0,T)\times\R^d}
{\bf D}\left(\frac{x}{\eps}\right)
\nabla u^\eps(t,x)
\cdot 
\left(
\nabla_x \psi
\left(
t, x - \frac{{\bf h}^*t}{\eps}
\right)
+
\nabla_y \psi_1
\left(
t, x - \frac{{\bf h}^*t}{\eps}, \frac{x}{\eps}
\right)
\right)
\, {\rm d}x\, {\rm d}t
+
\mathcal{O}(\eps)
= 0.
\end{align*}
We can pass to the limit in almost all the terms in above expression except for the fourth term on the left hand side. This is essentially because the product ${\bf h}^1\left(x,y\right)
\psi
\left(
t, x
\right)$
does not form an admissible test function in the sense of Definition \ref{defn:abs:admissible-test-fn}. However, if we consider the flow representation of the fluid field ${\bf h}^1(x,y)$ and assume that $\widetilde{{\bf h}^1}(\cdot,x,y)\in \mathcal{A}$ for certain ergodic algebra w.m.v. $\mathcal{A}$, then the product $\widetilde{{\bf h}^1}(\tau,x,y)\psi(t,x)$ forms an admissible test function in the sense of Definition \ref{defn:abs:admissible-test-fn}. Thus, using the notion of weak $\Sigma$-$\Phi_\tau$ convergence, we can prove the following result. The proof of which is a simple adaptation of the calculations already present in Section \ref{sec:homog_result}. Hence is left to the reader.

\begin{Thm}\label{thm:explicit-eps}
Suppose the flow representation of the fluid field ${\bf h}^1(x,y)$ belongs to an ergodic algebra w.m.v. $\mathcal{A}$. The two-scale with drift ${\bf h}^*$ limits for the solution family $u^\eps$ obtained in \eqref{eq:explicit-2scale-hstar} satisfy the homogenized equation
\begin{align}\label{eq:explicit:hom-eq}
\frac{\partial u_0}{\partial t}
+ \mathfrak{h}(x)\cdot \nabla u_0
- \nabla \cdot \Big(\mathfrak{D}\nabla u_0 \Big)
= 0
\qquad
\mbox{ for }(t,x)\in\, ]0,T[\times\R^d,
\end{align}
where 
the convective field in the homogenized equation is given by
\begin{align}\label{eq:explicit:hom-convect-field}
\mathfrak{h}(x)
=
\iint\limits_{\Delta(\mathcal{A})\times\T^d}
\widehat{\widetilde{{\bf h}^1}}(s,x,y)
\, {\rm d}\beta(s)\, {\rm d}y
\end{align}
and the effective diffusion coefficient in the homogenized equation is given by
\begin{align*}
\mathfrak{D}_{ij}
= 
\int\limits_{\T^d}
{\bf D}(y)
\Big(
\nabla_y \omega_j(y)
+ {\bf e}_j
\Big)
\cdot
\Big(
\nabla_y \omega_i(y)
+ {\bf e}_i
\Big)
\, {\rm d}y
\end{align*}
for $i,j\in \{1,\cdots,d\}$ where the $\omega_i$ solve the cell problem
\begin{align*}
{\bf h}(y)\cdot \left(\nabla_y \omega_i + {\bf e}_i\right)
- \nabla_y \cdot \left( {\bf D}(y) \left(\nabla_y \omega_i + {\bf e}_i\right) \right)
= {\bf h}^*\cdot {\bf e}_i\qquad \mbox{ in }\T^d.
\end{align*}
\end{Thm}
Remark that, due to particular choice of the fluid field \eqref{eq:explicit-eps-fluid-field}, the field ${\bf h}^1(x,\frac{x}{\eps})$ only contributes to the homogenized equation \eqref{eq:explicit:hom-eq} via the convective field \eqref{eq:explicit:hom-convect-field} and not the effective diffusion coefficient. Only the fluid field of $\mathcal{O}(\eps^{-1})$ contribute to the dispersive effects in the effective diffusion coefficient.

Remark also that even in the constant drift scenario, the previously known \emph{two-scale convergence with drift} developed in \cite{maruvsic2005homogenization} has a handicap in dealing with coefficients that depend on the macroscopic variable. Hence, the notion of weak convergence developed in this work generalizes the known multiple scale techniques (in the spirit of two-scale convergence of Nguetseng and Allaire) in homogenization theory to a great extent.

\section{Conclusion}\label{sec:conclusions}

The structural assumption of periodicity (in the $y$ variable) on the fluid field ${\bf b}(x,y)$ made in the previous sections is for the sake of simplicity. We can indeed develop a theory of $\Sigma$-convergence along flows (similar to the theory developed in Section \ref{sec:abs}) under the assumption that the oscillations in space belong to certain \emph{ergodic algebra with mean value}. To be precise, suppose ${\bf b}(x,y)$ a smooth fluid field which belongs to an ergodic algebra w.m.v. (say $\mathcal{A}_1$) in the $y$ variable. By the definition of algebra w.m.v. (precisely, property (iii) in Definition \ref{def:abs:algebra-w.m.v.}), ${\bf b}(x,\cdot)\in\mathcal{A}_1$ possesses a mean value, i.e.
\begin{align*}
{\bf b}\left(x,\frac{x}{\eps}\right) \rightharpoonup M{\bf b}(x)\quad
\mbox{ in }L^\infty(\R^d)\mbox{-weak*}\mbox{ as }\eps\to0.
\end{align*}
In this scenario, we take the mean field $\bar{\bf b}(x) = M{\bf b}(x)$ and consider the flow $\Phi_\tau$ associated with this mean field. To extend the notion of $\Sigma$-convergence along flows (Definition \ref{def:abs:2scale-flow}), we need to essentially characterize the limit
\begin{align*}
\lim_{\eps\to0}
\iint\limits_{(0,T)\times\R^d}
u^\eps(t,x)
\psi\left( t, \frac{t}{\eps}, \Phi_{-t/\eps}(x), \frac{x}{\eps}\right)
\, {\rm d}x\, {\rm d}t
\end{align*}
where the test function $\psi(t,\tau,x,y)$ belongs to an ergodic algebra w.m.v. (say $\mathcal{A}_2$) as a function of the fast time variable $\tau$ and belongs to an ergodic algebra w.m.v. $\mathcal{A}_1$ as a function of the $y$ variable. To prove compactness result, in the spirit of Theorem \ref{Thm:abs:compactness}, the approach is to consider the differentiation theory developed in the context of \emph{algebras w.m.v.} developed in \cite{nguetseng2003homogenization, nguetseng2004homogenization, casado2002two, sango2011generalized}. We also need to approach it using the reiterated homogenization techniques as in \cite{nguetseng2010reiterated}. The effective diffusion matrix obtained under the periodicity assumption (see \eqref{eq:hom:effective-diffusion}-\eqref{eq:hom:effective-diffusion-mathfrak-B}) is given in terms of the cell solutions obtained by solving elliptic problems on a torus. In this general setting, however, the expressions for effective diffusion shall involve solutions to some variational problems solved on the spectrum of the algebra w.m.v., i.e. $\Delta(\mathcal{A}_1)$ (cf. the works of Nguetseng \cite{nguetseng2003homogenization, nguetseng2004homogenization}). This potential theory of $\Sigma$-convergence along flows in a more general setting is quite intricate and is left for future investigations.

As is evident from Subsection \ref{ssec:Exm:counter}, even in the constant drift case, one can only homogenize the convection-diffusion problems in strong convection regime provided the flow representation of the diffusion matrix belongs to an algebra w.m.v., i.e. satisfies Assumption \ref{ass:hom:D-x-y-algebra}.

All along this article, we have considered time-independent coefficients. This resulted in the study of autonomous ordinary differential systems (see \eqref{eq:heu:ode-mean-field}). Considering flows associated with non-autonomous systems would be interesting. But, the authors believe that the analysis would be very complicated and it remains to be checked if we can get compactness results (in the spirit of Theorem \ref{Thm:abs:compactness}) for non-autonomous flows.

The assumption that the Jacobian matrices are bounded functions of the fast time variable is quite non-generic (see Section \ref{sec:Exm}). To lift this assumption would require an enormous amount of work in the theory of Banach algebras. The main difficulty is the appearance of new time scales (as is evident from the shear flow case considered in Example \ref{Exm:CExm:blow-up}). This problem largely remains to be solved. A partial result in this direction shall be given by the authors in a forthcoming publication \cite{holdingetal2016shearflow}.

Finally, the assumption of incompressibility on the fluid field has ensured that the associated flows are volume preserving (see (iv) in Assumption \ref{ass:abs:flow}). This property of the flows has played an intricate role in our analysis, notably the proof of Lemma \ref{lem:abs:test-function}. It is worth mentioning \cite{blanchet2009stochastic} where they have treated the homogenization of convection-diffusion problem in strong convection regime where the fluid field is given by an harmonic potential. In the context of purely periodic fluid fields, there are works that consider compressible flows and perform the homogenization of convection-diffusion problems in strong convection regime (see \cite{donato2005averaging, allaire2014homogenization}). The approach is to employ a factorization principle to factor out oscillations from the solution via principal eigenfunctions of an associated spectral problem and to cancel any exponential decay in time of the solution using the principal eigenvalue of the same spectral problem. This approach has not been attempted in the literature for locally periodic coefficients.

\section{Appendix}\label{sec:appendix}

In this section, we give the proof of Lemma \ref{lem:heu:basic-facts} on some basic facts on the flows.

\begin{proof}[Proof of Lemma \ref{lem:heu:basic-facts}]
We prove each claim in turn.
\begin{enumerate}[(i)]
\item 
Let $\varphi(\X)\in C^\infty_c(\R^d;\R)$ be an arbitrary test function and let the index $i$ be arbitrary. By the chain rule,
\begin{align*}
\frac{\partial}{\partial x_i}
\Big(
\varphi\left(\Phi_{-\tau}(x)\right)
\Big)
=
\sum_{j=1}^d
\frac{\partial \varphi}{\partial \X_j}\left(\Phi_{-\tau}(x)\right)
\frac{\partial \Phi^j_{-\tau}}{\partial x_i}(x).
\end{align*}
Integrating over $\R^d$ yields:
\begin{align*}
0 =
\int\limits_{\R^d}
\frac{\partial}{\partial x_i}
\left(
\varphi\left(\Phi_{-\tau}(x)\right)
\right)
\, {\rm d}x
= 
\int\limits_{\R^d}
\sum_{j=1}^d
\frac{\partial \varphi}{\partial \X_j}\left(\Phi_{-\tau}(x)\right)
\frac{\partial \Phi^j_{-\tau}}{\partial x_i}(x)
\, {\rm d}x.
\end{align*}
Making the change of variables: $\X=\Phi_{-\tau}(x)$, the above expression can be successively written as
\begin{align*}
0 =
\int\limits_{\R^d}
\sum_{j=1}^d
\frac{\partial \varphi}{\partial \X_j}(\X)
\frac{\partial \Phi^j_{-\tau}}{\partial x_i}\left( \Phi_\tau(\X)\right)
\, {\rm d}\X
= \int\limits_{\R^d}
\nabla_\XX \varphi(\X)
\cdot
\left(
\widetilde{J}_{ji}(\tau,\X)
\right)_{j=1}^d
\, {\rm d}\X,
\end{align*}
i.e. each column of $\widetilde{J}$ is divergence free in the sense of distributions, proving the claim.
\item We compute
\begin{equation*}
\nabla_\XX\cdot
\Big(
\widetilde{J}(\tau,\X)\widetilde{f}(\tau,\X)
\Big)
=
\widetilde{f}(\tau,\X)
\cdot
\left(
\nabla_\XX \cdot {}^\top\!\!\, \widetilde{J}(\tau,\X)
\right)
+
\widetilde{J}(\tau,\X)
:
\nabla_\XX\widetilde{f}(\tau,\X),
\end{equation*}
where $:$ is the Frobenius inner product. The first term on the right hand side vanishes thanks to (i). For the second term, we use the flow representation to obtain
\begin{equation*}
\nabla_\XX\cdot
\Big(
\widetilde{J}(\tau,\X)\widetilde{f}(\tau,\X)
\Big)
=
\widetilde{J}(\tau,\X) J(-\tau,\X)
:
\nabla_x f(\Phi_\tau(\X),y).
\end{equation*}
Thanks to the autonomy of the flow, the left side of the Frobenius product is the identity matrix. Therefore the above display vanishes as $f$ is divergence free.
\item Performing an integration by parts, we have:
\begin{align*}
\int\limits_{\R^d}
\phi(\X)
& \left(
{}^\top\!\!\, \widetilde{J}(\tau,\X) \nabla_\XX \varphi(\X)
\right)
\, {\rm d}\X
\\
& = 
- \int\limits_{\R^d}
\phi(\X)
\varphi(\X)
\left(
\nabla_\XX \cdot {}^\top\!\!\, \widetilde{J}(\tau,\X)
\right)
\, {\rm d}\X
-\int\limits_{\R^d}
\varphi(\X)
\left(
{}^\top\!\!\, \widetilde{J}(\tau,\X) \nabla_\XX \phi(\X)
\right)
\, {\rm d}\X.
\end{align*}
The first term on the right hand side of the previous expression vanishes, thanks to (i). Hence, we have proved the result.
\item Consider the time derivatives for the $i$-th component:
\begin{align}\label{eq:heu:dt-bi}
\frac{\rm d}{{\rm d}\tau}
\bar{\bf b}_i\left(\Phi_{-\tau}(x)\right)
= - \bar{\bf b}\left(\Phi_{-\tau}(x)\right)
\cdot
\nabla_{\scriptscriptstyle X} \bar{\bf b}_i\left(\Phi_{-\tau}(x)\right)
\end{align}
and
\begin{align}\label{eq:heu:dt-Jij-bj}
\frac{\rm d}{{\rm d}\tau}
\left[
\sum_{j=1}^d
J_{ij}(\tau,x) \bar{\bf b}_j(x)
\right]
= \frac{\rm d}{{\rm d}\tau}
\left[
\sum_{j=1}^d
\frac{\partial \Phi^i_{-\tau}}{\partial x_j}(x)
\bar{\bf b}_j(x)
\right]
=
-
\sum_{j=1}^d
\frac{\partial}{\partial x_j}
\Big(
\bar{{\bf b}}_i\left(\Phi_{-\tau}(x)\right)
\Big)
\bar{\bf b}_j(x).
\end{align}
The relation in \eqref{eq:heu:dt-Jij-bj} can be continued as
\begin{equation}\label{eq:heu:dt-Jij-bj-cont.}
\begin{aligned}
\frac{\rm d}{{\rm d}\tau}
\left[
\sum_{j=1}^d
J_{ij}(\tau,x) \bar{\bf b}_j(x)
\right]
&=
-
\sum_{j,k=1}^d
\frac{\partial \bar{{\bf b}}_i}{\partial {\scriptscriptstyle X}_k}\left( \Phi_{-\tau}(x)\right)
\frac{\partial \Phi^k_{-\tau}}{\partial x_j}(x)
\bar{\bf b}_j(x)\\
&=
-
\nabla_{{\scriptscriptstyle X}} \bar{\bf b}_i\left(\Phi_{-\tau}(x)\right) 
\cdot 
\Big(
J(\tau,x)\bar{\bf b}(x)
\Big).
\end{aligned}
\end{equation} 
Fix $x\in\R^d$ and define
\begin{align}
g_i(\tau) := 
\bar{\bf b}_i\left(\Phi_{-\tau}(x)\right) 
- 
\left[
\sum_{j=1}^d
J_{ij}(\tau,x) \bar{\bf b}_j(x)
\right].
\end{align}
Then, from \eqref{eq:heu:dt-bi} and \eqref{eq:heu:dt-Jij-bj-cont.}, we have:
\begin{align*}
\frac{\rm d}{{\rm d}t} g_i(\tau) = - \nabla_{{\scriptscriptstyle X}} \bar{\bf b}_i\left(\Phi_{-\tau}(x)\right) \cdot g(\tau).
\end{align*}
As $g(0)=0$, a Gr\"onwall type argument yields $g(\tau)=0$. Hence the result. \qedhere
\end{enumerate}
\end{proof}

\bibliography{biblio.bib}

\end{document}